\documentclass[a4paper, reqno, 11pt]{amsart}

\usepackage[margin=1in]{geometry}
\usepackage{amsfonts,amssymb}
\usepackage{amsthm}
\usepackage{graphicx}
\usepackage{float}
\usepackage{lipsum}
\usepackage{comment}
\usepackage{systeme}
\usepackage{lineno}
\usepackage{mathtools}
\usepackage{amsmath}
\usepackage{esint}
\usepackage{mathrsfs}
\usepackage{verbatim}

\usepackage{colortbl}
\usepackage[usenames,dvipsnames]{xcolor}
\usepackage{amssymb}
\usepackage[active]{srcltx}
\usepackage[colorlinks, pdfpagelabels, pdfstartview = FitH,bookmarksopen = true, bookmarksnumbered = true,linkcolor = blue,plainpages = false,hypertexnames = false, citecolor = red]{hyperref}
\usepackage[italian, textsize=tiny, color=BlueGreen, backgroundcolor=BlueGreen, linecolor=BlueGreen]{todonotes}

\usepackage{cite}

\allowdisplaybreaks

\usepackage{booktabs,amsmath}
\def\Xint#1{\mathchoice
    {\XXint\displaystyle\textstyle{#1}}%
    {\XXint\textstyle\scriptstyle{#1}}%
    {\XXint\scriptstyle\scriptscriptstyle{#1}}%
    {\XXint\scriptscriptstyle\scriptscriptstyle{#1}}%
      \!\int}
\def\XXint#1#2#3{{\setbox0=\hbox{$#1{#2#3}{\int}$}
    \vcenter{\hbox{$#2#3$}}\kern-.5\wd0}}
\def\dashint{\Xint-}

\def\YYint#1#2#3{{\setbox0=\hbox{$#1{#2#3}{\int}$}
    \lower1ex\hbox{$#2#3$}\kern-.46\wd0}}
\def\YYYint#1#2#3{{\setbox0=\hbox{$#1{#2#3}{\int}$}
    \lower0.35ex\hbox{$#2#3$}\kern-.48\wd0}}

\def\ZZint#1#2#3{{\setbox0=\hbox{$#1{#2#3}{\int}$}
    \raise1.15ex\hbox{$#2#3$}\kern-.57\wd0}}
\def\ZZZint#1#2#3{{\setbox0=\hbox{$#1{#2#3}{\int}$}
    \raise0.85ex\hbox{$#2#3$}\kern-.53\wd0}}

\def\Xiint#1{\mathchoice
    {\XXiint\displaystyle\textstyle{#1}}%
    {\XXiint\textstyle\scriptstyle{#1}}%
    {\XXiint\scriptstyle\scriptscriptstyle{#1}}%
    {\XXiint\scriptscriptstyle\scriptscriptstyle{#1}}%
      \!\iint}
\def\XXiint#1#2#3{{\setbox0=\hbox{$#1{#2#3}{\iint}$}
    \vcenter{\hbox{$#2#3$}}\kern-.5\wd0}}
\def\dashiint{\Xiint-}

\DeclareMathOperator*{\esup}{ess\,sup}
\DeclareMathOperator*{\dist}{dist}
\DeclareMathOperator*{\supp}{supp}
\DeclareMathOperator*{\ITail}{Tail_{\infty}}
\DeclareMathOperator*{\ITails}{Tail^{2}_{\infty}}

\DeclareMathOperator*{\loc}{loc}
\DeclareMathOperator*{\osc}{osc}
\DeclareMathOperator*{\pv}{p.v.}

\makeatletter
\@namedef{subjclassname@2020}{\textup{2020} Mathematics Subject Classification}
\makeatother

\newcommand{\rom}[1]{\uppercase\expandafter{\romannumeral #1\relax}}

\newcommand{\Chi}{\mbox{\Large$\chi$}}

\newcommand{\Title}{Higher H\"older regularity for nonlocal parabolic equations with irregular kernels}
\newcommand{\Author}{Sun-Sig Byun, Hyojin Kim \and Kyeongbae Kim}
\theoremstyle{plain}
\newtheorem{thm}{Theorem}[section]

\newtheorem{lem}[thm]{Lemma}
\newtheorem{cor}[thm]{Corollary}

\theoremstyle{definition}

\newtheorem*{defn*}{Definition}

\theoremstyle{remark}
\newtheorem{rmk}{Remark}

\numberwithin{equation}{section}


\newcommand{\dx}{\,dx}
\newcommand{\dt}{\,dt}
\newcommand{\dy}{\,dy}

\newcommand{\ds}{\,ds}
\newcommand{\dsigma}{\,d\sigma}
\newcommand{\ddiv}{\mathrm{div\,}}

\title{\Title}
\author{\Author}

\address{Sun-Sig Byun: Department of Mathematical Sciences and Research Institute of Mathematics, Seoul National University, Seoul 08826, Korea}
\email{byun@snu.ac.kr}

\address{Hyojin Kim: Department of Mathematical Sciences, Seoul National University, Seoul 08826, Korea}
\email{hyojin@snu.ac.kr}

\address{Kyeongbae Kim: Department of Mathematical Sciences, Seoul National University, Seoul 08826, Korea}
\email{kkba6611@snu.ac.kr}

\thanks{S. Byun was supported by NRF-2021R1A4A1027378. H. Kim was supported by NRF-2020R1C1C1A01009760. K. Kim was supported by NRF-2022R1A2C1009312.}

\begin{document}
\maketitle

\begin{abstract}
We study a nonlocal parabolic equation with an irregular kernel coefficient to establish higher H\"older regularity under an appropriate higher integrablilty on the nonhomogeneous terms and a minimal regularity assumption on the kernel coefficient.
\end{abstract}

\section{Introduction}
\subsection{Overview and Main results}
\noindent  
In this paper, we study the following parabolic nonlocal equations
\begin{equation}
\label{eq1}
\partial_{t}u+\mathcal{L}_{A}^{\Phi}u=f \text{ in }\Omega\times(0,T),
\end{equation}
where 
\begin{equation*}
    \mathcal{L}^{\Phi}_{A}u(x,t)\coloneqq\pv\int_{\mathbb{R}^{n}}\Phi(u(x,t)-u(y,t))\frac{A(x,y,t)}{|x-y|^{n+2s}}\dy,\quad (x,t)\in\Omega_{T},
\end{equation*}
for some $s\in(0,1)$. Here, $\Omega\subset\mathbb{R}^{n}$ $(n\geq2)$ is a bounded domain, 
$A:\mathbb{R}^{2n}\times\mathbb{R}\rightarrow\mathbb{R}$ is a kernel coefficient satisfying 
\begin{equation}
\label{coefficeint condition}
\begin{cases}
    A(x,y,t)=A(y,x,t) \ \ a.e.\ (x,y,t)\in\mathbb{R}^{2n}\times\mathbb{R}\\
    \lambda^{-1}\leq A(x,y,t)\leq\lambda \text{ for some constant }\lambda\geq1,
\end{cases}
\end{equation}
and
$\Phi:\mathbb{R}\rightarrow\mathbb{R}$ is a measurable function with $\Phi(0)=0$ such that
\begin{equation}
\label{Phi condition}
\begin{cases}
(\Phi(\xi)-\Phi(\xi'))(\xi-\xi')\geq\lambda^{-1}|\xi-\xi'|^{2}\\
|\Phi(\xi)-\Phi(\xi')|\leq\lambda|\xi-\xi'| \text{ for any }\xi,\xi'\in\mathbb{R},
\end{cases}
\end{equation}
where $\lambda$ is the same number in \eqref{coefficeint condition}. In particular, if $A=1$ and $\Phi(t)=t$, then \eqref{eq1} reduces to a nonhomogeneous fractional heat equation 
$u_{t}+(-\triangle)^{s}u = f$. 
We denote by $\mathcal{L}_{0}(\lambda)\equiv\mathcal{L}_{0}(\lambda;\mathbb{R}^{2n}\times\mathbb{R})$ to mean a class of measurable kernel coefficients satisfying \eqref{coefficeint condition}.
Throughout this paper, we always assume that $A$ satisfies \eqref{coefficeint condition} and $\Phi$ satisfies \eqref{Phi condition}.

Recently there has been a huge amount of research on nonlocal equations. For the elliptic case, Kassmann \cite{KH1} proved H\"older continuity of a weak solution to a fractional Laplace equation with zero boundary condition. Di Castro, Kuusi and Palatucci \cite{CKPf} proved H\"older continuity for a fractional p-Laplace equation. Moreover, they first introduced a tail space to consider nonzero boundary data. We refer to \cite{ KH1,BKO,BOS,Cr,Shf} for H\"older regularity, to \cite{BL,Nv,MSY,Ni,KMS1,FMSY,SM,ABES,Sn,Nexist} for higher Sobolev regularity, and to \cite{KMS2} for a potential estimate. In the parabolic case, Caffarelli, Chan and Vasseur \cite{CCV} proved H\"older regularity for a fractional parabolic equation with linear growth. We refer to \cite{Sl,S,DZZ,LPPS,KSp,FRp,V1,N1,APTl} for various regularity results. 

Now we turn to higher H\"older regularity. Brasco, Lindgren and Schikorra \cite{BLS2} proved higher H\"older regularity for a fractional p-Laplace equation with the nonhomogeneous term in the super-quadratic case. Brasco, Lindgren and Str\"{o}mqvist \cite{BLS} extended the argument used to a parabolic fractional p-Laplace equation. It is notable that Nowak \cite{NHH} considered a fractional equation with linear growth and an irregular kernel coefficient to obtain similar results in \cite{BLS2}. We would like to mention the recent and interesting work \cite{T} in which a similar result as in the present paper is obtained when $A$ is constant and $\Phi(\xi)=|\xi|^{p-2}\xi$ for $p\geq2$. For a further higher H\"older regularity, we refer to \cite{CS,GKS,Fr,CSP} and references therein.

The aim of this paper is two-fold. One is to establish a local boundedness of a weak solution to \eqref{eq1} when the nonhomogeneous term $f$ satisfies \eqref{forcing term}. The other is to obtain the higher H\"older regularity of a weak solution to \eqref{eq1} under a suitable regularity assumption on the kernel coefficient $A$.  
As usual, a solution to \eqref{eq1} is defined in the weak sense as below. See the next section for a precise description of the related  function spaces.

\begin{defn*}(Local weak solution)
Let $f\in L^{q,r}_{\loc}(\Omega_{T})=L_{\loc}^{r}\left(0,T;L^{q}_{\loc}(\Omega)\right)$, where $q$ and $r$ are any positive numbers such that {$q,r \ge 1$} and $\frac{n}{2qs}+\frac{1}{r}\leq 1+\frac{n}{4s}$. We say that 
\[u\in L^{2}_{\loc}\left(0,T;W^{s,2}_{\loc}(\Omega)\right)\cap L^{\infty}_{\loc}\left(0,T;L^{1}_{2s}(\mathbb{R}^{n})\right)\cap C_{\loc}\left(0,T;L_{\loc}^{2}(\Omega)\right)\]is a local weak subsolution(supersolution) to \eqref{eq1} if 
\begin{align*}
    &\int_{t_{1}}^{t_{2}}\int_{\Omega}-u\partial_{t}\phi\dx\dt+\int_{t_{1}}^{t_{2}}\int_{\mathbb{R}^{n}}\int_{\mathbb{R}^{n}}\Phi(u(x,t)-u(y,t))(\phi(x,t)-\phi(y,t))\frac{A(x,y,t)}{|x-y|^{n+2s}}\dx\dy\dt\\
    &\quad\leq(\geq)\int_{t_{1}}^{t_{2}}\int_{\Omega}f\phi\dx\dt-\int_{\Omega}u\phi\dx\Bigg\rvert_{t=t_{1}}^{t=t_{2}}
\end{align*}
for all nonnegative functions $\phi\in L^{2}(I;W^{s,2}(\Omega))\cap W^{1,2}\left(I;L^{2}(\Omega)\right)$ with spatial support compactly embedded in $\Omega$ and $I\coloneqq[t_{1},t_{2}]\Subset(0,T)$. In particular, we say that $u$ is a local weak solution to \eqref{eq1} if $u$ is both a local subsolution and a local supersolution to \eqref{eq1}.
\end{defn*}
\noindent
Note that $L^{2}([t_{1},t_{2}];W^{s,2}(\Omega))\cap W^{1,2}\left([t_{1},t_{2}];L^{2}(\Omega)\right)$ is embedded in $L^{\hat{r}}([t_{1},t_{2}];L^{\hat{q}}(\Omega))$ for some positive numbers $\hat{q}$ and $\hat{r}$ such that {$\hat{q},\hat{r} \ge 1$} and $\frac{n}{2\hat{q}s}+\frac{1}{\hat{r}}=\frac{n}{4s}$, see Lemma \ref{embedSB}, to find 
\[\int_{t_{1}}^{t_{2}}\int_{\Omega}|f\phi|<\infty.\] Existence and uniqueness of a weak solution \eqref{eq1} with appropriate initial and boundary conditions will be discussed in Appendix \ref{Appendix}.

Now we introduce our main results. The first one is the local boundedness for a local weak subsolution to $\eqref{eq1}$ with the following assumption of $f$,
\begin{equation}
\label{forcing term}
f\in L^{q,r}_{\loc}(\Omega_{T})\quad \text{with}\quad\frac{n}{2qs}+\frac{1}{r}<1.
\end{equation}

\begin{thm}(Local boundedness) 
\label{LBlem}
Suppose that $u$ is a local weak subsolution to \eqref{eq1} with  \eqref{forcing term}.
 Then there is a constant $c \equiv c(n,s,q,r,\lambda)$ such that
\begin{align*}
    \sup_{z\in Q_{\rho_{0}/2}(z_{0})}u(z)&\leq c\Bigg[\left(\dashiint_{Q_{\rho_{0}}(z_{0})}u^{2}(x,t)\dx\dt\right)^{\frac{1}{2}}+\ITail\left(u,z_{0},\rho_{0}/2,\rho_{0}^{2s}\right)\\
    &\qquad+\rho_{0}^{2s-\left(\frac{n}{q}+\frac{2s}{r}\right)}\|f\|_{L^{q,r}(Q_{\rho_{0}}(z_{0}))}\Bigg],
\end{align*}
whenever $Q_{\rho_{0}}(z_{0})\equiv B_{\rho_{0}}(x_{0})\times(t_{0}-\rho_{0}^{2s},t_{0}]\Subset\Omega_{T}$.
\end{thm}
\begin{rmk}
We first notice that \eqref{forcing term} is an appropriate condition to get the local boundedness (see \cite{LPPS}).
{In addition, we observe that in the limiting case of $s\to 1$, \eqref{forcing term} is a suitable condition to get the local boundedness for the local case when $s=1$ (see \cite[Chapter 3]{LSU}).}
\end{rmk}
\noindent
The second one is the higher H\"older regularity. Here, we introduce a new kernel class for the desired result. We say that $\Tilde{A}\in\mathcal{L}_{1}(\lambda;\Omega\times\Omega\times(0,T))$ if $\Tilde{A}\in\mathcal{L}_{0}(\lambda;\Omega\times\Omega\times(0,T))$ and there is a measurable function $a:\mathbb{R}^{n}\times(0,T)\to \mathbb{R}$ such that
\begin{equation*}
\Tilde{A}(x,y,t)=a(x-y,t),\quad (x,y,t)\in\Omega\times\Omega\times(0,T).
\end{equation*}\begin{thm} 
\label{Holder}
Let $u$ be a local weak solution to \eqref{eq1} with \eqref{forcing term} and let $\alpha$ be a positive  number such that
\begin{equation}\label{alphacondthm}
\alpha<\min\left\{2s-\left(\frac{n}{q}+\frac{2s}{r}\right),1\right\}.
\end{equation}Then there is a constant $\delta=\delta(n,s,q,r,\lambda,\alpha)>0$ such that if for any $\Tilde{z}\in\Omega_{T}$, there exist sufficiently small $\rho_{\Tilde{z}}>0$ and  
 $\Tilde{A}_{\Tilde{z}}\in\mathcal{L}_{1}\left(\lambda;B_{\rho_{\Tilde{z}}}(\Tilde{x})\times B_{\rho_{\Tilde{z}}}(\Tilde{x})\times [\Tilde{t}-\rho_{\Tilde{z}}^{2s},\Tilde{t}]\right)$ with
\begin{equation}
\label{assumption closedness of A}
    \|\Tilde{A}_{\Tilde{z}}-A\|_{L^{\infty}\left(B_{\rho_{\Tilde{z}}}(\Tilde{x})\times B_{\rho_{\Tilde{z}}}(\Tilde{x})\times [\Tilde{t}-\rho_{\Tilde{z}}^{2s},\Tilde{t}]\right)}\leq \delta,
\end{equation}
then we have $u\in C^{\alpha,\frac{\alpha}{2s}}_{\loc}(\Omega_{T})$. In particular, for any $Q_{\rho_{0}}({z_{0}})\Subset\Omega_{T}$, we have 
\begin{align*}
    [u]_{C^{\alpha,\frac{\alpha}{2s}}(Q_{\rho_{0}/2}(z_{0}))}&\leq c\Bigg(\rho_{0}^{-\frac{n+2s}{2}}\|u\|_{L^{2}(Q_{\rho_{0}}(z_{0}))}+\ITail\left(u,z_{0},\rho_{0}/2,\rho_{0}^{2s}\right)\\
    &\qquad+\rho_{0}^{2s-\left(\frac{n}{q}+\frac{2s}{r}\right)}\|f\|_{L^{q,r}(Q_{\rho_{0}}(z_{0}))}\Bigg),
\end{align*}
for some constant $c \equiv c\left(n,s,q,r,\lambda,\alpha,\{\rho_{z}\}_{z\in Q_{\rho_{0}/2}(z_{0})}\right)$.
\end{thm}
{
\begin{rmk} The condition \eqref{alphacondthm} is natural to obtain H\"older continuity. When $s=1$, \eqref{alphacondthm} is exactly the same as the condition to get the H\"older continuity of weak solutions to local parabolic problems. See \cite[Chpater 2]{DE} for more details. In addition, if $f$ is autonomous, then \eqref{alphacondthm} is sharp. In this case, \eqref{alphacondthm} becomes $f\in L^{q}_{\mathrm{loc}}(\Omega)$ with $q>\frac{n}{2s}$ and a particular class of solutions is given by stationary solutions to the corresponding elliptic problem of \eqref{eq1}. From \cite[Corollary 1.2]{KMS2}, we deduce that if $q\leq\frac{n}{2s}$, then it does not satisfy the continuity criterion. Hence, it verifies that \eqref{alphacondthm} is an optimal condition to obtain the H\"older regularity of a weak solution to \eqref{eq1} when the right-hand side is given by an autonomous function.

\end{rmk}}

{\begin{rmk}
\label{rmkkernel}
We give some comments of the assumption \eqref{assumption closedness of A} on the kernel coefficient $A$.
\begin{enumerate}
\item We clearly point out that any continuous kernel coefficient $A$ satisfies \eqref{assumption closedness of A}. To be specific, we consider
\begin{equation}
\label{kernel11}
A_{1}(x,y,t)=K_{1,1}(x,y,t)+\Chi_{|x-y|\geq \epsilon}K_{1,2}(x,y,t)
\end{equation}
 for $\epsilon>0$, where $K_{1,1} \in \mathcal{L}_{0}(\lambda/2)$ is continuous, $K_{1,2} \in \mathcal{L}_{0}(\lambda/2)$ is merely measurable { for some $\lambda \ge 2$.}
Let us fix $\delta>0$. Then for any $\tilde{z} = (\tilde{x},\tilde{t}) \in \Omega_{T}$, there is a sufficiently small $\rho_{\tilde{z}}\in (0,\epsilon)$ such that 
\begin{equation*}
    \|A_{1}(x,y,t)-A_{1}(\tilde{x},\tilde{x},\tilde{t})\|_{L^{\infty}\left(B_{\rho_{\Tilde{z}}}(\Tilde{x})\times B_{\rho_{\Tilde{z}}}(\Tilde{x})\times [\Tilde{t}-\rho_{\Tilde{z}}^{2s},\Tilde{t}]\right)}\leq \delta.
\end{equation*}
Taking
\begin{equation*}
    \tilde{A}_{1,\tilde{z}}(x,y,t)=A(\tilde{x},\tilde{x},\tilde{t})\quad\text{for }(x,y,t)\in B_{\rho_{\Tilde{z}}}(\Tilde{x})\times B_{\rho_{\Tilde{z}}}(\Tilde{x})\times [\Tilde{t}-\rho_{\Tilde{z}}^{2s},\Tilde{t}],
\end{equation*}
we see that $\tilde{A}_{1,\tilde{z}}\in \mathcal{L}_{1}\left(\lambda;B_{\rho_{\Tilde{z}}}(\Tilde{x})\times B_{\rho_{\Tilde{z}}}(\Tilde{x})\times [\Tilde{t}-\rho_{\Tilde{z}}^{2s},\Tilde{t}]\right)$ and $A$ satisfies \eqref{assumption closedness of A}. Furthermore, we observe that any continuous function satisfies the assumption by taking $K_{1,2}=0$ in \eqref{kernel11}. Generally, if the kernel coefficient $A_{1}$ is continuous only near the diagonal, then the assumption holds. 
Moreover, if $A_{1}$ is H\"older continuous, then we can remove the dependence on $\{\rho_{z}\}_{z\in Q_{\rho_{0}/2}(z_{0})}$ of a universal constant $c$ determined in Theorem \ref{Holder}. See Corollary \ref{main theorem}. 

\medskip

\item We note that the kernel coefficient $A$ need not be continuous near the diagonal.
Let us consider a function $A_{2}(x,y,t)=K_{2,1}(x,y,t)K_{2,2}(x,y,t)$, where $K_{2,1}(x,y,t) \in \mathcal{L}_{0}(\sqrt{\lambda})$ is continuous near the diagonal and {$K_{2,2}(x,y,t)\in \mathcal{L}_{1}(\sqrt{\lambda};\Omega\times\Omega\times(0,T))\cap \mathcal{L}_{0}(\sqrt{\lambda})$}. Let us fix $\delta>0$. Then for any $\tilde{z} = (\tilde{x},\tilde{t})\in \Omega_{T}$, there is a sufficiently small $\rho_{\tilde{z}}\in (0,\epsilon)$ such that 
\begin{equation*}
    \|A_{2}(x,y,t)-A_{2}(\tilde{x},\tilde{x},\tilde{t})\|_{L^{\infty}\left(B_{\rho_{\Tilde{z}}}(\Tilde{x})\times B_{\rho_{\Tilde{z}}}(\Tilde{x})\times [\Tilde{t}-\rho_{\Tilde{z}}^{2s},\Tilde{t}]\right)}\leq \frac{\delta}{\sqrt{\lambda}}.
\end{equation*}
Similarly, we take
\begin{equation*}
    \tilde{A}_{2,\tilde{z}}(x,y,t)=
        K_{2,1}(\tilde{x},\tilde{x},\tilde{t})K_{2,2}(x,y,t)\quad\text{for }(x,y,t)\in B_{\rho_{\Tilde{z}}}(\Tilde{x})\times B_{\rho_{\Tilde{z}}}(\Tilde{x})\times [\Tilde{t}-\rho_{\Tilde{z}}^{2s},\Tilde{t}]
\end{equation*}
to find that $\tilde{A}_{2,\tilde{z}}\in \mathcal{L}_{1}\left(\lambda;B_{\rho_{\Tilde{z}}}(\Tilde{x})\times B_{\rho_{\Tilde{z}}}(\Tilde{x})\times [\Tilde{t}-\rho_{\Tilde{z}}^{2s},\Tilde{t}]\right)$ and $A_{2}$ satisfies \eqref{assumption closedness of A}.

\item The assumption \eqref{assumption closedness of A} naturally appears in the literature when a perturbation argument is used, as follows from \cite{NHH}.

\end{enumerate}
\end{rmk}

\begin{rmk}
We here compare a nonlocal case with a local case. For the local parabolic problem  
\begin{equation*}
    \partial_{t}v-\ddiv(B(x,t)Dv)=f \text{ in }Q_{1}
\end{equation*}
when $B(x,t)$ is a continuous coefficient and $f\in L^{q,r}(Q_{1})$ with $\frac{n}{2q}+\frac{1}{r}<1$, we note that $v\in C^{\beta,\frac{\beta}{2}}_{\loc}$ for every $\beta\in\left(0,\min\left\{2-\left(\frac{n}{q}+\frac{2}{r}\right),1\right\}\right)$. See \cite{TUg} for the case $B \equiv 1$. For a general case $B \not\equiv 1$, a freezing argument as in \cite[Theorem 3.8]{QF} leads to the desired result. Based on this observation for the local case, it might be expected for the nonlocal case that a weak solution to \eqref{eq1} has at most $C^{\alpha,\frac{\alpha}{2}}_{\loc}$ regularity for every $\alpha\in\left(0,\min\left\{2s-\left(\frac{n}{q}+\frac{2}{r}\right),s\right\}\right)$. However, we observe that the regularity of $u$ is better than this expectation. Indeed, Theorem \ref{Holder} asserts that it exceeds $C^{s,\frac{s}{2}}_{\loc}$ regularity when $2s-\left(\frac{n}{q}+\frac{2}{r}\right)>s$. This phenomenon is first observed {in \cite{BLS2}} for the elliptic case.

\end{rmk}}

\subsection{Plan of the paper}
Our approach to proving the local boundedness is based on \cite[Chapter 2]{LSU} alongside a suitable modification for the fractional Sobolev space and the nonzero tail term. 
For the higher H\"older regularity, we use a notion of discrete fractional derivatives in \cite{BLS,BLS2, NHH} to prove the almost Lipschitz regularity of a weak solution to a homogeneous equation \eqref{normalized homo} below, and an approximation technique as in \cite{NHH,CS} to transfer its regularity to a weak solution of \eqref{eq1} with \eqref{forcing term} and \eqref{assumption closedness of A}.

The paper is organized as follows. Section \ref{section2} deals with the  notations, embeddings between function spaces, technical lemma, and H\"older regularity for the homogeneous problem. In Section \ref{section3}, we give Caccioppoli-type inequality of \ref{caccioppoli estimate} to prove the local boundedness and H\"older regularity of \eqref{eq1} with \eqref{forcing term}. In Section \ref{section4}, we improve H\"older regularity for the homogeneous problem with the kernel coefficient which is invariant under the translation in only spatial direction. Finally, in Section  \ref{section 5}, we obtain the higher H\"older regularity for \eqref{eq1} with \eqref{assumption closedness of A}. In Appendix \ref{Appendix}, we give the existence for the initial and boundary value problem with the inhomogeneous term $f\in L^{q,r}$, where $\frac{n}{2sq}+\frac{1}{r}\leq 1+\frac{n}{4s}$.


\section{Preliminaries and Notations}
\label{section2}
Throughout the paper, we write $c$ by a general positive constant, possibly changing from line to line. In particular, we give the relevant dependencies on parameters using parentheses, i.e., $c \equiv c(n,s,\alpha)$. 
We now introduce some geometric and function notations.
\begin{enumerate}
\item The standard Lebesgue measures in $\mathbb{R}^{n}$ and $\mathbb{R}$ are denoted by $\dx$ and $\dt$. A usual point in $\mathbb{R}^{n}\times\mathbb{R}$ is $z=(x,t)$.
\item We denote the open ball in $\mathbb{R}^{n}$ with center $x_{0}$ and radius $\rho>0$ by $B_{\rho}(x_{0})$. In particular if the center is obvious, we omit the center.
\item We shall use $Q_{\rho,\tau}(z_{0})\equiv B_{\rho}(x_{0})\times(t_{0}-\tau,t_{0}]$ for {$z_{0} = (x_{0},t_{0})\in\mathbb{R}^{n}\times\mathbb{R}$} and $\tau>0$. Also we write $Q_{\rho}(z_{0})\equiv Q_{\rho,\rho^{2s}}(z_{0})$.
\item For any $Q_{\rho,\tau}(z_{0})$, we define a parabolic boundary of $Q_{\rho,\tau}(z_{0})$ by 
\begin{equation*}
    \partial_{P}Q_{\rho,\tau}(z_{0})=B_{\rho}(x_{0})\times\{t=t_{0}-\tau\}\cup \partial B_{\rho}(x_{0})\times[t_{0}-\tau,t_{0}].
\end{equation*}
\item Given a measurable set $K\subset\mathbb{R}^{n}$, $|K|$ means the volume of $K$ with respect to the Lebesgue measure in $\mathbb{R}^{n}$.
\item Given a measurable set $K\subset\mathbb{R}^{n+1}$ and a measurable function $f:K\rightarrow\mathbb{R}$,
\[\overline{f}_{K}=\frac{1}{|K|}\iint_{K}f\dx\dt=\dashiint_{K}f\dx\dt\]
is the average of $f$ over $K$. In particular, we denote 
\[\overline{f}_{B}(t)=\dashint_{B}f(x,t)\dx, \quad t \in \mathbb{R}, \]
where $B\subset\mathbb{R}^{n}$.
\item Given a measurable function $f:\mathbb{R}^{n+1}\rightarrow\mathbb{R}$ and $h\in\mathbb{R}^{n}$, we define 
\[f_{h}(x,t)=f(x+h,t),\quad\delta_{h}f(x,t)=f_{h}(x,t)-f(x,t)\quad \text{and}\quad \delta_{h}^{2}f=\delta_{h}(\delta_{h}f).\] 
\item $p'$ is denoted by the H\"older conjugate of $p\in[1,\infty]$.
\end{enumerate}

We now describe relevant function spaces.
 Let $u:\Omega\rightarrow \mathbb{R}$ be a measurable function. For $s\in(0,1)$ and $p\in[1,\infty)$, we define a seminorm
\[[u]_{W^{s,p}(\Omega)}=\left(\int_{\Omega}\int_{\Omega}\frac{|u(x)-u(y)|^{p}}{|x-y|^{n+sp}}\dx\dy\right)^{\frac{1}{p}}\]
and the fractional Sobolev space
\[W^{s,p}(\Omega)=\{u\in L^{p}(\Omega)\,;\,[u]_{W^{s,p}(\Omega)}<\infty\}.\]
In particular, we say that $u$ is in $W^{s,p}_{0}(\Omega)$ if $u$ is in $W^{s,p}(\mathbb{R}^{n})$ and $u\equiv0$ in $\mathbb{R}^{n}\setminus\Omega \ a.e$. 

\noindent
Next we introduce a nonlocal tail space. We {write $L^{1}_{2s}(\mathbb{R}^{n})$ as  
\[L^{1}_{2s}(\mathbb{R}^{n}) = \left\{ u \in L^{1}_{loc}(\mathbb{R}^{n})\,; \int_{\mathbb{R}^{n}}\frac{|u(y)|}{(1+|y|)^{n+2s}}\dy<\infty \right\}. \]}
Then we give a suitable nonlocal tail for a parabolic case. 
Given $u\in L^{\infty}\left(I;L^{1}_{2s}(\mathbb{R}^{n})\right)$, we define
\[\ITail(u,z_{0},\rho,\tau)=\esup_{t\in [t_{0}-\tau,t_{0}]}{\rho^{2s}}\int_{\mathbb{R}^{n}\setminus B_{\rho}(x_{0})}\frac{|u(y,t)|}{|y-x_{0}|^{n+2s}}\dy,\]
for any $B_{\rho}(x_{0})\subset\mathbb{R}^{n}$ and $[t_{0}-\tau,t_{0}]\subset I$ where $I$ is a time interval in $\mathbb{R}$.
In particular, we write
\[\ITail(u;z_{0},\rho) \coloneqq \esup_{t\in [t_{0}-\rho^{2s},t_{0}]}{\rho^{2s}}\int_{\mathbb{R}^{n}\setminus B_{\rho}(x_{0})}\frac{|u(y,t)|}{|y-x_{0}|^{n+2s}}\dy.\]
Note that a simple calculation gives that $u\in L^{\infty}(I;L^{1}_{2s}(\mathbb{R}^{n}))$ if and only if
\[\ITail (u,x_{0},\rho,I)<\infty\]
for any $B_{\rho}(x_{0})\subset\mathbb{R}^{n}$. We next introduce Besov-type spaces to get higher regularity as follows.
\begin{defn*}(Besov-type spaces) Let $p\in[1,\infty)$ and  
let $u:\mathbb{R}^{n}\rightarrow\mathbb{R}$ be in $L^{p}(\mathbb{R}^{n})$. Define two seminorms by
\[[u]_{\mathcal{N}_{\infty}^{\beta,p}(\mathbb{R}^{n})}=\sup_{|h|>0}\left\|\frac{\delta_{h}u}{|h|^{\beta}}\right\|_{L^{p}(\mathbb{R}^{n})} \text{ for } \beta\in(0,1]\]
and 
\[[u]_{\mathcal{B}_{\infty}^{\beta,p}(\mathbb{R}^{n})}=\sup_{|h|>0}\left\|\frac{\delta^{2}_{h}u}{|h|^{\beta}}\right\|_{L^{p}(\mathbb{R}^{n})} \text{ for } \beta\in(0,2).\]
Then we define two Besov-type spaces.
\[\mathcal{N}_{\infty}^{\beta,p}(\mathbb{R}^{n})=\left\{u\in L^{p}(\mathbb{R}^{n})\ ;[u]_{\mathcal{N}_{\infty}^{\beta,p}(\mathbb{R}^{n})}<\infty\right\} \text{ for } \beta\in(0,1]\]
and
\[\mathcal{B}_{\infty}^{\beta,p}(\mathbb{R}^{n})=\left\{u\in L^{p}(\mathbb{R}^{n})\ ; [u]_{\mathcal{B}_{\infty}^{\beta,p}(\mathbb{R}^{n})}<\infty\right\} \text{ for } \beta\in(0,2).\]
\end{defn*}

Now we present several lemmas that will be used in the remainder of the paper. The first five results are related to the embeddings.
We start with Campanato's embedding for the fractional version for $s\in(0,1)$. 

We refer to \cite{C,Gc,SDgc} for a local case.
\begin{lem}[Campanato's embedding]
\label{Campanato embedding for fractional}
Let $p>1$ and $s\in(0,1)$.
Let $u\in L^{p}(Q_{2R_{0}}(z_{0}))$. Suppose that there is a $\alpha\in(0,1)$ such that for any $z\in Q_{R_{0}}(z_{0})$ and $\rho>0$ with $Q_{\rho}(z)\subset Q_{2R_{0}}(z_{0})$,
\begin{equation*}
    \dashiint_{Q_{\rho}(z)}|u-\overline{u}_{Q_{\rho}(z)}|^{p}\dx{\dt}\leq M^{p}\rho^{p\alpha},
\end{equation*}
for some constant $M>0$. Then we have $u\in C^{\alpha,\frac{\alpha}{2s}}(Q_{R_{0}}(z_{0}))$. In particular, there is a constant $c \equiv c(n,p,s,\alpha,R_{0})$ such that
\begin{equation*}
    [u]_{C^{\alpha,\frac{\alpha}{2s}}(Q_{R_{0}}(z_{0}))}\leq c\left(M+\|u\|_{L^{\infty}(Q_{R_{0}}(z_{0}))}\right),
\end{equation*}
where
\begin{equation*}
    [u]_{C^{\alpha,\frac{\alpha}{2s}}(Q)}\coloneqq\sup_{\substack{(x,t),(x',t')\in Q\\ (x,t)\neq(x',t')}}\frac{|u(x,t)-u(x',t')|}{|x-x'|^{\alpha}+|t-t'|^{\frac{\alpha}{2s}}}\quad \mbox{for any }Q\subset Q_{2R_{0}(z_{0})}.
\end{equation*}
\end{lem}
\begin{proof}
Take $z_{1}\in Q_{R_{0}}(z_{0})$ and $0<\rho_{1}<\rho_{2}\leq\min\left\{\left(2^{2s}-1\right)^{\frac{1}{2s}}R_{0},R_{0}\right\}$ so that $Q_{\rho_{2}}(z_{1})\subset Q_{2R_{0}}(z_{0})$. Then we observe that 
\begin{align*}
    |\overline{u}_{Q_{\rho_{1}}(z_{1})}-\overline{u}_{Q_{\rho_{2}}(z_{1})}|^{p}\leq cM^{p}\left(\rho_{1}^{p\alpha}+\rho_{2}^{p\alpha}\left(\frac{\rho_{2}}{\rho_{1}}\right)^{n+2s}\right).
\end{align*}
 For any $0<R\leq\min\left\{\left(2^{2s}-1\right)^{\frac{1}{2s}}R_{0},R_{0}\right\}$ with $\rho_{1}=2^{-i-1}R$ and $\rho_{2}=2^{-i}R$, we see that
\begin{equation*}
    |\overline{u}_{Q_{2^{-i-1}R}(z_{1})}-\overline{u}_{Q_{2^{-i}R}(z_{1})}|\leq c2^{-\alpha(i+1)}MR^{\alpha}.
\end{equation*}
With the standard argument as in \cite[Theorem 3.1]{QF}, we have $\|u\|_{L^{\infty}(Q_{R_{0}}(z_{0}))}<\infty$ and
\begin{equation*}
    |u(z_{1})-\overline{u}_{Q_{\rho}(z_{1})}|\leq cM\rho^{\alpha} \quad\text{for any  }\rho\leq\min\left\{\left(2^{2s}-1\right)^{\frac{1}{2s}}R_{0},R_{0}\right\}.
\end{equation*}
Fix $z_{1},z_{2}\in Q_{R_{0}}(z_{0})$ with $R=\max\left\{|x_{1}-x_{2}|,|t_{1}-t_{2}|^{\frac{1}{2s}}\right\}<\min\left\{\left(2^{2s}-1\right)^{\frac{1}{2s}}\frac{R_{0}}{2},\frac{R_{0}}{2}\right\}$ , then we get 
\begin{equation*}
    Q_{2R}(z_{i})\subset Q_{2R_{0}}(z_{0}) \text{ for }i=1,2.
\end{equation*}
As in \cite[Theorem 3.1]{QF}, we have the following 
\begin{equation}
\label{campanato hol 1}
    |u(z_{1})-u(z_{2})|\leq cMR^{\alpha}\leq cM\left(|x_{1}-x_{2}|^{\alpha}+|t_{1}-t_{2}|^{\frac{\alpha}{2s}}\right).
\end{equation}
On the other hand, we deduce
\begin{equation}
\label{campanato hol 2}
    |u(z_{1})-u(z_{2})|\leq c\|u\|_{L^{\infty}(Q_{R_{0}}(z_{0}))}\left(|x_{1}-x_{2}|^{\alpha}+|t_{1}-t_{2}|^{\frac{\alpha}{2s}}\right),
\end{equation}
whenever $z_{1},z_{2}\in Q_{R_{0}}(z_{0})$ with $\max\left\{|x_{1}-x_{2}|,|t_{1}-t_{2}|^{\frac{1}{2s}}\right\}>\min\left\{\left(2^{2s}-1\right)^{\frac{1}{2s}}\frac{R_{0}}{2},\frac{R_{0}}{2}\right\}$.
We combine \eqref{campanato hol 1} and \eqref{campanato hol 2} to complete the proof.
\end{proof}

\noindent
We write $V^{2}_{s}(B_{R}\times I) \equiv L^{2}(I;W^{s,2}(B_{R}))\cap L^{\infty}(I;L^{2}(B_{R}))$ and its norm is given by
\[\|u\|_{V^{2}_{s}(B_{R}\times I)}=\left(\int_{t_{0}-\tau}^{t_{0}}[u(\cdot,t)]^{2}_{{W}^{s,2}(B_{R})}\dt\right)^{\frac{1}{2}}+\esup_{t\in I}\|u(\cdot,t)\|_{L^{2}(B_{R})}.\] 
Then we prove the following embedding which is a essential tool to deal with a nonhomogeneous term in \eqref{eq1}. We refer to {\cite[Chapter 1, Proposition 3.3]{DE}} for the local case when $s=1$.
{\begin{lem}
\label{embedSF}
Let $f\in V_{s}^{2}(B_{R}\times I)$. Suppose that \begin{equation*}
\begin{cases}
    \hat{q}\in \left[2,\frac{2n}{n-2s}\right]&\quad\text{and}\quad \hat{r}\in \left[2,\infty\right]\quad\text{if }2s<n,\\
    \hat{q}\in [2,\infty)&\quad\text{and}\quad \hat{r}\in \left(4s,\infty\right]\quad\text{if }2s=n,\\
    \hat{q}\in [2,\infty]&\quad\text{and}\quad \hat{r}\in [4s,\infty]\quad\text{if }n<2s.
\end{cases}
\end{equation*} satisfy
\begin{equation*}
\begin{aligned}
 \frac{n}{2\hat{q}s}+\frac{1}{\hat{r}}=\frac{n}{4s}.
\end{aligned}
\end{equation*}
Then there is a constant $c \equiv c(n,s,\hat{q},\hat{r})$ such that
\begin{align*}
    \|f\|_{L^{\hat{q},\hat{r}}(B_{R}\times I)}&\leq cR^{-s}\|f\|_{L^{2}(I;L^{2}(B_{R}))}+c\|f\|_{V^{2}_{s}(B_{R}\times I)}.
\end{align*}
Moreover if $f\in L^{2}(I;W^{s,2}_{0}(B_{R}))\cap L^{\infty}(I;L^{2}(B_{R}))$, we also have 
\begin{align}
\label{zero embedding}
    \|f\|_{L^{\hat{q},\hat{r}}(B_{R}\times I)}&\leq c\left(\left(\int_{t_{0}-\tau}^{t_{0}}[f(\cdot,t)]^{2}_{{W}^{s,2}(\mathbb{R}^{n})}\right)^{\frac{1}{2}}+\esup_{t\in I}\|f(\cdot,t)\|_{L^{2}(B_{R})}\right)
\end{align}
for some constant $c \equiv c(n,s,\hat{q},\hat{r})$.
\end{lem}
\begin{proof}
If $(\hat{q},\hat{r})=(2,\infty)$ and $(\hat{q},\hat{r})=\left(\frac{2n}{n-2s},2\right)$, then we check the above results directly. Therefore we may assume that $2<\hat{r}<\infty$.  We observe that for $\alpha=\frac{2}{\hat{r}}\in(0,1)$, 
\[\alpha\frac{n-2s}{2n}+(1-\alpha)\frac{1}{2}=\frac{1}{\hat{q}}.\]
Applying scaled version of \cite[Lemma 2.1]{DZZ} with $p=p_{2}=2$, $p_{1}=\hat{q}$ and $\theta=\alpha$, we have
\begin{equation}
\label{gnineq}
\|f(\cdot,t)\|_{L^{\hat{q}}(B_{R})}\leq c\left([f(\cdot,t)]_{W^{s,2}(B_{R})}+R^{-s}\|f(\cdot,t)\|_{L^{2}(B_{R})}\right)^{\alpha}\|f(\cdot,t)\|_{L^{2}(B_{R})}^{1-\alpha} \quad \text{a.e. }t\in I,
\end{equation}
where $c=c(n,s,\hat{q})$.
Using \eqref{gnineq} and young's inequality, we get that
\begin{align*}
    \left(\int_{I}\|f(\cdot,t)\|_{L^{\hat{q}}(B_{R})}^{\hat{r}}\dt\right)^{\frac{1}{\hat{r}}}&\leq c\left(\left(\int_{I}[f(\cdot,t)]^{2}_{{W}^{s,2}(B_{R})}\dt\right)^{\frac{1}{2}}+R^{-s}\|f\|_{L^{2}(I;L^{2}(B_{R}))}\right)^{\frac{2}{\hat{r}}}\\
    &\qquad\times\esup_{t\in I}\|f(\cdot,t)\|_{L^{2}(B_{R})}^{1-\frac{2}{\hat{r}}}\\
    &\leq cR^{-s}\|f\|_{L^{2}(I;L^{2}(B_{R}))}+c\|f\|_{V^{2}_{s}(B_{R}\times I)}
\end{align*}
for some constant $c=c(n,s,\hat{q},\hat{r})$.
In particular, \eqref{zero embedding} is a direct consequence from \cite[Theorem 6.5]{DPGV}.
\end{proof}}Next, we see three Besov-type embeddings without proof.\begin{lem}\cite[Lemma 2.6.]{BLS2}
\label{embed1}
Let $0<\beta<1$ and $1\leq p<\infty$. Let $u$ be in $L^{p}(B_{2R_{0}})$ for some $R_{0}>0$. Suppose for some $0<h_{0}<\frac{R_{0}}{4}$, we have 
\begin{equation*}
    \sup_{0<|h|<h_{0}}\left\|\frac{\delta_{h}^{2}u}{|h|^{\beta}}\right\|_{L^{p}(B_{R_{0}+h_{0}})}<\infty.
\end{equation*}
Then we obtain
\begin{equation*}
    \sup_{0<|h|<h_{0}}\left\|\frac{\delta_{h}u}{|h|^{\beta}}\right\|_{L^{p}(B_{R_{0}})}\leq c\left(\sup_{0<|h|<h_{0}}\left\|\frac{\delta_{h}^{2}u}{|h|^{\beta}}\right\|_{L^{p}(B_{R_{0}+h_{0}})}+\|u\|^{p}_{L^{p}(B_{R_{0}+h_{0}})}\right)
\end{equation*}
for some $c \equiv c(n,p,\beta,h_{0})$.
\end{lem}
\begin{lem}\cite[Lemma 2.8.]{BLS2}
\label{Holder for Besov}
Let $0<\beta<1$ and $1\leq p<\infty$ with $p\beta>n$. If $u\in\mathcal{N}^{\beta,p}_{\infty}(\mathbb{R}^{n})$, then $u\in C^{\alpha}_{\loc}(\mathbb{R}^{n})$ for any $0<\alpha<\beta-\frac{n}{p}$ with 
\begin{equation*}
    \frac{|u(x)-u(y)|}{|x-y|^{\alpha}}\leq c\left([u]_{\mathcal{N}^{\beta,p}_{\infty}(\mathbb{R}^{n})}\right)^{\frac{p\alpha+n}{p\beta}}\left(\|u\|_{L^{p}(\mathbb{R}^{n})}\right)^{1-\frac{p\alpha+n}{p\beta}} \text{ for any }x\neq y \text{ in }\mathbb{R}^{n}
\end{equation*}
where $c \equiv c(n,p,\alpha,\beta)$. In particular, if $u\in L^{p}(B_{R+2h_{0}})$ with $R>0$ and $h_{0}>0$  
then for any $\alpha\in(0,\beta-\frac{n}{p})$, there is a constant $c \equiv c(n,p,\alpha,\beta,M,N,R,h_{0})$ such that
\begin{equation*}
    \frac{|u(x)-u(y)|}{|x-y|^{\alpha}}\leq c,\quad \text{for any }x\neq y\in B_{\frac{R}{2}},
\end{equation*}
provided that
\begin{equation*}
    \sup_{0<|h|<h_{0}}\left\|\frac{\delta_{h}u}{|h|^{\beta}}\right\|_{L^{p}(B_{R+h_{0}})}\leq M\text{ and }\|u\|_{L^{p}(B_{R+h_{0}})}\leq N,\quad \text{for } M\text{ and }N>0.
\end{equation*}
\end{lem}
\noindent
\begin{rmk}
Using a cutoff function with the first statement, we can prove the second statement. See in \cite[Theorem 4.2.]{BLS2}.
\end{rmk}
\begin{lem}\cite[Proposition 2.6.]{BL}
\label{embedSB}
Let $s\in(0,1)$. We have two embeddings.
\begin{enumerate}
\item Suppose $u\in W^{s,2}(\mathbb{R}^{n})$. Then we get
\[\sup_{|h|>0}\left\|\frac{\delta_{h}u}{|h|^{s}}\right\|^{2}_{L^{2}(\mathbb{R}^{n})}\leq c(n,s)[u]^{2}_{W^{s,2}(\mathbb{R}^{n})}\]
\item Suppose $u\in W^{s,2}_{\loc}(\Omega)$ where $\Omega\subset\mathbb{R}^{n}$ is open set and $B_{R}\Subset\Omega$ with $h_{0}<\frac{\dist(B_{R},\partial\Omega)}{2}$. Then we deduce 
\[\sup_{0<|h|<h_{0}}\left\|\frac{\delta_{h}u}{|h|^{s}}\right\|^{2}_{L^{2}(B_{R})}\leq c(s,n,R,h_{0})\left([u]^{2}_{W^{s,2}(B_{R+h_{0}})}+\|u\|^{2}_{L^{2}(B_{R+h_{0}})}\right)\]
\end{enumerate}
\end{lem}
For a local boundedness of a weak solution \eqref{eq1} with \eqref{forcing term}, we need the following technical lemma.
\begin{lem}\cite[Chapter 1, {Lemma 4.2}]{DE}
\label{TLI}
Let $M,b>1$ and $\kappa,\delta>0$ be given. For each $h\in\mathbb{N}$, 
\begin{align*}
    &Y_{h+1}\leq Mb^{h}(Y_{h}^{1+\delta}+Z_{h}^{1+\kappa}Y_{h}^{\delta})\\
    &Z_{h+1}\leq Mb^{h}(Y_{h}+Z_{h}^{1+\kappa}).
\end{align*}
If $Y_{0}+Z_{0}^{1+\kappa}\leq(2M)^{-\frac{1+\kappa}{\sigma}}b^{-\frac{1+\kappa}{\sigma^{2}}}$ where $\sigma=\min\{\kappa,\delta\}$, then
\[\lim_{h\to\infty}Y_{h}=\lim_{h\to\infty}Z_{h}=0.\]
\end{lem}

Before ending this section, we give the H\"older regularity of a homogeneous equation to transfer the regularity to a weak solution to \eqref{eq1} with \eqref{forcing term}. In {\cite{DZZ,N1,APTl}}, the H\"older continuity is proved when $\Phi(t)=t$. In a similar way, we can prove the H\"older continuity when $\Phi$ satisfies \eqref{Phi condition} with some minor modifications. Therefore we get the following H\"older estimate.
\begin{lem}
\label{HomoHolder}
Suppose that 
\[u\in L^{2}\left((-1,0];W^{s,2}(B_{1})\right)\cap L^{\infty}\left((-1,0];L^{1}_{2s}(\mathbb{R}^{n})\right)\cap C\left((-1,0];L^{2}(B_{1})\right)\]
is a local weak solution to 
\begin{equation*}
    \partial_{t}u+\mathcal{L}^{\Phi}_{A}u=0 \text{ in }Q_{1}.
\end{equation*}
Then we have that $u$ is in $C_{\loc}^{\beta,\frac{\beta}{2s}}(Q_{1})$ for some $\beta=\beta(n,s,\lambda)$. In particular, we obtain for any $0<\rho\leq\frac{R}{2}$ with $Q_{R}(z_{0})\Subset Q_{1}$ 
\begin{equation*}
    \osc_{Q_{\rho}(z_{0})}u\leq c\left(\frac{\rho}{R}\right)^{\beta}\left[\left(\dashiint_{Q_{R}}|u|^{2}\dx\dt\right)^{\frac{1}{2}}+\ITail\left(u;x_{0},\frac{R}{2},t_{0}-R^{2s},t_{0}\right)\right]
\end{equation*}
and
\begin{equation*}
    \dashiint_{Q_{\rho}(z_{0})}|u-\overline{u}_{Q_{\rho}(z_{0})}|^{2}\dx\dt\leq c\left(\frac{\rho}{R}\right)^{2\beta}\left[\left(\dashiint_{Q_{R}}|u|^{2}\dx\dt\right)+\ITails\left(u;x_{0},\frac{R}{2},t_{0}-R^{2s},t_{0}\right)\right],
\end{equation*}
where $c \equiv c(n,s,\lambda)$.
\end{lem}


\section{Local boundedness and H\"older regularity for inhomogeneous equation}
\label{section3}
Let $\zeta :\mathbb{R}\to\mathbb{R}$ be a smooth even function with $\supp\zeta\subset(-\frac{1}{2},\frac{1}{2})$ and $\int_{\mathbb{R}}\zeta=1$. For any locally integrable function $u:\Omega\times(0,T)\to \mathbb{R}$, we define 
\begin{align*}
    u^{\epsilon}(x,t)\coloneqq\dashint_{t-\frac{\epsilon}{2}}^{t+\frac{\epsilon}{2}}\zeta\left(\frac{t-\sigma}{\epsilon}\right)u(x,\sigma)d\sigma=\dashint_{-\frac{1}{2}}^{\frac{1}{2}}\zeta(\sigma)u(x,t-\epsilon\sigma)d\sigma,\quad (x,t)\in\Omega_{T}
\end{align*}
for $0<\epsilon<\min\{t,T-t\}$.
We check the following elementary lemma.
\begin{lem}(Property of mollification)
\label{MT}
Let $X$ be a seperable Banach space and $1\leq p<\infty$.
\begin{enumerate}
\item Suppose $u\in C([0,T];X)$. Then for each $t\in(0,T)$, 
\begin{align*}
    u^{\epsilon}(\cdot,t)\to u(\cdot,t) \text{ in } X 
\end{align*}
as $\epsilon$ tends to 0.

\item Suppose $u\in L^{p}(0,T;X)$. Then 
\begin{align*}
    u^{\epsilon}\to u \text{ in } L^{p}_{loc}(0,T;X)
\end{align*}
as $\epsilon$ tends to 0.
\end{enumerate}
\end{lem}

First, we give a Caccioppoli-type estimate, Lemma \ref{caccioppoli estimate}, for \eqref{eq1} with \eqref{forcing term}.
Let us consider a parameter $\kappa\in\left(0,\frac{2s}{n}\right)$ so that
\begin{equation}
\label{kappa condition}
    0<\frac{n}{2qs}+\frac{1}{r}=1-\frac{n}{2s}\kappa<1.
\end{equation} 
Let us also consider two parameters $\hat{q}\coloneqq2(1+\kappa)q'$ and $\hat{r}\coloneqq2(1+\kappa)r'$
to see that 
\begin{equation}
\label{hatqhatr}
\frac{n}{2\hat{q}s}+\frac{1}{\hat{r}}=\frac{n}{4s},\quad \hat{q}\in\left[2,\frac{2n}{n-2s}\right]\quad\mbox{and}\quad\hat{r}\in[2,\infty].
\end{equation}
If a function $u$ belongs to $L^{2}(0,T;W^{s,2}(\Omega))$, then for $x_{0}\in\Omega$, $\rho>0$, $k\geq0$, and $t\in(0,T)$, we define 
\begin{equation}
\label{def of level set}
    {E(t;x_{0},\rho,k)}=\{x\in B_{\rho}(x_{0}) \ ; u(x,t)>k\}\quad \text{and}\quad w_{+}=(u-k)_{+},
\end{equation}
where $B_{\rho}(x_{0})\subset\Omega$.

\begin{lem}
\label{caccioppoli estimate}
Suppose that $u$ is a local weak subsolution to \eqref{eq1} with \eqref{forcing term}. Then for any $Q_{R,T_{2}}(z_{0})\Subset\Omega_{T}$ with $0<\rho<R$ and $0<T_{1}<T_{2}<R^{2s}$, there exists a constant $c \equiv c(n,s,q,r,\lambda)$ such that
\begin{align*}
    &\int_{t_{0}-T_{1}}^{t_{0}}\int_{B_{\rho}(x_{0})}\int_{B_{\rho}(x_{0})}\frac{|w_{+}(x,t)-w_{+}(y,t)|^{2}}{|x-y|^{n+2s}}\dx\dy\dt+
    \esup_{t\in[t_{0}-T_{1},t_{0}]}\int_{B_{\rho}(x_{0})}w_{+}^{2}(x,t)\dx\\
    &\quad\leq c\left(\frac{R^{2(1-s)}}{(R-\rho)^{2}}+\frac{1}{T_{2}-T_{1}}\right)\int_{t_{0}-T_{2}}^{t_{0}}\int_{B_{R}(x_{0})}w_{+}^{2}\dx\dt\\
    &\qquad+c\left(\frac{R^{n+2s}}{(R-\rho)^{n+2s}}\esup_{t\in[t_{0}-T_{2},t_{0}]}\int_{\mathbb{R}^{n}\setminus B_{\rho}(x_{0})}\frac{w_{+}(y,t)}{|x_{0}-y|^{n+2s}}\dy\right)\times\|w_{+}\|_{L^{1}(Q_{R,T_{2}}(z_{0}))}\\
    &\qquad+ck^{2}R^{-n\kappa}\left(\int_{t_{0}-T_{2}}^{t_{0}}|{E(t;x_{0},R,k)}|^{\frac{\hat{r}}{\hat{q}}}\dt\right)^{\frac{2(1+\kappa)}{\hat{r}}},
\end{align*}
whenever $k\geq\|f\|_{L^{q,r}(Q_{R,T_{2}}(z_{0}))}R^{{n\kappa}}$.
\end{lem}
\begin{proof}

Let $\epsilon>0$ be a sufficiently small number so that $J\coloneqq[t_{0}-T_{2}-\epsilon,t_{0}+\epsilon]\Subset(0,T)$ and $t_{0}-T_{2}+\epsilon<t_{0}-\frac{T_{1}+T_{2}}{2}-\epsilon$.
We take a nonnegative function $\eta=\eta(t)\in C^{\infty}(\mathbb{R})$ satisfying 
\begin{equation}
\label{time test function}
\eta(t)\equiv1\quad \mbox{in }t\geq t_{0}-T_{1},\quad \eta(t)\equiv0 \quad\mbox{in } t\leq t_{0}-\left(\frac{T_{1}+T_{2}}{2}\right),\quad \|\eta'(t)\|_{L^{\infty}(\mathbb{R})}\leq\frac{4}{T_{2}-T_{1}},
\end{equation}
and choose a cutoff function $\psi=\psi(x)\in C_{c}^{\infty}(B_{{(R+\rho)/2}}(x_{0}))$ such that 
\begin{equation}
\label{space test function}
\psi\equiv 1\quad\mbox{in } B_{\rho}(x_{0})\quad\mbox{and}\quad \|\nabla\psi\|_{L^{\infty}(\mathbb{R}^{n})}\leq\frac{{4}}{R-\rho}.
\end{equation}
Define
\begin{equation*}
    \phi_{{\epsilon}}(x,t)=(w_{+}^{\epsilon}\psi^{2}\eta^{2})^{\epsilon}(x,t),\quad (x,t)\in B_{R}(x_{0})\times J,
\end{equation*}
where $w_{+}^{\epsilon}$ is the convolution of $w_{+}$ with respect to time variable. Note that 
\begin{equation}
\label{test limit}
    \phi_{\epsilon}\to \phi \coloneqq w_{+}\psi^{2}\eta^{2}  \mbox{ in } L^{2}\big(J;W^{s,2}(B_{R}(x_{0}))\big) \mbox{ as }\epsilon\to0,
\end{equation}
by Lemma \ref{MT}. Take $\phi_{\epsilon}$ as a test function to find
\begin{align*}
    I_{1}^{\epsilon}+I_{2}^{\epsilon}&\coloneqq\int_{t_{0}-T_{2}}^{\tau}\int_{B_{R}(x_{0})}-u\partial_{t}\phi_{\epsilon} \dx\dt\\
    &\quad+\int_{t_{0}-T_{2}}^{\tau}\int_{\mathbb{R}^{n}}\int_{\mathbb{R}^{n}}\Phi(u(x,t)-u(y,t))(\phi_{\epsilon}(x,t)-\phi_{\epsilon}(y,t))\frac{A(x,y,t)}{|x-y|^{n+2s}}\dx\dy\dt\\
    &\leq\int_{t_{0}-T_{2}}^{\tau}\int_{B_{R}(x_{0})}f\phi_{\epsilon}\dx\dt-\int_{B_{R}(x_{0})}u\phi_{\epsilon}\dx\Bigg\rvert_{t=t_{0}-T_{2}}^{t=\tau}\eqqcolon I_{3}^{\epsilon}+I_{4}^{\epsilon},
\end{align*}
for any $\tau\in\left[t_{0}-\left(\frac{T_{1}+T_{2}}{2}\right),t_{0}\right]$.
Now we find the limits of $I_{1}^{\epsilon},I_{2}^{\epsilon},I_{3}^{\epsilon}$, and $I_{4}^{\epsilon}$.

\medskip
\noindent
\textbf{Estimate of $I_{1}^{\epsilon}$.} Using Fubini's theroem, we obtain
\begin{align*}
I_{1}^{\epsilon}&=-\int_{B_{R}(x_{0})}\dashint^{t_{0}-T_{2}+\frac{\epsilon}{2}}_{t_{0}-T_{2}-\frac{\epsilon}{2}}\int_{t_{0}-T_{2}}^{\sigma+\frac{\epsilon}{2}}\frac{1}{\epsilon}u(x,t)\zeta'\left(\frac{t-\sigma}{\epsilon}\right){\left(w_{+}^{\epsilon}\psi^{2}\eta^{2}\right)}(x,\sigma)\dt\dsigma\dx\\
&\quad-\int_{B_{R}(x_{0})}\dashint^{\tau+\frac{\epsilon}{2}}_{\tau-\frac{\epsilon}{2}}\int^{\tau}_{\sigma-\frac{\epsilon}{2}}\frac{1}{\epsilon}u(x,t)\zeta'\left(\frac{t-\sigma}{\epsilon}\right){\left(w_{+}^{\epsilon}\psi^{2}\eta^{2}\right)}(x,\sigma)\dt\dsigma\dx\\
&\quad+\int_{B_{R}(x_{0})}\int_{t_{0}-T_{2}+\frac{\epsilon}{2}}^{\tau-\frac{\epsilon}{2}}\partial_{t}u^{\epsilon}(x,t){\left(w_{+}^{\epsilon}\psi^{2}\eta^{2}\right)}(x,t)\dt\dx\eqqcolon-I_{1,1}^{\epsilon}-I_{1,2}^{\epsilon}+I_{1,3}^{\epsilon}.
\end{align*}
By using a suitable change of variables, we write $I_{1,1}^{\epsilon}$ the following 
\begin{align*}
I_{1,1}^{\epsilon}=\int_{B_{R}(x_{0})}\int_{-\frac{1}{2}}^{\frac{1}{2}}\int_{0}^{\sigma+\frac{1}{2}}u(x,\epsilon t+t_{0}-T_{2})\zeta'(t-\sigma){\left(w_{+}^{\epsilon}\psi^{2}\eta^{2}\right)}(x,\epsilon \sigma+t_{0}-T_{2})\dt\dsigma\dx.
\end{align*}
Thus, we have
\begin{align*}
\lim_{\epsilon\to0}I_{1,1}^{\epsilon}
&=\int_{B_{R}(x_{0})}\int_{-\frac{1}{2}}^{\frac{1}{2}}\int_{0}^{\sigma+\frac{1}{2}}u(x,t_{0}-T_{2}){\phi(x,t_{0}-T_{2})}\zeta'(t-\sigma)\dt\dsigma\dx\\
&=-\int_{B_{R}(x_{0})}u(x,t_{0}-T_{2})\phi(x,t_{0}-T_{2})\dx
\end{align*}
by the fact that $u,w_{+}\in C(J;L^{2}(\Omega))$ with Lemma \ref{MT}. Similarly, we get
\begin{align*}
&\lim_{\epsilon\to0}I_{1,2}^{\epsilon}=\int_{B_{R}(x_{0})}u(x,\tau)\phi(x,\tau)\dx.
\end{align*}
For $I_{1,3}^{\epsilon}$, we use an integration by parts, which gives
\begin{align*}
    I_{1,3}^{\epsilon}&=\int_{B_{R}(x_{0})}\int_{t_{0}-T_{2}+\frac{\epsilon}{2}}^{\tau-\frac{\epsilon}{2}}\partial_{t}w_{+}^{\epsilon}{\left(w_{+}^{\epsilon}\psi^{2}\eta^{2}\right)}\dt\dx\\
    &=\int_{B_{R}(x_{0})}\int_{t_{0}-T_{2}+\frac{\epsilon}{2}}^{\tau-\frac{\epsilon}{2}}\partial_{t}\left(\frac{w_{+}^{\epsilon}(x,t)^{2}}{2}\right)\psi^{2}(x)\eta^{2}(t)\dt\dx\\
    &=-\int_{B_{R}(x_{0})}\int_{t_{0}-T_{2}+\frac{\epsilon}{2}}^{\tau-\frac{\epsilon}{2}}\frac{w_{+}^{\epsilon}(x,t)^{2}}{2}\psi^{2}(x)\left(\eta^{2}(t)\right)'\dt\dx\\
    &\quad+\int_{B_{R}(x_{0})}\frac{w_{+}^{\epsilon}(x,t)^{2}}{2}\psi^{2}(x)\eta^{2}(t)\dx\Bigg\rvert_{t=t_{0}-T_{2}-\frac{\epsilon}{2}}^{t=\tau+\frac{\epsilon}{2}}
\end{align*}
As in $I_{1,1}^{\epsilon}$, we observe that 
\begin{align*}
    \lim_{\epsilon\to0}I_{1,3}^{\epsilon}&=-\int_{t_{0}-T_{2}}^{\tau}\int_{B_{R}(x_{0})}w_{+}^{2}(x,t)\psi^{2}(x)\eta(t)\eta'(t)\dt\dx
    +\int_{B_{R}(x_{0})}\frac{w_{+}^{2}(x,t)}{2}\psi^{2}(x)\eta^{2}(t)\dx\Bigg\rvert_{t=t_{0}-T_{2}}^{t=\tau}.
\end{align*}
Combining all the limits of $I_{1,1}^{\epsilon}$, $I_{1,2}^{\epsilon}$, and $I_{1,3}^{\epsilon}$, we find that
\begin{align*}
    \lim_{\epsilon\to0}I_{1}^{\epsilon}&=-\int_{t_{0}-T_{2}}^{t_{1}}\int_{B_{R}(x_{0})}w_{+}^{2}(x,t)\psi^{2}(x)\eta(t)\eta'(t)\dt\dx+\int_{B_{R}(x_{0})}\frac{w_{+}^{2}(x,t)}{2}\psi^{2}(x)\eta^{2}(t)\dx\Bigg\rvert_{t=t_{0}-T_{2}}^{t=\tau}\\
    &\qquad-\int_{B_{R}(x_{0})}u(x,t){\phi(x,t)}\dx\Bigg\rvert_{t=t_{0}-T_{2}}^{t=\tau}.
\end{align*}
\textbf{Estimate of $I_{2}^{\epsilon}$.} We write $I_{2}^{\epsilon}$ as follows:
\begin{align*}
    I_{2}^{\epsilon}&=\int_{t_{0}-T_{2}}^{\tau}\int_{B_{R}(x_{0})}\int_{B_{R}(x_{0})}\Phi(u(x,t)-u(y,t))(\phi_{\epsilon}(x,t)-\phi_{\epsilon}(y,t))\frac{A(x,y,t)}{|x-y|^{n+2s}}\dx\dy\dt\\
    &\quad+\int_{t_{0}-T_{2}}^{\tau}\int_{\mathbb{R}^{n}\setminus B_{R}(x_{0})}\int_{B_{R}(x_{0})}\Phi(u(x,t)-u(y,t))\phi_{\epsilon}(x,t)\frac{A(x,y,t)}{|x-y|^{n+2s}}\dx\dy\dt\\
    &\quad-\int_{t_{0}-T_{2}}^{\tau}\int_{B_{R}(x_{0})}\int_{\mathbb{R}^{n}\setminus B_{R}(x_{0})}\Phi(u(x,t)-u(y,t))\phi_{\epsilon}(y,t)\frac{A(x,y,t)}{|x-y|^{n+2s}}\dx\dy\dt\\
    &\eqqcolon I_{2,1}^{\epsilon}+I_{2,2}^{\epsilon}-I_{2,3}^{\epsilon}.
\end{align*}
Using H\"older's inequality, we observe that
\begin{equation}
\label{moll limit1}
\begin{aligned}
    &\left|I_{2,1}^{\epsilon}-\int_{t_{0}-T_{2}}^{\tau}\int_{B_{R}(x_{0})}\int_{B_{R}(x_{0})}\Phi(u(x,t)-u(y,t))
    (\phi(x,t)-\phi(y,t))\frac{A(x,y,t)}{|x-y|^{n+2s}}\dx\dy\dt\right|\\
    &\quad \leq\lambda^{2}\|u\|_{L^{2}\big(J;W^{s,2}(B_{R}(x_{0}))\big)}\|\phi_{\epsilon}-\phi\|_{L^{2}\big(J;W^{s,2}(B_{R}(x_{0}))\big)}.
\end{aligned}
\end{equation}
We next focus on the estimate of $I_{2,2}^{\epsilon}$. Due to the fact that 
\begin{equation}
\label{tail point}    
    |y-x|\geq|y-x_{0}|-|x-x_{0}|\geq \frac{(R-\rho)}{{2R}}|y-x_{0}|,\quad x\in B_{{(R+\rho)/2}}(x_{0})\subset \supp\psi,\, y\in\mathbb{R}^{n}\setminus B_{R}(x_{0})
\end{equation}and H\"older's inequality,
we have
\begin{equation}
\label{moll limit2}
\begin{aligned}
    &\left|I_{2,2}^{\epsilon}-\int_{t_{0}-T_{2}}^{\tau}\int_{\mathbb{R}^{n}\setminus B_{R}(x_{0})}\int_{B_{R}(x_{0})}\Phi(u(x,t)-u(y,t))\phi(x,t)\frac{A(x,y,t)}{|x-y|^{n+2s}}\dx\dy\dt\right|\\
    &\quad\leq \lambda^{2}\int_{t_{0}-T_{2}}^{\tau}\int_{\mathbb{R}^{n}\setminus B_{R}(x_{0})}\int_{B_{R}(x_{0})}\frac{|u(x,t)|+|u(y,t)|}{|x-y|^{n+2s}}(\phi_{\epsilon}(x,t)-\phi(x,t))\dx\dy\dt\\
    &\quad\leq c\int_{t_{0}-T_{2}}^{\tau}\int_{\mathbb{R}^{n}\setminus B_{R}(x_{0})}\int_{B_{R}(x_{0})}\frac{|u(x,t)|}{|x_{0}-y|^{n+2s}}(\phi_{\epsilon}(x,t)-\phi(x,t))\dx\dy\dt\\
    &\qquad+c\int_{t_{0}-T_{2}}^{\tau}\int_{\mathbb{R}^{n}\setminus B_{R}(x_{0})}\int_{B_{R}(x_{0})}\frac{|u(y,t)|}{|x_{0}-y|^{n+2s}}(\phi_{\epsilon}(x,t)-\phi(x,t))\dx\dy\dt\\
    &\quad\leq c(\rho,R)\|u\|_{L^{2}\big(J;L^{2}(B_{R})(x_{0}))\big)}\|\phi_{\epsilon}-\phi\|_{L^{2}\big(J;L^{2}(B_{R}(x_{0}))\big)}\\
    &\qquad+c(\rho,R)\ITail(u;x_{0},R,J)\|\phi_{\epsilon}-\phi\|_{L^{2}\big(J;L^{2}(B_{R}(x_{0}))\big)}.
\end{aligned}
\end{equation}
Similarly, we deduce
\begin{equation}
\label{moll limit3}
\begin{aligned}
&\left|I_{2,3}^{\epsilon}-\int_{t_{0}-T_{2}}^{\tau}\int_{B_{R}(x_{0})}\int_{\mathbb{R}^{n}\setminus B_{R}(x_{0})}\Phi(u(x,t)-u(y,t))\phi(y,t)\frac{A(x,y,t)}{|x-y|^{n+2s}}\dx\dy\dt\right|\\
&\quad\leq c(\rho,R)\|u\|_{L^{2}\big(J;L^{2}(B_{R})(x_{0}))\big)}\|\phi_{\epsilon}-\phi\|_{L^{2}\big(J;L^{2}(B_{R}(x_{0}))\big)}\\
    &\qquad+c(\rho,R)\ITail(u;x_{0},R,J)\|\phi_{\epsilon}-\phi\|_{L^{2}\big(J;L^{2}(B_{R}(x_{0}))\big)}.
\end{aligned}
\end{equation}
Combining the above estimates \eqref{moll limit1}, \eqref{moll limit2}, \eqref{moll limit3} and using the fact \eqref{test limit}, we discover 
\begin{align*}
\lim_{\epsilon\to0}I_{2}^{\epsilon}&=\int_{t_{0}-T_{2}}^{\tau}\int_{\mathbb{R}^{n}}\int_{\mathbb{R}^{n}}\Phi(u(x,t)-u(y,t))(\phi(x,t)-\phi(y,t))\frac{A(x,y,t)}{|x-y|^{n+2s}}\dx\dy\dt.
\end{align*} 
\textbf{Estimate of $I^{\epsilon}_{3}$.} From H\"older's inequality and Lemma \ref{embedSF}, we get
\begin{align*}
    &\left|I_{3}^{\epsilon}-\int_{t_{0}-T_{2}}^{\tau}\int_{B_{R}(x_{0})}f(x,t)\phi(x,t)\dx\dt\right|\\
    &\quad\leq \|f\|_{L^{r}\big(J;L^{q}(B_{R}(x_{0}))\big)}\|\phi_{\epsilon}-\phi\|_{L^{r'}\big(J;L^{q'}(B_{R}(x_{0}))\big)}\\
    &\quad\leq c\|f\|_{L^{r}\big(J;L^{q}(B_{R}(x_{0}))\big)}\|\phi_{\epsilon}-\phi\|_{V^{s}_{2}((B_{R}(x_{0})\times J))}\\
    &\qquad+c(R)\|f\|_{L^{r}\big(J;L^{q}(B_{R}(x_{0}))\big)}\|\phi_{\epsilon}-\phi\|_{L^{2}\big(J;L^{2}(B_{R}(x_{0}))\big)},
\end{align*}
and this estimate with \eqref{test limit} yields
\[\lim_{\epsilon\to0}I_{3}^{\epsilon}=\int_{t_{0}-T_{2}}^{\tau}\int_{B_{R}}f(x,t)\phi(x,t)\dx\dt.\]
\textbf{Estimate of $I^{\epsilon}_{4}$.} Since $u$ and ${\phi} \in C(J;L^{2}(B_{R}))$, we deduce
\begin{align*}
    \lim_{\epsilon\to0}I_{4}^{\epsilon}=-\int_{B_{R}(x_{0})}{u(x,t)}\phi(x,t)\dx\Bigg\rvert_{t=t_{0}-T_{2}}^{t=\tau}.
\end{align*}
We combine all the estimates of $I_{1}^{\epsilon},I_{2}^{\epsilon},I_{3}^{\epsilon}$, and $I_{4}^{\epsilon}$ 
to see the following
\begin{align*}
&\int_{B_{R}(x_{0})}\frac{w_{+}^{2}(x,t)}{2}\psi^{2}(x)\eta^{2}(t)\dx\Bigg\rvert_{t=t_{0}-T_{2}}^{t=\tau}\\
&+\int_{t_{0}-T_{2}}^{\tau}\int_{\mathbb{R}^{n}}\int_{\mathbb{R}^{n}}\Phi(u(x,t)-u(y,t))(\phi(x,t)-\phi(y,t))\frac{A(x,y,t)}{|x-y|^{n+2s}}\dx\dy\dt\\
&\qquad\leq\int_{t_{0}-T_{2}}^{\tau}\int_{B_{R}(x_{0})}f(x,t)\phi(x,t)\dx\dt
+\int_{t_{0}-T_{2}}^{\tau}\int_{B_{R}(x_{0})}w_{+}^{2}(x,t)\psi^{2}(x)\eta(t)\eta'(t)\dt\dx.
\end{align*}
From the Caccioppoli-type estimate as in \cite[Lemma 3.3]{DZZ}, we further investigate the second term in the left hand-side so that we have
\begin{align*}
    J_{0}&\coloneqq\int_{t_{0}-T_{2}}^{\tau}\int_{B_{R}(x_{0})}\int_{B_{R}(x_{0})}\frac{|w_{+}(x,t)\psi(x)-w_{+}(y,t)\psi(y)|^{2}}{|x-y|^{n+2s}}\eta^{2}(t)\dx\dy\dt\\
    &\quad+\int_{B_{R}(x_{0})}\frac{(w_{+}(x,t)\psi(x)\eta(t))^{2}}{2}\dx\Bigg\rvert_{t=t_{0}-T_{2}}^{t=\tau}\\
    &\leq c\int_{t_{0}-T_{2}}^{\tau}\int_{B_{R}(x_{0})}\int_{B_{R}(x_{0})}\frac{\max\{w_{+}(x,t),w_{+}(y,t)\}^{2}|\psi(x)-\psi(y)|^{2}\eta^{2}(t)}{|x-y|^{n+2s}}\dx\dy\dt\\
    &\quad+c\int_{t_{0}-T_{2}}^{\tau}\int_{\mathbb{R}^{n}\setminus B_{R}(x_{0})}\int_{B_{R}(x_{0})}\frac{w_{+}(y,t)}{|x-y|^{n+2s}}\phi(x,t)\dx\dy\dt\\
    &\quad+c\int_{t_{0}-T_{2}}^{\tau}\int_{B_{R}(x_{0})}f(x,t)\phi(x,t)\dx\dt\\
    &\quad+c\int_{t_{0}-T_{2}}^{\tau}\int_{B_{R}(x_{0})}w_{+}^{2}(x,t)\psi^{2}(x)\eta(t)\eta'(t)\dx\dt=J_{1}+J_{2}+J_{3}+J_{{4}}.
\end{align*}
Now we estimate $J_{i}$ for each $i=0,1,2,3$ and 4. Note that $J_{0},J_{1},J_{2}$ and $J_{{4}}$ can be estimated with the help of \eqref{time test function}, \eqref{space test function} and \eqref{tail point} as follows:

\noindent
\textbf{Estimate of $J_{0}$.}
\begin{align*}
J_{0}\geq\int_{t_{0}-T_{1}}^{\tau}\int_{B_{\rho}(x_{0})}\int_{B_{\rho}(x_{0})}\frac{|w_{+}(x,t)-w_{+}(y,t)|^{2}}{|x-y|^{n+2s}}\dx\dy\dt+
    \esup_{t\in[t_{0}-T_{1},\tau]}\int_{B_{\rho}(x_{0})}\frac{w_{+}^{2}(x,t)}{2}\dx
\end{align*}
\textbf{Estimate of $J_{1}$.} 
\begin{align*}
    J_{1}
    \leq c\frac{R^{2(1-s)}}{(R-\rho)^{2}}\int_{t_{0}-T_{2}}^{\tau}\int_{B_{R}(x_{0})}w_{+}^{2}(x,t)\dx\dt.
\end{align*}
\textbf{Estimate of $J_{2}$.}
\begin{align*}
    J_{2}
    \leq \frac{R^{n+2s}}{(R-\rho)^{n+2s}}\left(\esup_{t\in[t_{0}-T_{2},t_{0}]}\int_{\mathbb{R}\setminus B_{\rho}(x_{0})}\frac{w_{+}(y,t)}{|x_{0}-y|^{n+2s}}\dy\right)\|w_{+}\|_{L^{1}(Q_{R,T_{2}}(z_{0}))}.
\end{align*}
\textbf{Estimate of $J_{4}$.}
\begin{align*}
    J_{4}
    \leq \frac{2}{T_{2}-\tau}\int_{t_{0}-T_{2}}^{\tau}\int_{B_{R}(x_{0})}w_{+}^{2}(x,t)\dx\dt.
\end{align*}
In the case of $J_{3}$, {applying H\"older's inequality}, we have
\begin{align*}
    J_{3}&\leq\int_{t_{0}-T_{2}}^{\tau}\int_{B_{R}(x_{0})}|f(x,t)|w_{+}(x,t)\psi(x)\eta(t)\dx\dt\\
    &\leq \|f\|_{L^{q,r}(B_{R}(x_{0})\times(t_{0}-T_{2},\tau))}\left\|\Chi_{\{ u\ge k \}}\right\|_{L^{\tilde{q},\tilde{r}}(B_{R}(x_{0})\times(t_{0}-T_{2},\tau))}\|w_{+}\psi\eta\|_{L^{\hat{q},\hat{r}}(B_{R}(x_{0})\times(t_{0}-T_{2},\tau))}\\
    &\leq kR^{-n\kappa}\left\|\Chi_{\{ u\ge k \}}\right\|_{L^{\tilde{q},\tilde{r}}(B_{R}(x_{0})\times(t_{0}-T_{2},\tau))}\left\|w_{+}\psi\eta\right\|_{L^{\hat{q},\hat{r}}(B_{R}(x_{0})\times(t_{0}-T_{2},\tau))},
\end{align*}
where $\tilde{q} = \frac{\hat{q}}{2\kappa+1}$ and $\tilde{r} = \frac{\hat{r}}{2\kappa+1}$. 
From \eqref{kappa condition} and \eqref{hatqhatr}, we obtain
\begin{align}\label{j3dist1}
\begin{split}
\left\|\Chi_{\{ u\ge k \}}\right\|_{L^{\tilde{q},\tilde{r}}(B_{R}(x_{0})\times(t_{0}-T_{2},\tau))}
&= \left( \int_{t_{0}-T_{2}}^{\tau}|E(t;x_{0},R,k)|^{\frac{\tilde{r}}{\tilde{q}}}\dt \right)^{\frac{1}{\tilde{r}}}\\
&= \left( \int_{t_{0}-T_{2}}^{\tau}|E(t;x_{0},R,k)|^{\frac{\hat{r}}{\hat{q}}}\dt \right)^{\frac{2\kappa+1}{\hat{r}}},
\end{split}
\end{align}
and
\begin{align}\label{j3dist2}
&\left( \int_{t_{0}-T_{2}}^{\tau}|E(t;x_{0},R,k)|^{\frac{\hat{r}}{\hat{q}}}\dt \right)^{\frac{2\kappa}{\hat{r}}}
\le c\left( \int_{t_{0}-T_{2}}^{\tau}R^{\frac{n\hat{r}}{\hat{q}}}\dt \right)^{\frac{2\kappa}{\hat{r}}}
\le cR^{\left( 2s+\frac{n\hat{r}}{\hat{q}} \right)\frac{2\kappa}{\hat{r}}}
= cR^{n\kappa}.
\end{align}
Then we use Cauchy's inequality, \eqref{j3dist1}, \eqref{j3dist2}, and Lemma \ref{embedSF} to get
\begin{align*}
    J_{3}&\leq \frac{k^{2}}{4\epsilon}R^{-2n\kappa}\left\|\Chi_{\{ u\ge k \}}\right\|_{L^{\tilde{q},\tilde{r}}(B_{R}(x_{0})\times(t_{0}-T_{2},\tau))}^{2} 
	+ \epsilon\|w_{+}\psi\eta\|^{2}_{L^{\hat{q},\hat{r}}(B_{R}(x_{0})\times(t_{0}-T_{2},\tau))}\\
    &\leq \frac{k^{2}}{4\epsilon}R^{-n\kappa}\left(\int_{t_{0}-T_{2}}^{\tau}|{E(t;x_{0},R,k)}|^{\frac{\hat{r}}{\hat{q}}}\dt\right)^{\frac{2(1+\kappa)}{\hat{r}}}+c\epsilon\left(J_{0}+R^{-2s}\int_{t_{0}-T_{2}}^{\tau}\int_{B_{R}(x_{0})}w_{+}^{2}(x,t)\dx\dt\right)\\
    &\leq ck^{2}R^{-n\kappa}\left(\int_{t_{0}-T_{2}}^{\tau}|{E(t;x_{0},R,k)}|^{\frac{\hat{r}}{\hat{q}}}\dt\right)^{\frac{2(1+\kappa)}{\hat{r}}}+\frac{1}{4}\left(J_{0}+R^{-2s}\int_{t_{0}-T_{2}}^{\tau}\int_{B_{R}(x_{0})}w_{+}^{2}(x,t)\dx\dt\right),
\end{align*}
by taking $\epsilon=\frac{1}{4c}$. Since $\tau\in\left[t_{0}-\left(\frac{T_{1}+T_{2}}{2}\right),t_{0}\right]$ is arbitrary, we can combine estimates $J_{0},J_{1},J_{2},J_{3}$ and $J_{4}$ to complete the proof. 

\end{proof}
Now using Lemma \ref{caccioppoli estimate}, we prove the local boundedness for a local weak subsolution to \eqref{eq1} with \eqref{forcing term}.

\textbf{Proof of Theorem \ref{LBlem}.}

\noindent
\textbf{Step 1: Normalization.} Define 
\begin{equation*}
    \Tilde{u}(x,t)=u\left(\rho_{0}x+x_{0},\rho_{0}^{2s}t+t_{0}\right),\quad (x,t)\in Q_{1},
\end{equation*}
\begin{equation*}
    \Tilde{A}(x,y,t)=A\left(\rho_{0}x+x_{0},\rho_{0}y+x_{0},\rho_{0}^{2s}t+t_{0}\right),\quad (x,y,t)\in\mathbb{R}^{2n}\times\mathbb{R},
\end{equation*}
and
\begin{equation*}
    \Tilde{f}(x,t)=\rho_{0}^{2s}f\left(\rho_{0}x+x_{0},\rho_{0}^{2s}t+t_{0}\right),\quad (x,t)\in Q_{1}.
\end{equation*}
Then \[\Tilde{u}\in L^{2}_{\loc}\left(-1,0;W^{s,2}_{\loc}(B_{1})\right)\cap L^{\infty}_{\loc}\left(-1,0;L^{1}_{2s}(\mathbb{R}^{n})\right)\cap C_{\loc}\left(-1,0;L_{\loc}^{2}(B_{1})\right)\]is a local weak subsolution to
\begin{equation*}
    \partial_{t}\Tilde{u}+\mathcal{L}^{\Phi}_{\Tilde{A}}\Tilde{u}=\Tilde{f} \text{ in } Q_{1}.
\end{equation*}
\noindent  
Now take $k>0$ such that
\begin{equation*}
k > \|\Tilde{f}\|_{L^{q,r}(Q_{1})}+\esup_{t\in[-1,0]}\int_{\mathbb{R}^{n}\setminus B_{\frac{1}{2}}}\frac{|\Tilde{u}(y,t)|}{|y|^{n+2s}}\dy.
\end{equation*}
From the Lemma \ref{caccioppoli estimate}, for any $\frac{1}{2}<\rho<R<1$ and $\frac{1}{2^{2s}}<T_{1}<T_{2}\leq R^{2s}$, we have
\begin{equation}
\label{energy of normal}
\begin{aligned}
    &\|(\Tilde{u}-k)_{+}\|^{2}_{V^{2}_{s}(Q_{\rho,T_{1}})}\\
    &\quad\leq c\left(\frac{1}{(R-\rho)^{2}}+\frac{1}{T_{2}-T_{1}}\right)\int_{-T_{2}}^{0}\int_{B_{R}}(\Tilde{u}-k)_{+}^{2}\dx\dt\\
    &\qquad+c\left(\frac{1}{(R-\rho)^{n+2s}}\esup_{t\in[-T_{2},0]}\int_{\mathbb{R}^{n}\setminus B_{\rho}}\frac{(\Tilde{u}-k)_{+}(y,t)}{|y|^{n+2s}}\dy\right)\times\|({\Tilde{u}-k)_{+}}\|_{L^{1}(Q_{R,T_{2}})}\\
    &\qquad+ck^{2}R^{-n\kappa}\left(\int_{-T_{2}}^{0}|{\Tilde{E}(t;0,R,k)}|^{\frac{\hat{r}}{\hat{q}}}\dt\right)^{\frac{2(1+\kappa)}{\hat{r}}},
\end{aligned}
\end{equation}
where ${\Tilde{E}}$ is the level set of $\Tilde{u}$ as in \eqref{def of level set}. We write the second term on the right hand-side as $I$ for simplicity. 
Then, using  Cauchy's inequality and H\"older's inequality, we have
\begin{align*}
    I&\leq \frac{c}{(R-\rho)^{n+2s}}\iint_{Q_{R,T_{2}}}k(\Tilde{u}-k)_{+}\dx\dt \\
    &\leq \frac{c}{(R-\rho)^{n+2s}}\left(\iint_{Q_{R,T_{2}}}k^{2}\Chi_{(\Tilde{u}(x,t)>k)}\dx\dt+\iint_{Q_{R,T_{2}}}(\Tilde{u}-k)_{+}^{2}\dx\dt\right)\\
    &\leq \frac{c}{(R-\rho)^{n+2s}}k^{2}\left(\int_{-T_{2}}^{0}|{\Tilde{E}(t;0,R,k)}|^{\frac{\hat{r}}{\hat{q}}}\dt\right)^{\frac{2(1+\kappa)}{\hat{r}}}+\frac{c}{(R-\rho)^{n+2s}}\iint_{Q_{R,T_{2}}}(\Tilde{u}-k)_{+}^{2}\dx\dt,
\end{align*}
With \eqref{energy of normal} and the estimate of I, we get
\begin{equation}
\label{sc1}
\begin{aligned}
    \|(\Tilde{u}-k)_{+}\|^{2}_{V^{2}_{s}(Q_{\rho,T_{1}})}&\leq c\Bigg(\left(\frac{1}{(R-\rho)^{n+2s}}+\frac{1}{T_{2}-T_{1}}\right)\|(\Tilde{u}-k)_{+}\|^{2}_{L^{2}(Q_{R,T_{2}})}\\
    &\quad\quad+\frac{c}{(R-\rho)^{n+2s}}k^{2}\left(\int_{-T_{2}}^{0}|{\Tilde{E}(t;0,R,k)}|^{\frac{\hat{r}}{\hat{q}}}\dt\right)^{\frac{2(1+\kappa)}{\hat{r}}}\Bigg).
\end{aligned}
\end{equation}
\textbf{Step 2: Iteration.} For any nonnegative integer $h$, we write
\begin{align*}
    \rho_{h}=\frac{1}{2}+\frac{1}{2^{h+2}},\quad \tau_{h}=\rho_{h}^{2s},\quad  \overline{\rho}_{h}=\frac{\rho_{h}+\rho_{h+1}}{2}\quad \text{and}\quad
     Q_{h}=Q_{\rho_{h},\tau_{h}}
\end{align*}
Moreover, take
\begin{align*}
    &k_{h}=N+N\left(1-\frac{1}{2^{h}}\right),\\
     &\zeta_{h}\in C_{c}^{\infty}(B_{\overline{\rho}_{h}}),\quad 0\leq\zeta_{h}\leq1,\quad \zeta_{h}\equiv1\text{ on},\quad \|D\zeta_{h}\|_{\infty}\leq2^{h+5},\\
    &y_{h}=N^{-2}\int_{-\tau_{h}}^{0}\int_{B_{\rho_{h}}}(\Tilde{u}-k_{h})_{+}^{2}\dx\dt,\quad z_{h}=\left(\int_{-\tau_{h}}^{0}|{\Tilde{E}(t;\rho_{h},k_{h})}|^{\frac{\hat{r}}{\hat{q}}}\dt\right)^{\frac{2}{\hat{r}}},\\
    &\Lambda_{h}=\int^{0}_{-\tau_{h+1}}{\Tilde{E}(t;0,\overline{\rho}_{h},k_{h+1})}\dt,
\end{align*}
{where $N>\|\Tilde{f}\|_{L^{q,r}(Q_{1})}+\esup_{t\in[-1,0]}\int_{\mathbb{R}^{n}\setminus B_{\frac{1}{2}}}\frac{|\Tilde{u}(y,t)|}{|y|^{n+2s}}\dy$ will be determined later in \eqref{N condition}.}
Note that for any nonnegative integer $h$,
\begin{equation}
\label{Lambdah}
\begin{aligned}
&\tau_{h}-\tau_{h+1}=2s\int_{\rho_{h+1}}^{\rho_{h}}t^{2s-1}\dt\geq s(\rho_{h}-\rho_{h+1}),\\
&\Lambda_{h}\leq(k_{h+1}-k_{h})^{-2}N^{2}y_{h}\leq2^{2(h+4)}y_{h}.
\end{aligned}
\end{equation}
Now we obtain an iterative inequality of $y_{h}$ and $z_{h}$. Using H\"older's inequality and Lemma \ref{embedSF}, we have
\begin{align*}
    y_{h+1}&\leq N^{-2}\int_{-\tau_{h+1}}^{0}\int_{B_{\overline{\rho}_{h}}}(\Tilde{u}-k_{h+1})_{+}^{2}\zeta_{h}^{2}\dx\dt\\
    &\leq N^{-2}\|(\Tilde{u}-k_{h+1})_{+}\zeta_{h}\|^{2}_{L^{2}(Q_{\overline{\rho}_{h},\tau_{h+1}})}\\
    &\leq N^{-2}\Lambda_{h}^{\frac{2s}{n+2s}}\|(\Tilde{u}-k_{h+1})_{+}\zeta_{h}\|^{2}_{L^{2\left(1+\frac{2s}{n}\right)}(Q_{\overline{\rho}_{h},\tau_{h+1}})}\\
    &\leq cN^{-2}\Lambda_{h}^{\frac{2s}{n+2s}}\Big(\|(\Tilde{u}-k_{h+1})_{+}\zeta_{h}\|^{2}_{V^{2}_{s}\left(Q_{\overline{\rho}_{h},\tau_{h+1}}\right)}
    +\overline{\rho}_{h}^{-2s}\|(\Tilde{u}-k_{h+1})_{+}\zeta_{h}\|^{2}_{L^{2}(Q_{\overline{\rho}_{h},\tau_{h+1}})}\Big)\\
    &\leq cN^{-2}\Lambda_{h}^{\delta}\left(\|(\Tilde{u}-k_{h+1})_{+}\|^{2}_{V^{2}_{s}\left(Q_{\overline{\rho}_{h},\tau_{h+1}}\right)}+N^{2}y_{h}\right),
\end{align*}
where $\delta=\frac{2s}{n+2s}$.
By applying $\rho=\overline{\rho}_{h}$, $R=\rho_{h}$, $T_{1}=\tau_{h+1}$, $T_{2}=\tau_{h}$ and $k=k_{h+1}$ to \eqref{sc1} and using \eqref{Lambdah} and Lemma \ref{embedSF}, we observe that the followings
\begin{equation}
\label{Iyh}
\begin{aligned}
    y_{h+1}&\leq cN^{-2}\left(2^{2(h+4)}y_{h}\right)^{\delta}\left(\left(2^{h(2s+n)}+2^{h}\right)\|(\Tilde{u}-k_{h+1})_{+}\|^{2}_{L^{2}(Q_{\rho_{h},\tau_{h}})}+2^{h(n+2s)}k_{h+1}^{2}z_{h}^{1+\kappa}\right)\\
    &\quad+c2^{2(h+4)}y_{h}^{1+\delta}\\
    &\leq cN^{-2}\left(2^{2(h+4)}y_{h}\right)^{\delta}\left(\left(2^{h(n+2s)}+2^{h}\right)y_{h}N^{2}+2^{h(n+2s)}k_{h+1}^{2}z_{h}^{1+\kappa}\right)+c2^{2(h+4)}y_{h}^{1+\delta}\\
    &\leq c\left(\left(2^{h(n+4)}+2^{h+2h}\right)y_{h}^{1+\delta}+2^{h(n+4)}z_{h}^{1+\kappa}y_{h}^{\delta}\right),
\end{aligned}
\end{equation}
and
\begin{equation}
\label{Izh}
\begin{aligned}
    (k_{h+1}-k_{h})^{2}z_{h+1}&=(k_{h+1}-k_{h})^{2}\left(\int_{-\tau_{h+1}}^{0}|{\Tilde{E}(t;0,\rho_{h+1},k_{h+1})|}^{\frac{\hat{r}}{\hat{q}}}\dt\right)^{\frac{2}{\hat{r}}}\\
    &\leq c\|(\Tilde{u}-k_{h})_{+}\|^{2}_{L^{\hat{q},\hat{r}}(Q_{\rho_{h+1},\tau_{h+1}})}\\
    &\leq c\|(\Tilde{u}-k_{h})_{+}\zeta_{h}\|^{2}_{L^{\hat{q},\hat{r}}(Q_{\overline{\rho}_{h},\tau_{h+1}})}\\
    &\leq c\left(\|(\Tilde{u}-k_{h})_{+}\zeta_{h}\|^{2}_{V^{2}_{s}(Q_{\overline{\rho}_{h},\tau_{h+1}})}+\|(\Tilde{u}-k_{h})_{+}\zeta_{h}\|^{2}_{L^{2}(Q_{\overline{\rho}_{h},\tau_{h+1}})}\right)\\
    &\leq c\left(\left(2^{h(n+2)}+2^{h}\right)\|(\Tilde{u}-k_{h})_{+}\|{^{2}}_{L^{2}(Q_{\overline{\rho}_{h},\tau_{h+1}})}+2^{h(n+2s)}k_{h}^{2}z_{h}^{1+\kappa}\right)\\
    &\leq c2^{h(n+2)}N^{2}y_{h}+c2^{h(n+2)}N^{2}z_{h}^{1+\kappa}.
\end{aligned}
\end{equation}
From \eqref{Iyh} and \eqref{Izh}, we observe that 
\begin{align*}
    &y_{h+1}\leq c\left(2^{n+4}\right)^{h}(y_{h}^{1+\delta}+z^{1+\kappa}_{h}y_{h}^{\delta}),\\
    &z_{h+1}\leq c\left(2^{n+4}\right)^{h}(y_{h}+z^{1+\kappa}_{h}).
\end{align*}
{In addition, in light of \eqref{Iyh} and \eqref{Izh} with $k_{0}=N$, $k_{-1}=\|\Tilde{f}\|_{L^{q,r}(Q_{1})}+\esup_{t\in[-1,0]}\int_{\mathbb{R}^{n}\setminus B_{\frac{1}{2}}}\frac{|\Tilde{u}(y,t)|}{|y|^{n+2s}}\dy$, $\rho_{-1}=1$ and $\tau_{-1}=1$, we have  }
\begin{align*}
    &y_{0}\leq \frac{1}{N^{2}}\iint_{Q_{1}}(\Tilde{u}-N)_{+}^{2}
    \leq\frac{1}{N^{2}}\|\Tilde{u}\|^{2}_{L^{2}(Q_{1})},\\
    &z_{0}\leq \frac{c}{(N-k_{-1})^{2}}\left(\|(\Tilde{u}-k_{-1})_{+}\|{^{2}}_{L^{2}(Q_{1})}+k_{-1}^{2}\right).
\end{align*}
 Choose a sufficiently large $N$ such that
\begin{equation}
    \label{N condition}
    N\geq c^{\frac{1+\kappa}{\sigma}}\left(\|\Tilde{u}\|_{L^{2}(Q_{1})}+\|\Tilde{f}\|_{L^{q,r}(Q_{1})}+\esup_{t\in[-1,0]}\int_{\mathbb{R}^{n}\setminus B_{\frac{1}{2}}}\frac{|\Tilde{u}(y,t)|}{|y|^{n+2s}}\dy\right),\quad \sigma=\min\{\delta,\kappa\},
\end{equation}
which implies
\begin{align*}
    &y_{0}\leq (2c)^{-\frac{1+\kappa}{\sigma}}\times \left(2^{n+4}\right)^{-\frac{1+\kappa}{\sigma^{2}}}\\
    &z_{0}\leq(2c)^{-\frac{1+\kappa}{\sigma}}\times\left(2^{n+4}\right)^{-\frac{1+\kappa}{\sigma^{2}}}.
\end{align*}
Applying Lemma \ref{TLI}, we have
\begin{align*}
    \sup_{Q_{1/2}}|\Tilde{u}(x,t)|&\leq c\left(\|\Tilde{u}\|_{L^{2}(Q_{1})}+\|\Tilde{f}\|_{L^{q,r}(Q_{1})}+\esup_{t\in[-1,0]}\int_{\mathbb{R}^{n}\setminus B_{\frac{1}{2}}}\frac{|\Tilde{u}(y,t)|}{|y|^{n+2s}}\dy\right),
\end{align*}
By scaling back, we see that
\begin{equation*}
\begin{aligned}
    \sup_{z\in Q_{\rho_{0}/2}(z_{0})}|u(z)|&\leq c\Bigg(\left(\dashiint_{Q_{\rho_{0}}(z_{0})}u^{2}(x,t)\dx\dt\right)^{\frac{1}{2}}+\ITail\left(u,z_{0},\rho_{0}/2,\rho_{0}^{2s}\right)\\
    &\qquad+\rho_{0}^{2s-\left(\frac{n}{q}+\frac{2s}{r}\right)}||f||_{L^{q,r}(Q_{\rho_{0}}(z_{0}))}\Bigg).
\end{aligned}
\end{equation*}
\qed

\medskip
We next want to assert the H\"older regularity of a local weak solution to \eqref{eq1} with \eqref{forcing term} by comparing the solution to the homogeneous equation \eqref{Reference PDE} below.
More precisely, we assert that 
there is a constant $\Tilde{\beta} \equiv \Tilde{\beta}(n,s,q,r,\beta)$ such that $u\in C^{\Tilde{\beta},\frac{\Tilde{\beta}}{2s}}(Q_{R_{0}}(z_{0}))$  for every $Q_{2R_{0}}(z_{0})\Subset\Omega_{T}$. 

\noindent
Now, we fix any $\rho>0$ and $R>0$ so that
\begin{equation}
\label{condition for nonho}
0<\rho<R<\min\left\{1,\left(2^{2s}-1\right)^{\frac{1}{2s}}\frac{R_{0}}{4},\frac{R_{0}}{4}\right\}
\end{equation}
and choose a point $\Tilde{z}\in Q_{R_{0}}(z_{0})$.
Note that $Q_{4R}(\Tilde{z})\subset Q_{2R_{0}}(z_{0})\Subset\Omega_{T}$.
Throughout this section, we set \begin{equation*}
    \gamma=2n\kappa,
\end{equation*} which is the positive number from \eqref{forcing term}.
\begin{lem}
\label{Comparison estimate}
Let $u$ be a local weak solution to \eqref{eq1} with \eqref{forcing term} and \eqref{condition for nonho}.
Suppose that
\begin{equation*}
\begin{aligned}
v\in L^{2}\left(\Tilde{t}-(3R)^{2s},\Tilde{t};W^{s,2}(B_{3R}(\Tilde{x}))\right)&\cap L^{\infty}\left(\Tilde{t}-(3R)^{2s},\Tilde{t};L^{1}_{2s}(\mathbb{R}^{n})\right)\\
&\cap C\left(\left[\Tilde{t}-(3R)^{2s},\Tilde{t}\right];L^{2}(B_{3R}(\Tilde{x}))\right)
\end{aligned}
\end{equation*} is the local weak solution to 
\begin{align}
\label{Reference PDE}
   \left\{
\begin{alignedat}{3}
\partial_{t}v+\mathcal{L}^{\Phi}_{A}v &=0 &&\qquad \mbox{in  $Q_{3R}(\Tilde{z})$} \\
v&=u&&\qquad  \mbox{on  $\partial_{P}Q_{3R}(\Tilde{z})\cup\left(\left(\mathbb{R}^{n}\setminus B_{3R}(\Tilde{x})\right)\times\left[\Tilde{t}-(3R)^{2s},\Tilde{t}\right]\right)$}.
\end{alignedat} \right. 
\end{align}
Then we have the estimates
\begin{equation*}
    \int_{\Tilde{t}-(3R)^{2s}}^{\Tilde{t}}[(u-v)(\cdot,t)]^{2}_{W^{s,2}(\mathbb{R}^{n})}\dt+\sup_{t\in[\Tilde{t}-(3R)^{2s},\Tilde{t}]}\int_{B_{3R}(\Tilde{x})}(u-v)^{2}(x,t)\dx\leq cR^{\gamma+n}\|f\|^{2}_{L^{q,r}(Q_{3R}(\Tilde{z}))}
\end{equation*}
and 
\begin{equation*}
    \dashiint_{Q_{3R}(\Tilde{z})}(u-v)^{2}\dx\dt\leq cR^{\gamma}\|f\|^{2}_{L^{q,r}(Q_{3R}(\Tilde{z}))},
\end{equation*}
for some constant  $c \equiv c(n,s,q,r,\lambda)$.

\end{lem}
\begin{rmk}
The existence of $v$ in \eqref{Reference PDE} follows from the Appendix \ref{Appendix}.
\end{rmk}

\begin{proof}
Let $w=u-v$ to find that 
\begin{equation*}
    w\in L^{2}\left(\Tilde{t}-(3R)^{2s},\Tilde{t};W^{s,2}_{0}(B_{3R}(\Tilde{x}))\right)\cap C\left(\left[\Tilde{t}-(3R)^{2s},\Tilde{t}\right];L^{2}(B_{3R}(\Tilde{x}))\right)
\end{equation*}
and
\begin{equation*}
    \partial_{t}w+\mathcal{L}_{A}^{\Phi}u-\mathcal{L}_{A}^{\Phi}v=f\text{ in }Q_{3R}(\Tilde{z}).
\end{equation*}
By an approximation argument with the mollification in time, we find that
\begin{align*}
    &\int_{B_{3R}(\Tilde{x})}\frac{w^{2}(x,t_{1})}{2}\dx+\int_{\Tilde{t}-(3R)^{2s}}^{t_{1}}\int_{\mathbb{R}^{n}}\int_{\mathbb{R}^{n}}A(x,y,t)\frac{\Phi(u(x,t)-u(y,t))}{|x-y|^{n+2s}}(w(x,t)-w(y,t))\dx\dy\dt\\
    &-\int_{\Tilde{t}-(3R)^{2s}}^{t_{1}}\int_{\mathbb{R}^{n}}\int_{\mathbb{R}^{n}}A(x,y,t)\frac{\Phi(v(x,t)-v(y,t))}{|x-y|^{n+2s}}(w(x,t)-w(y,t))\dx\dy\dt\\
    &\quad=\int_{\Tilde{t}-(3R)^{2s}}^{t_{1}}\int_{B_{3R}(\Tilde{x})}fw\dx\dt,
\end{align*}
where $t_{1}\in\left[\Tilde{t}-(3R)^{2s},\Tilde{t}\right]$. we have the estimate 
\begin{align*}
    \int_{\Tilde{t}-(3R)^{2s}}^{t_{1}}\int_{\mathbb{R}^{n}}\int_{\mathbb{R}^{n}}A(x,y,t)\frac{\Phi(u(x,t)-u(y,t))-\Phi(v(x,t)-v(y,t))}{|x-y|^{n+2s}}(w(x,t)-w(y,t))\dx\dy\dt\\
    \geq \lambda^{-2}\int_{\Tilde{t}-(3R)^{2s}}^{t_{1}}\int_{\mathbb{R}^{n}}\int_{\mathbb{R}^{n}}\frac{|w(x,t)-w(y,t)|^{2}}{|x-y|^{n+2s}}\dx\dy\dt,
\end{align*}
which implies 
\begin{align*}
    &\sup_{t\in\left[\Tilde{t}-(3R)^{2s},\Tilde{t}\right]}\int_{B_{3R}(\Tilde{x})}\frac{w^{2}(x,t)}{2}\dx+\lambda^{-2}\int_{\Tilde{t}-(3R)^{2s}}^{\Tilde{t}}\int_{\mathbb{R}^{n}}\int_{\mathbb{R}^{n}}\frac{|w(x,t)-w(y,t)|^{2}}{|x-y|^{n+2s}}\dx\dy\dt\\
    &\quad\leq \int_{Q_{3R}(\Tilde{z})}fw\dx\dt.
\end{align*}
Using H\"older's inequality, Lemma \ref{embedSF}, and Cauchy's inequality, we estimate the last term as follows:
\begin{align*}
    &\int_{Q_{3R}(\Tilde{z})}fw\dx\dt\\
    &\leq cR^{\frac{n+\gamma}{2}}\|f\|_{L^{q,r}(Q_{3R}(\Tilde{z}))}\Bigg[\left(\int_{\Tilde{t}-(3R)^{2s}}^{\Tilde{t}}[w(\cdot,t)]^{2}_{W^{s,2}(\mathbb{R}^{n})}\dt+\sup_{t\in\left[\Tilde{t}-(3R)^{2s},\Tilde{t}\right]}\int_{B_{3R}(\Tilde{x})}w^{2}(x,t)\dx\right)^{\frac{1}{2}}\Bigg]\\
    &\leq cR^{n+\gamma}\|f\|^{2}_{L^{q,r}(Q_{3R}(\Tilde{z}))}\\
    &\qquad+\frac{\lambda^{-2}}{4}\int_{\Tilde{t}-(3R)^{2s}}^{\Tilde{t}}\int_{\mathbb{R}^{n}}\int_{\mathbb{R}^{n}}\frac{|w(x,t)-w(y,t)|^{2}}{|x-y|^{n+2s}}\dx\dy\dt+\frac{\lambda^{-2}}{4}\sup_{t\in\left[\Tilde{t}-(3R)^{2s},\Tilde{t}\right]}\int_{B_{3R}(\Tilde{x})}w^{2}(x,t)\dx.
\end{align*}Combining all the estimates together with Lemma \ref{embedSF}, we get 
\begin{equation}
\label{u-v sup l2 estimate}
    \int_{\Tilde{t}-(3R)^{2s}}^{\Tilde{t}}[(u-v)(\cdot,t)]_{W^{s,2}(\mathbb{R}^{n})}\dt+\sup_{t\in\left[\Tilde{t}-(3R)^{2s},\Tilde{t}\right]}\int_{B_{3R}(\Tilde{x})}(u-v)^{2}(x,t)\dx\leq cR^{n+\gamma}\|f\|^{2}_{L^{q,r}(Q_{3R}(\Tilde{z}))}
\end{equation}
and 
\begin{equation*}
    \dashiint_{Q_{3R}(\Tilde{z})}(u-v)^{2}\dx\dt\leq cR^{\gamma}\|f\|^{2}_{L^{q,r}(Q_{3R}(\Tilde{z}))}.
\end{equation*}
\end{proof}
\noindent
Using the local boundedness of $u$, Lemma \ref{LBlem}, and Lemma \ref{Comparison estimate}, we have the following estimate.
\begin{lem}(Decay transfer)
\label{Decay transfer lemma}
Under the same conditions and  conclusion as in Lemma \ref{Comparison estimate}, there holds
\begin{align*}
    &\dashiint_{Q_{\rho}}|u-\overline{u}_{Q_{\rho}}|^{2}\dx\dt\\
    &\quad\leq c\left(\frac{R}{\rho}\right)^{n+2s}R^{\gamma}\|f\|^{2}_{L^{q,r}(Q_{4R}(\Tilde{z}))}\\
    &\qquad+c\left(\frac{\rho}{R}\right)^{2\beta }\left(R^{\gamma}\|f\|^{2}_{L^{q,r}(Q_{4R}(\Tilde{z}))}+\|u\|^{2}_{L^{\infty}(Q_{4R}(\Tilde{z}))}+\ITails(u,\Tilde{z},4R)\right),
\end{align*}
for some constant $c \equiv c(n,s,q,r,\lambda)$.
\end{lem}
\begin{proof}
Using triangle inequality, we have
\begin{align*}
    \dashiint_{Q_{\rho}}|u-\overline{u}_{Q_{\rho}}|^{2}\dx\dt&\leq {c}\bigg[\dashiint_{Q_{\rho}}|u-v|^{2}\dx\dt+\dashiint_{Q_{\rho}}|v-\overline{v}_{Q_{\rho}}|^{2}\dx\dt\\
    &\qquad+\dashiint_{Q_{\rho}}|\overline{u}_{Q_{\rho}}-\overline{v}_{Q_{\rho}}|^{2}\dx\dt\bigg]\\
    &\leq c\dashiint_{Q_{\rho}}|u-v|^{2}\dx\dt+c\dashiint_{Q_{\rho}}|v-\overline{v}_{Q_{\rho}}|^{2}\dx\dt\\
    &\eqqcolon {I_{1}+I_{2}}.
\end{align*}
Using Lemma \ref{Comparison estimate}, we get 
\begin{align*}
    I_{1} \leq c\left(\frac{R}{\rho}\right)^{n+2s}R^{\gamma}\|f\|^{2}_{L^{q,r}(Q_{4R})}.
\end{align*}
From Lemma \ref{HomoHolder}, we deduce
\begin{align*}
    I_{2} \leq c\left(\frac{\rho}{R}\right)^{2\beta}\left(\dashiint_{Q_{2R}}|v|^{2}\dx\dt+\ITails(v;\Tilde{z},R,(2R)^{2s})\right).
\end{align*}
{Set $I_{2,1} \coloneqq \dashiint_{Q_{2R}}|v|^{2}\dx\dt$ and $I_{2,2} \coloneqq \ITails(v;\Tilde{z},R,(2R)^{2s})$}.
Then we have
\begin{align*}
    I_{2,1} &\leq 2\left(\dashiint_{Q_{2R}}|u|^{2}\dx\dt\right)+ 2\left(\dashiint_{Q_{2R}}|u-v|^{2}\dx\dt\right)\\
    &\leq 2\left(\dashiint_{Q_{4R}}|u|^{2}\dx\dt\right)+cR^{\gamma}\|f\|_{L^{q,r}(Q_{4R})}^{2}.
\end{align*} Since $u-v\equiv0$ in $\left(\left(\mathbb{R}^{n}\setminus B_{3R}(\Tilde{x})\right)\times\left[\Tilde{t}-(3R)^{2s},\Tilde{t}\right]\right)$, it follows from \eqref{u-v sup l2 estimate} that 
\begin{align*}
    I_{2,2}&\leq2\ITails(u-v;\Tilde{z},R,(2R)^{2s})+2\ITails(u;\Tilde{z},R,(2R)^{2s})\\
    &\leq c\Bigg(\Bigg(\esup_{t\in\left[\Tilde{t}-(2R)^{2s},\Tilde{t}\right]}\dashint_{B_{3R}(\Tilde{x})}|(u-v)(y,t)|\dy\Bigg)^{2}+\|u\|^{2}_{L^{\infty}(Q_{4R})}+\ITails(u;\Tilde{z},4R)\Bigg)\\
    &\leq c\left(R^{\gamma}\|f\|_{L^{q,r}(Q_{4R})}^{2}+\|u\|^{2}_{L^{\infty}(Q_{4R})}+\ITails(u;\Tilde{z},4R)\right).
\end{align*}
We combine all the estimates of $I_{2,1}$ and $I_{2.2}$ to finish the proof.
\end{proof}
\noindent
Now we are ready to prove the  H\"older regularity. 
\begin{lem}
\label{Holder for data}
Under the same conditions and conclusion as in Lemma \ref{Decay transfer lemma}, we have  $u\in C^{\Tilde{\beta},\frac{\Tilde{\beta}}{2s}}(Q_{R_{0}}(z_{0}))$ for some $\Tilde{\beta}=\Tilde{\beta}(n,s,q,r,\lambda)$
with the estimate 
\begin{equation*}
\begin{aligned}
    [u]_{C^{\Tilde{\beta},\frac{\Tilde{\beta}}{2s}}(Q_{R_{0}}(z_{0}))}&\leq c\Big(\|f\|_{L^{q,r}(Q_{2R_{0}}(z_{0}))}+\|u\|_{L^{\infty}(Q_{2R_{0}}(z_{0}))}+\ITail(u;z_{0},2R_{0})\Big),
\end{aligned}
\end{equation*}
for some constant $c \equiv c(n,s,q,r,\lambda).$
\end{lem}
\begin{proof}
From Lemma \ref{Decay transfer lemma}, we find that for every $\Tilde{z}\in Q_{R_{0}}(z_{0})$,
\begin{align*}
    &\dashiint_{Q_{\rho}(\Tilde{z})}|u-\overline{u}_{Q_{\rho}(\Tilde{z})}|^{2}\dx\dt\\
    &\quad\leq c\left(\frac{R}{\rho}\right)^{n+2s}R^{\gamma}\|f\|^{2}_{L^{q,r}(Q_{4R}(\Tilde{z}))}\\
    &\qquad+c\left(\frac{\rho}{R}\right)^{2\beta }\left(R^{\gamma}\|f\|^{2}_{L^{q,r}(Q_{4R}(\Tilde{z}))}+\|u\|^{2}_{L^{\infty}(Q_{4R}(\Tilde{z}))}+\ITails(u,\Tilde{z},4R)\right)\\
    &\quad\leq c\left(\frac{R}{\rho}\right)^{n+2s}R^{\gamma}\|f\|^{2}_{L^{q,r}(Q_{2R_{0}}(z_{0}))}\\
    &\qquad+c\left(\frac{\rho}{R}\right)^{2\beta }\left((R_{0})^{\gamma}\|f\|^{2}_{L^{q,r}(Q_{2R_{0}}(z_{0}))}+\|u\|^{2}_{L^{\infty}(Q_{2R_{0}}(z_{0}))}+\ITails\left(u,\Tilde{z},4R\right)\right).
\end{align*}
We notice that
\begin{align*}
    \ITail\left(u,\Tilde{z},4R\right)&\leq (4R)^{2s}\int_{B_{2R_{0}}(x_{0})\setminus B_{4R}(\Tilde{x})}\frac{\|u\|_{L^{\infty}(Q_{2R_{0}}(z_{0}))}}{|y-\Tilde{x}|^{n+2s}}\dy\\
    &\quad+(4R)^{2s}\esup_{t\in[t_{0}-(2R_{0})^{2s},t_{0}]}\int_{\mathbb{R}^{n}\setminus B_{2R_{0}}(x_{0})}\frac{|u(y,t)|}{|y-x_{0}|^{n+2s}}\dy\\
    &\leq c\|u\|_{L^{\infty}(Q_{2R_{0}}(z_{0}))}+c\ITail\left(u,z_{0},2R_{0}\right),
\end{align*}
where we have used the fact that
\[|y-\Tilde{x}|\geq|y-x_{0}|-|\Tilde{x}-x_{0}|\geq\frac{|y-x_{0}|}{2} \text{ for any }y\in\mathbb{R}^{n}\setminus B_{2R_{0}}(x_{0}).\] 
Thus 
\begin{align*}
    &\dashiint_{Q_{\rho}(\Tilde{z})}|u-\overline{u}_{Q_{\rho}(\Tilde{z})}|^{2}\dx\dt\\
    &\quad\leq c\left(\frac{R}{\rho}\right)^{n+2s}R^{\gamma}\|f\|^{2}_{L^{q,r}(Q_{2R_{0}}(z_{0}))}\\
    &\qquad+c\left(\frac{\rho}{R}\right)^{2\beta }\left((R_{0})^{\gamma}\|f\|^{2}_{L^{q,r}(Q_{2R_{0}}(z_{0}))}+\|u\|^{2}_{L^{\infty}(Q_{2R_{0}}(z_{0}))}+\ITails(u;z_{0},2R_{0})\right).
\end{align*}
Now take  $\rho=R^{\frac{n+2\beta+2s+\gamma}{n+2\beta+2s}}<R$ to discover that  
\begin{equation*}
\begin{aligned}
\frac{1}{\rho^{2\Tilde{\beta}}}\dashiint_{Q_{\rho}(\Tilde{z})}|u-\overline{u}_{Q_{\rho}(\Tilde{z})}|^{2}\dx\dt
&\leq c\Big(\|f\|^{2}_{L^{q,r}(Q_{2R_{0}}(z_{0}))}\\
    &\qquad+\|u\|^{2}_{L^{\infty}(Q_{2R_{0}}(z_{0}))}+\ITails(u;x_{0},2R_{0})\Big),
\end{aligned}
\end{equation*}
for $\Tilde{\beta}=\frac{\beta\gamma}{n+2s+2\beta+\gamma}$ where $\beta$ is as in Lemma \ref{HomoHolder}.
The conclusion follows from Lemma \ref{Campanato embedding for fractional}. 
\end{proof}


\section{Improved H\"older regularity for homogeneous equation with kernel coefficient invariant under the translation in only spatial direction}
\label{section4}
In the previous chapter, we have proved the H\"older regularity to \eqref{eq1} with \eqref{forcing term} for some exponent. For now, we assume the kernel coefficient 
\begin{equation}
\label{kernel coefficient condition}
\Tilde{A}\in\mathcal{L}_{1}(\lambda;B_{1}\times B_{1}\times(-1{,}0))\cap\mathcal{L}_{0}(\lambda),
\end{equation}to obtain the higher H\"older regularity.
We focus on a weak solution 
\[u\in L^{2}\big((-2^{2s},0];W^{s,2}(B_{2})\big)\cap L^{\infty}\big((-2^{2s},0];L^{1}_{2s}(\mathbb{R}^{n})\big)
    \cap C\big((-2^{2s},0];L^{2}(B_{2})\big)\] to the locally normalized problem
\begin{equation}
\label{normalized homo}
    \partial_{t}u+\mathcal{L}^{\Phi}_{\Tilde{A}}u=0 \text{ in } Q_{2}.
\end{equation}
Many of the results and computations in this section are based on those in \cite{BLS, BLS2}.

First, we start with the following discrete differentiation of the equation, which is an extension of the one described in  \cite[Lemma 3.3]{BLS} as well as a parabolic analogue of the elliptic one in \cite[Prosition 3.1.]{NHH}.

\begin{lem}
\label{LDd1}
Assume that $u$ is a local weak solution to \eqref{normalized homo} with \eqref{kernel coefficient condition}. Let $\psi(x)\in C_{c}^{\infty}(B_{R})$ be a nonnegative function and $\eta(t)\in C^{\infty}(\mathbb{R})$ a nonnegative function with 
\[\eta(t)=0 \text{ at }t\leq t_{1}\text{ and }\eta(t)=1 \text{ at }t\geq t_{2}\]
where $-1<t_{1}<t_{2}<0$ and $R<1$. Then for any locally Lipschitz function $g:\mathbb{R}\to\mathbb{R}$ and  $h\in\mathbb{R}^{n}$ such that $0<|h|<\frac{1}{4(1-R)}$, we have
\begin{align}\label{LDd1.est}
\begin{split}
    &\int_{t_{1}}^{t_{2}}\int_{B_{R}}\int_{B_{R}}\big(\Phi(u_{h}(x,t)-u_{h}(y,t))-\Phi(u(x,t)-u(y,t))\big)\\
    &\quad\times\big(g(u_{h}(x,t)-u(x,t))\psi(x)^{2}-g(u_{h}(y,t)-u(y,t))\psi(y)^{2}\big)\eta(t)\frac{A(x,y,t)}{|x-y|^{n+2s}}\dx\dy\dt\\
    &\qquad=-\int_{t_{1}}^{t_{2}}\int_{\mathbb{R}^{n}\setminus B_{R}}\int_{B_{R}}\frac{\big(A_{h}(x,y,t)\Phi(u_{h}(x,t)-u_{h}(y,t))-A(x,y,t)\Phi(u(x,t)-u(y,t))\big)}{|x-y|^{n+2s}}\\
    &\qquad\qquad\times\big(g(u_{h}(x,t)-u(x,t))\psi(x)^{2}\eta(t)\big)\dx\dy\dt\\
    &\quad\qquad-\int_{t_{1}}^{t_{2}}\int_{B_{R}}\int_{\mathbb{R}^{n}\setminus B_{R}}\frac{\big(A_{h}(x,y,t)\Phi(u_{h}(x,t)-u_{h}(y,t))-A(x,y,t)\Phi(u(x,t)-u(y,t))\big)}{|x-y|^{n+2s}}\\
    &\qquad\qquad\times\big(g(u_{h}(y,t)-u(y,t))\psi(y)^{2}\eta(t)\big)\dx\dy\dt\\
    &\quad\qquad-\int_{B_{R}}G(\delta_{h}u(x,t_{2}))\psi^{2}(x)\dx+\int_{t_{1}}^{t_{2}}\int_{B_{R}}G(\delta_{h}u(x,t))\psi^{2}(x)\eta'(t)\dx\dt,
\end{split}
\end{align}
where $G(t)=\int_{0}^{t}g(s)\ds$ for any $t\in\mathbb{R}$ and { $A_{h}(x,y,t)=A(x+h,y+h,t)$ for $(x,y,t)\in\mathbb{R}^{n}\times\mathbb{R}^{n}\times\mathbb{R}$.}
\end{lem}
\begin{proof}
For convenience, we denote the first three terms in the above equation by $J_{1}$, $J_{2}$, and $J_{3}$.
Take
\begin{align*}
    &h_{0}=\dist(\supp\psi,\partial B_{R})\quad \text{and}\quad \epsilon_{0}=\frac{1}{2}\min\{-t_{2},t_{1}+1,t_{2}-t_{1}\},
\end{align*}  
and for any $0<\epsilon<\epsilon_{0}$, define
\[ {\phi_{\epsilon}(x,t)\coloneqq \left(g(u_{h}^{\epsilon}-u^{\epsilon})\psi^{2}\eta_{\epsilon}\right)^{\epsilon}(x,t)=\left(g\left(\delta_{h}u^{\epsilon}\right)\psi^{2}\eta_{\epsilon}\right)^{\epsilon}(x,t),\quad (x,t)\in Q_{1}},\] where 
\[ \eta_{\epsilon}(t)=\eta\left(\frac{t_{2}-t_{1}}{t_{2}-t_{1}-\epsilon}\left(t-t_{2}+\frac{\epsilon}{2}\right)+t_{2}\right),\quad t\in\mathbb{R}.\]
Testing to \eqref{normalized homo} with $\phi_{\epsilon}$, we get
\begin{equation}
\label{Dd1}
\begin{aligned}
    &\int_{t_{1}}^{t_{2}}\int_{B_{R}}-u\partial_{t}\phi_{\epsilon}\dx\dt\\
    &\quad+\int_{t_{1}}^{t_{2}}\int_{\mathbb{R}^{n}}\int_{\mathbb{R}^{n}}\Phi(u(x,t)-u(y,t))(\phi_{\epsilon}(x,t)-\phi_{\epsilon}(y,t))\frac{A(x,y,t)}{|x-y|^{n+2s}}\dx\dy\dt\\
    &=-\int_{B_{R}}u(x,t)\phi_{\epsilon}(x,t)\dx\Bigg\rvert_{t=t_{1}}^{t=t_{2}}.
\end{aligned}
\end{equation}
On the other hand, we take a test function $\phi_{-h}(x,t)\coloneqq\phi(x-h,t)$ for $0<|h|<\frac{h_{0}}{4}$ and using a change of variables, we have
\begin{equation}
\label{Dd2}
\begin{aligned}
    &\int_{t_{1}}^{t_{2}}\int_{B_{R}}-u_{h}\partial_{t}\phi_{\epsilon}\dx\dt\\
    &\quad+\int_{t_{1}}^{t_{2}}\int_{\mathbb{R}^{n}}\int_{\mathbb{R}^{n}}\Phi(u_{h}(x,t)-u_{h}(y,t))(\phi_{\epsilon}(x,t)-\phi_{\epsilon}(y,t))\frac{A_{h}(x,y,t)}{|x-y|^{n+2s}}\dx\dy\dt\\
    &=-\int_{B_{R}}u_{h}(x,t)\phi_{\epsilon}(x,t)\dx\Bigg\rvert_{t=t_{1}}^{t=t_{2}}.
\end{aligned}
\end{equation} 
By subtracting \eqref{Dd2} from \eqref{Dd1}, we observe that
\begin{align*}
I^{\epsilon}&\coloneqq \int_{t_{1}}^{t_{2}}\int_{B_{R}}(-u_{h}+u)\partial_{t}\phi_{\epsilon}\dx\dt+\int_{B_{R}}(u_{h}(x,t)-u(x,t))\phi_{\epsilon}(x,t)\dx\Bigg\rvert_{t=t_{1}}^{t=t_{2}}\\
&=-\int_{t_{1}}^{t_{2}}\int_{B_{R}}\int_{B_{R}}\big(\Phi(u(x,t)-u(y,t))-\Phi(u_{h}(x,t)-u_{h}(y,t))\big)\\
&\hspace{3.1cm}\times(\phi_{\epsilon}(x,t)-\phi_{\epsilon}(y,t))\frac{A(x,y,t)}{|x-y|^{n+2s}}\dx\dy\dt\\
&\quad-\int_{t_{1}}^{t_{2}}\int_{\mathbb{R}^{n}\setminus B_{R}}\int_{B_{R}}\frac{\big(A(x,y,t)\Phi(u(x,t)-u(y,t))-A_{h}(x,y,t)\Phi(u_{h}(x,t)-u_{h}(y,t))\big)}{|x-y|^{n+2s}}\\
&\hspace{3.6cm}\times\phi_{\epsilon}(x,t)\dx\dy\dt\\
&\quad-\int_{t_{1}}^{t_{2}}\int_{B_{R}}\int_{\mathbb{R}^{n}\setminus B_{R}}\frac{\big(A(x,y,t)\Phi(u(x,t)-u(y,t))-A_{h}(x,y,t)\Phi(u_{h}(x,t)-u_{h}(y,t))\big)}{|x-y|^{n+2s}}\\
&\hspace{3.6cm}\times\phi_{\epsilon}(y,t)\dx\dy\dt\\
&\eqqcolon -J_{1}^{\epsilon}-J_{2}^{\epsilon}-J_{3}^{\epsilon}.
\end{align*}
By following as in \cite[Lemma 3.3]{BLS}, we find that
\begin{align*}
    \lim_{\epsilon\to0}I^{\epsilon}=\int_{B_{R}}G(\delta_{h}u(x,t_{1}))\psi^{2}(x)\dx-\int_{t_{1}}^{t_{2}}\int_{B_{R}}G(\delta_{h}u(x,t))\psi^{2}(x)\eta'(t)\dx\dt
\end{align*} 
We now use the local boundedness result, Lemma \ref{LBlem} which implies $u\in L^{\infty}([t_{1}-\epsilon_{0},t_{2}+\epsilon_{0}]\times B_{R+2h_{0}})$ and the locally Lipschitz regularity of $g$ to observe 
\[ \phi_{\epsilon}(x,t)\to g(u_{h}-u)\psi^{2}(x)\eta(t)\quad \text {in }L^{2}([t_{1}-\epsilon_{0},t_{2}+\epsilon_{0}];W^{s,2}(B_{R}))\]
and
\[ \phi_{\epsilon}(x,t)\to g(u_{h}-u)\psi^{2}(x)\eta(t) \quad\text{in }C([t_{1}-\epsilon_{0},t_{2}+\epsilon_{0}];L^{2}(B_{R})).\] 
Similar arguments performed on the estimate of $I_{2}^{\epsilon}$ in Lemma \ref{caccioppoli estimate} yield that
\begin{align*}
    \lim_{\epsilon\to0}J_{i}^{\epsilon}=J_{i},\quad \text{for each }i=1,2,3.
\end{align*}
Combining the each limits of $I^{\epsilon},J_{1}^{\epsilon},J_{2}^{\epsilon}$, and $J_{3}^{\epsilon}$, we conclude the proof.
\end{proof}
With the above Lemma \ref{LDd1}, we see the gaining of integrability and differentiability of a weak solution $u$ to \eqref{normalized homo} in the following.

\begin{lem}
\label{base for iteration} 
Let $u$ be a local weak solution to \eqref{normalized homo} with \eqref{kernel coefficient condition} satisfying
\begin{equation*}
    \|u\|_{L^{\infty}(Q_{1})}\leq1 \ \mbox{and} \ \ITail(u;0,1)\leq1.
\end{equation*}
  Assume that for some $p\geq 2$, $\theta\in\mathbb{R}$ with $0<\frac{1+\theta p}{p}<1$,
\[ \int_{t_{1}}^{t_{2}}\sup_{0<|h|<h_{0}}\left\|\frac{\delta^{2}_{h}u(\cdot,t)}{|h|^{\frac{1+\theta p}{p}}}\right\|_{L^{p}(B_{1})}^{p}\dt<\infty,\]
whenever $0<h_{0}\leq\frac{1}{10}$ and $-1<t_{1}<t_{2}\leq0$.
Then we have
\begin{flalign*}
    &\int_{t_1+\tau}^{t_{2}}\sup_{0<|h|<h_{0}}\left\|\frac{\delta^{2}_{h}u(\cdot,t)}{|h|^{\frac{1+2s+\theta p}{p+1}}}\right\|_{L^{p+1}(B_{R-4h_{0}})}^{p+1}\dt+\sup_{0<|h|<h_{0}}\left\|\frac{\delta_{h}u(\cdot,t_{2})}{|h|^{\frac{1+\theta p}{p+1}}}\right\|_{L^{p+1}(B_{R-4h_{0}})}^{p+1}\\
    &\quad\leq c\left(\int_{t_{1}}^{t_{2}}\sup_{0<|h|<h_{0}}\left\|\frac{\delta^{2}_{h}u(\cdot,t)}{|h|^{\frac{1+\theta p}{p}}}\right\|_{L^{p}(B_{R+4h_{0}})}^{p}\dt+1\right)
\end{flalign*}
for some $c \equiv c(n,s,\lambda,p,h_{0},\tau,\theta)$, whenever $4h_{0}<R\leq1-5h_{0}$ and $0<\tau<t_{2}-t_{1}$.
\end{lem}
\begin{proof}
We first assume $t_{2}<0$ to use Lemma \ref{LDd1}. the case $t_{2}=0$ will be considered in Step 4.

\medskip

\noindent
\textbf{Step 1: Discrete differentiation of the equation.} Set $r=R-4h_{0}$ and fix $h\in\mathbb{R}^{n}$ such that $0<|h|<h_{0}$. Let $\psi(x)\in C_{c}^{\infty}(B_{R})$ be a cut-off function satisfying 
\begin{equation*}
    0\leq\psi\leq1,\quad \psi\equiv1 \ \text{in} \ B_{r} \quad \text{and} \quad \|\nabla\psi\|_{L^{\infty}(B_{R})}\leq\frac{2}{R-r}=\frac{1}{2h_{0}}.
\end{equation*}
Moreover, let $\eta(t) \in C^{\infty}(\mathbb{R}^{n})$ be a nonnegative function such that 
\begin{equation}
\label{test ftn for moser time}
\eta(t) \equiv 0 \ \mbox{in} \ (-\infty,t_{1}],\quad \eta(t) \equiv 1 \ \mbox{in} \ [t_{1}+\tau,\infty) \quad \mbox{and}\quad \|\eta'\|_{L^{\infty}(\mathbb{R})}\leq\frac{2}{\tau}.
\end{equation}
Set $g(t)\coloneqq|t|^{p-1}t$ and $G(t)\coloneqq\frac{1}{p+1}|t|^{p+1}$ for $t\in\mathbb{R}$. By using \eqref{LDd1.est} in Lemma \ref{LDd1} and dividing both sides by $|h|^{1+\theta p}$, we have
\begin{align*}
 \rom{1} &:= \int_{t_{1}}^{t_{2}}\int_{B_{R}}\int_{B_{R}}\frac{A(x,y,t)\big(\Phi(u_{h}(x,t)-u_{h}(y,t))-\Phi(u(x,t)-u(y,t))\big)}{|h|^{1+\theta p}|x-y|^{n+2s}}\\
    &\qquad\times\big[g(u_{h}(x,t)-u(x,t))\psi^{2}(x)-g(u_{h}(y,t)-u(y,t))\psi^{2}(y)\big]\eta(t)\dx\dy\dt\\
    &=-\int_{t_{1}}^{t_{2}}\int_{\mathbb{R}^{n}\setminus B_{R}}\int_{B_{R}}\frac{A_{h}(x,y,t)\big(\Phi(u_{h}(x,t)-u_{h}(y,t))-\Phi(u(x,t)-u(y,t))\big)}{|h|^{1+\theta p}|x-y|^{n+2s}}\\
    &\quad\qquad\times g(u_{h}(x,t)-u(x,t))\psi^{2}(x)\eta(t)\dx\dy\dt\\
    &\quad-\int_{t_{1}}^{t_{2}}\int_{B_{R}}\int_{\mathbb{R}^{n}\setminus B_{R}}\frac{A_{h}(x,y,t)\big(\Phi(u_{h}(x,t)-u_{h}(y,t))-\Phi(u(x,t)-u(y,t))\big)}{|h|^{1+\theta p}|x-y|^{n+2s}}\\
    &\quad\qquad\times g(u_{h}(y,t)-u(y,t))\psi^{2}(y)\eta(t)\dx\dy\dt\\
    &\quad-\int_{B_{R}}\frac{G(\delta_{h}u(x,t_{2}))}{|h|^{1+\theta p}}\psi^{2}(x)\dx+\int_{t_{1}}^{t_{2}}\int_{B_{R}}\frac{G(\delta_{h}u(x,t))}{|h|^{1+\theta p}}\psi^{2}(x)\eta'(t)\dx\dt\\
&=: - \rom{2} - \rom{3} - \rom{4} + \rom{5}.
\end{align*}

\noindent
\textbf{Step 2: Estimates of $\rom{1} \ \mbox{through} \ \rom{5}$.} First, we will estimate $\rom{1}$,$\rom{2}$ and $\rom{3}$. Let $F_{1}(t)$, $F_{2}(t)$ and $F_{3}(t)$ be the integrands of $\rom{1}$, $\rom{2}$ and $\rom{3}$ with respect to the measure $\eta(t) \,dt$, respectively. Then for fixed $t \in [t_{1},t_{2}]$, similar calculations as in \cite[Proposition 3.1, Step 2 - Step 4]{NHH} lead to the following estimates,
\begin{align*}
F_{1}(t)\geq\frac{1}{c}\left[\frac{|\delta_{h}u(\cdot,t)|^{\frac{p-1}{2}}\delta_{h}u(\cdot,t)}{|h|^{\frac{1+\theta p}{2}}}\psi(\cdot)\right]_{W^{s,2}(B_{R})}^{2}-c\int_{B_{R}}\frac{|\delta_{h}u(x,t)|^{p+1}}{|h|^{1+\theta p}}\dx,
\end{align*}
and
\begin{align*}
|F_{2}(t)|+|F_{3}(t)|\le c\int_{B_{R}}\frac{|\delta_{h}u(x,t)|^{p}}{|h|^{1+\theta p}}\dx.
\end{align*}
where $c$ depends only on $n,s,\lambda$ and $h_{0}$, but it is independent on $t$.
It follows that
\begin{equation*}
    \rom{1}\geq\frac{1}{c}\int_{t_{1}}^{t_{2}}\left[\frac{|\delta_{h}u(\cdot,t)|^{\frac{p-1}{2}}\delta_{h}u(\cdot,t)}{|h|^{\frac{1+\theta p}{2}}}\psi(\cdot)\right]_{W^{s,2}(B_{R})}^{2}\eta(t)\dt-c\int_{t_{1}}^{t_{2}}\int_{B_{R}}\frac{|\delta_{h}u(x,t)|^{p}}{|h|^{1+\theta p}}\dx\dt,
\end{equation*}
and
\begin{equation*}
    |\rom{2}|+|\rom{3}|\leq c\int_{t_{1}}^{t_{2}}\int_{B_{R}}\frac{|\delta_{h}u(x,t)|^{p}}{|h|^{1+\theta p}}\dx\dt,
\end{equation*}
where we have used $\eta \le 1$.
Next, we will derive estimates of $\rom{4}$ and $\rom{5}$. From the definition of $G$ and $\psi$, we get
\begin{equation*}
\rom{4} \ge \frac{1}{c}\int_{B_{r}}\frac{|\delta_{h}u(x,t_{2})|^{p+1}}{|h|^{1+\theta p}}\dx.
\end{equation*}
Since $\|u_{h}\|_{L^{\infty}(B_{R}\times[t_{1},t_{2}])}\leq\|u\|_{L^{\infty}(B_{1}\times[-1,0])}\leq1$ and \eqref{test ftn for moser time}, we deduce that 
\begin{equation*}
    |\rom{5}|\leq \frac{c}{\tau}\int_{t_{1}}^{t_{2}}\int_{B_{R}}\frac{|\delta_{h}u(x,t)|^{p}}{|h|^{1+\theta p}}\dx\dt.
\end{equation*}
Combining all the above estimates, we have 
\begin{align}
\label{EI}
    \int_{t_{1}}^{t_{2}}\left[\frac{|\delta_{h}u(\cdot,t)|^{\frac{p-1}{2}}\delta_{h}u(\cdot,t)}{|h|^{\frac{1+\theta p}{2}}}\psi(\cdot)\right]_{W^{s,2}(B_{R})}^{2}\eta(t)\dt&+\int_{B_{r}}\frac{|\delta_{h}u(x,t_{2})|^{p+1}}{|h|^{1+\theta p}}\dx\nonumber\\
    &\leq c(\tau)\int_{t_{1}}^{t_{2}}\int_{B_{R}}\frac{|\delta_{h}u(x,t)|^{p}}{|h|^{1+\theta p}}\dx\dt.
\end{align}

\noindent
We further estimate using the advantage of \cite[Proposition 3.1, Step 5]{NHH} and \eqref{test ftn for moser time},
\begin{align}\label{before1}
\begin{split}
    \int_{t_1+\tau}^{t_{2}}\left\|\frac{\delta^{2}_{h}u(\cdot,t)}{|h|^{\frac{1+2s+\theta p}{p+1}}}\right\|_{L^{p+1}(B_{r})}^{p+1}\dt&\leq     c\int_{t_{1}}^{t_{2}}\left[\frac{|\delta_{h}u(\cdot,t)|^{\frac{p-1}{2}}\delta_{h}u(\cdot,t)}{|h|^{\frac{1+\theta p}{2}}}\psi(\cdot)\right]_{W^{s,2}(B_{R})}^{2}\eta(t)\dt\\
    &\quad+c\int_{t_{1}}^{t_{2}}\int_{B_{R}}\frac{|\delta_{h}u(x,t)|^{p}}{|h|^{1+\theta p}}\dx\dt.
\end{split}
\end{align}
Moreover, Lemma \ref{embed1} and $\|u\|_{L^{\infty}(Q_{1})}\leq1$ imply
\begin{equation}
\label{before2}
\begin{aligned}
    \int_{t_{1}}^{t_{2}}\int_{B_{R}}\frac{|\delta_{h}u(x,t)|^{p}}{|h|^{1+\theta p}}\dx\dt&\leq c\left(\int_{t_{1}}^{t_{2}}\int_{B_{R+4h_{0}}}\frac{|\delta^{2}_{h}u(x,t)|^{p}}{|h|^{1+\theta p}}\dx\dt+\int_{t_{1}}^{t_{2}}\|u(\cdot,t)\|_{L^{p}(B_{R+4h_{0}})}^{p}\dt\right)\\
    &\leq c\left(\int_{t_{1}}^{t_{2}}\int_{B_{R+4h_{0}}}\frac{|\delta^{2}_{h}u(x,t)|^{p}}{|h|^{1+\theta p}}\dx\dt+1\right).
\end{aligned}
\end{equation}
Gathering together the estimates \eqref{EI}, \eqref{before1} and \eqref{before2} gives
\begin{align}\label{step3}
\begin{split}
    \int_{t_{1}+\tau}^{t_{2}}\sup_{0<|h|<h_{0}}\left\|\frac{\delta^{2}_{h}u(\cdot,t)}{|h|^{\frac{1+2s+\theta p}{p+1}}}\right\|_{L^{p+1}(B_{R-4h_{0}})}^{p+1}\dt&+\int_{B_{r}}\frac{|\delta_{h}u(x,t_{2})|^{p+1}}{|h|^{1+\theta p}}\dx\\
    &\leq c\left(\int_{t_{1}}^{t_{2}}\sup_{0<|h|<h_{0}}\left\|\frac{\delta^{2}_{h}u(\cdot,t)}{|h|^{\frac{1+\theta p}{p}}}\right\|_{L^{p}(B_{R+4h_{0}})}^{p}\dt+1\right).
\end{split}
\end{align}

\noindent
\textbf{Step 4: Case for $t_{2}=0.$} Note that $c$ in \eqref{step3} is independent of $t_{1}$. By the monotone convergence theorem, we get
\begin{equation}\label{step4.1}
    \lim_{t_{2}\to0}\int_{t_1+\tau}^{t_{2}}\sup_{0<|h|<h_{0}}\left\|\frac{\delta^{2}_{h}u(\cdot,t)}{|h|^{\frac{1+2s+\theta p}{p+1}}}\right\|_{L^{p+1}(B_{R-4h_{0}})}^{p+1}\dt=\int_{t_1+\tau}^{0}\sup_{0<|h|<h_{0}}\left\|\frac{\delta^{2}_{h}u(\cdot,t)}{|h|^{\frac{1+2s+\theta p}{p+1}}}\right\|_{L^{p+1}(B_{R-4h_{0}})}^{p+1}\dt.
\end{equation}
Since for any $0<|h|<h_{0}$, $\delta_{h}^{2}u(x,t)$ is in $C([-1,0];L^{2}(B_{1}))$, we have 
\[\lim_{t_{2}\to0}\|\delta_{h}^{2}u(\cdot,t_{2})-\delta_{h}^{2}u(\cdot,0)\|_{L^{2}(B_{R})}=0.\]which implies 
\[\lim_{t_{2}\to0}\delta_{h}^{2}u(x,t_{2})=\delta_{h}^{2}u(x,0)\quad a.e. \text{ in }B_{R}. \]
Therefore Fatou's lemma yields
\begin{equation}\label{step4.2}
    \left\|\frac{\delta_{h}u(\cdot,0)}{|h|^{\frac{1+\theta p}{p+1}}}\right\|_{L^{p+1}(B_{R-4h_{0}})}^{p+1}\leq\liminf_{t_{2}\to0}\left\|\frac{\delta_{h}u(\cdot,t_{2})}{|h|^{\frac{1+\theta p}{p+1}}}\right\|_{L^{p+1}(B_{R-4h_{0}})}^{p+1},\quad 0<|h|<h_{0}.
\end{equation}
Taking into account \eqref{step3}, \eqref{step4.1} and \eqref{step4.2}, we conclude
\begin{align*}
&\int_{t_1+\tau}^{0}\sup_{0<|h|<h_{0}}\left\|\frac{\delta^{2}_{h}u(\cdot,t)}{|h|^{\frac{1+2s+\theta p}{p+1}}}\right\|_{L^{p+1}(B_{R-4h_{0}})}^{p+1}\dt+\sup_{0<|h|<h_{0}}\left\|\frac{\delta_{h}u(\cdot,0)}{|h|^{\frac{1+\theta p}{p+1}}}\right\|_{L^{p+1}(B_{R-4h_{0}})}^{p+1}\\
    &\quad\leq c\left(\int_{t_{1}}^{0}\sup_{0<|h|<h_{0}}\left\|\frac{\delta^{2}_{h}u(\cdot,t)}{|h|^{\frac{1+\theta p}{p}}}\right\|_{L^{p}(B_{R+4h_{0}})}^{p}\dt+1\right).
\end{align*}
\end{proof}
Now we prove a higher H\"older regularity for \eqref{normalized homo} with respect to the spatial variables.
\begin{lem}
\label{higher holder of homo 1}
Let $u$ be a local weak solution to \eqref{normalized homo} with \eqref{kernel coefficient condition}. Then for any $0<\alpha<\min\{2s,1\}$,
\begin{align*}
    \esup_{t\in[-\left(1/2\right)^{2s},0]}[u(\cdot,t)]_{C^{\alpha}(B_{1/2})}&\leq c\Bigg(\|u\|_{L^{\infty}(Q_{1})}+\left(\int_{-1}^{0}[u(\cdot,t)]_{W^{s,2}(B_{1})}^{2}\dt\right)^{\frac{1}{2}}+\ITail(u;0,1)\Bigg),
\end{align*}
where $c \equiv c(n,s,\alpha,\lambda)$. 
\end{lem}
\begin{proof}
\textbf{Step 1: Normalization.} Let
\begin{align*}
M&\coloneqq\|u\|_{L^{\infty}(Q_{1})}+\left(\int_{-1}^{0}[u(\cdot,t)]^{2}_{W^{s,2}(B_{1})}\dt\right)^{\frac{1}{2}}+\ITail(u;0,1)
\end{align*}
and set
\[\Tilde{u}(x,t)\coloneqq \frac{1}{M}u(x,t),\quad (x,t)\in Q_{2}\quad\text{and}\quad\Tilde{\Phi}(\xi)\coloneqq\frac{1}{M}\Phi(M\xi),\quad \xi\in\mathbb{R}.\]
Then $\Tilde{u}$ is a local weak solution to $\partial_{t}\Tilde{u}+\mathcal{L}_{\Tilde{A}}^{\Tilde{\Phi}}\Tilde{u}=0$ in $Q_{2}$. In particular, $\Tilde{u}$ satisfies 
\begin{equation}
\label{scsc}
\|\Tilde{u}\|_{L^{\infty}(Q_{1})}\leq 1, \ \int_{-1}^{0}[\Tilde{u}(\cdot,t)]^{2}_{W^{s,2}(B_{1})}\dt\leq 1, \ \mbox{and} \ \ITail(\Tilde{u};0,1)\leq1.
\end{equation}
\textbf{Step 2: Iteration.} Let $q_{i}\coloneqq i+2$ and $\theta_{i}\coloneqq\frac{2s(i+1)-1}{i+2}$ for $i \ge 0$. Then $\frac{1+\theta_{i}q_{i}}{q_{i}}$ is an increasing sequence satisfying
\begin{equation*}
    \frac{1+2s+\theta_{i}q_{i}}{q_{i}+1}=\frac{1+\theta_{i+1}q_{i+1}}{q_{i+1}}, \quad s=\frac{1+\theta_{0}q_{0}}{q_{0}} \le \frac{1+\theta_{i}q_{i}}{q_{i}}<2s, \quad \text{and} \quad \underset{i \rightarrow \infty}{\lim}\frac{1+\theta_{i}q_{i}}{q_{i}}=2s.
\end{equation*}
Now fix $\alpha\in(0,\min\{2s,1\})$. We consider the following two cases.

\medskip

\noindent
\textbf{(Case 1 : $s\leq\frac{1}{2}$).} 
In this case, we can choose $i_{\alpha} \in \mathbb{N}$ such that 
\begin{equation}
\label{alpha1}
    \alpha<\frac{1+\theta_{i_{\alpha}}q_{i_{\alpha}}-n}{q_{i_{\alpha}}}<2s \le 1.
\end{equation}
Let 
\begin{align}\label{h_0 case 1}
h_{0}=\frac{1}{64{i_{\alpha}}}
\end{align}
and set for any $i \ge 0$, 
\begin{equation}\label{R_i,T_i}
 R_{i}=\frac{7}{8}-4h_{0}(2i+1)\quad\mbox{and}\quad   T_{i}=-\left(\frac{7}{8}\right)^{2s}+32h_{0}\left(\left(\frac{7}{8}\right)^{2s}-\left(\frac{3}{4}\right)^{2s}\right)i.
\end{equation} 
Then $R_{i}$ and $T_{i}$ have the following relations.
\[R_{i}-4h_{0}=R_{i+1}+4h_{0},\quad T_{i+1}=T_{i}+32h_{0}\left(\left(\frac{7}{8}\right)^{2s}-\left(\frac{3}{4}\right)^{2s}\right),\]
\[R_{0}+4h_{0}=\frac{7}{8},\quad R_{i_{\alpha}}+4h_{0}=\frac{3}{4}\quad\text{and}\quad T_{i_{\alpha}}=-\frac{1}{2}\left(\left(\frac{7}{8}\right)^{2s}+\left(\frac{3}{4}\right)^{2s}\right).\]
Note that $\frac{7}{8}+4h_{0}\leq1$ and $\delta^{2}_{h}\Tilde{u}=\delta_{2h}\Tilde{u}-2\delta_{h}\Tilde{u}$ to find
\begin{align*}
    \int_{T_{0}}^{0}\sup_{0<|h|<h_{0}}\left\|\frac{\delta^{2}_{h}\Tilde{u}(\cdot,t)}{|h|^{\frac{1+\theta_{0} q_{0}}{q_{0}}}}\right\|_{L^{q_{0}}(B_{R_{0}+4h_{0}})}^{q_{0}}\dt&= \int_{-\left(\frac{7}{8}\right)^{2s}}^{0}\sup_{0<|h|<h_{0}}\left\|\frac{\delta^{2}_{h}\Tilde{u}(\cdot,t)}{|h|^{s}}\right\|_{L^{2}(B_{7/8})}^{2}\dt\\
    &\leq c\left(\int_{-1}^{0}[\Tilde{u}(\cdot,t)]^{2}_{W^{s,2}(B_{1})}\dt+\int_{-1}^{0}\|\Tilde{u}(\cdot,t)\|^{2}_{L^{2}(B_{1})}\dt\right) \leq c,
\end{align*}
where we have also used Lemma \ref{embedSB} and \eqref{scsc}. Thus Lemma \ref{base for iteration} implies
\begin{align*}
&\int_{T_{i+1}}^{T}\sup_{0<|h|<h_{0}}\left\|\frac{\delta^{2}_{h}\Tilde{u}(\cdot,t)}{|h|^{\frac{1+\theta_{i+1}q_{i+1}}{q_{i+1}}}}\right\|_{L^{q_{i+1}}(B_{R_{i}-4h_{0}})}^{q_{i+1}}\dt+\sup_{0<|h|<h_{0}}\left\|\frac{\delta_{h}\Tilde{u}(\cdot,T)}{|h|^{\frac{1+\theta_{i+1} q_{i+1}}{q_{i+1}}}}\right\|_{L^{q_{i+1}}(B_{R_{i}-4h_{0}})}^{q_{i+1}}\\
    &\quad\leq c\left(\int_{T_{i}}^{T}\sup_{0<|h|<h_{0}}\left\|\frac{\delta^{2}_{h}\Tilde{u}(\cdot,t)}{|h|^{\frac{1+\theta_{i} q_{i}}{q_{i}}}}\right\|_{L^{q_{i}}(B_{R_{i}+4h_{0}})}^{q_{i}}\dt+1\right),
\end{align*}
for all $i=0,\ldots ,i_{\alpha-1}$ and $-\left(\frac{3}{4}\right)^{2s}\leq T\leq0$. After finite steps, we obtain
\begin{align*}
    \esup_{T\in[-\left(3/4\right)^{2s},0]}\sup_{0<|h|<h_{0}}\left\|\frac{\delta_{h}\Tilde{u}(\cdot,T)}{|h|^{\frac{1+\theta_{i_{\alpha}} q_{i_{\alpha}}}{q_{i_{\alpha}}}}}\right\|_{L^{q_{i_{\alpha}}}(B_{3/4})}^{q_{i_{\alpha}}}\leq c(n,s,\lambda,\alpha).
\end{align*}
Thus Lemma \ref{Holder for Besov} with \eqref{alpha1} yields
\begin{equation*}
    \esup_{t\in[-(1/2)^{2s},0]}[\Tilde{u}(\cdot,t)]_{C^{\alpha}(B_{1/2})}\leq c(n,s,\lambda,\alpha).
\end{equation*}
\textbf{(Case 2 : $s>\frac{1}{2}$).} On the other hand, when $2s>1$, there exists some $i_{\alpha}\in\mathbb{N}$ such that 
\begin{equation*}
    \frac{1+\theta_{i_{\alpha}-1}q_{i_{\alpha}-1}}{q_{i_{\alpha}-1}}<1\leq\frac{1+\theta_{i_{\alpha}}q_{i_{\alpha}}}{q_{i_{\alpha}}}.
\end{equation*}
Now take $\gamma \in (\alpha,1)$. Then there exists some $j_{\alpha}\in\mathbb{N}$ such that
\begin{equation}
\label{alpha2}
    \alpha<\gamma-\frac{n}{q_{i_{\alpha}+j_{\alpha}}}.
\end{equation}
Take 
\begin{align}\label{h_0 case 2}
    h_{0}=\frac{1}{64({i_{\alpha}}+j_{\alpha})},
\end{align}
and let $R_{i}$ and $T_{i}$ be as in \eqref{R_i,T_i}, but we use \eqref{h_0 case 2} instead of \eqref{h_0 case 1}.
Then
\[R_{0}+4h_{0}=\frac{7}{8},\quad R_{i_{\alpha}+j_{\alpha}}+4h_{0}=\frac{3}{4},\quad \mbox{and}\quad T_{i_{\alpha}+j_{\alpha}}=-\frac{1}{2}\left(\left(\frac{7}{8}\right)^{2s}+\left(\frac{3}{4}\right)^{2s}\right).\]
A similar calculation in Case 1 shows that
\begin{equation*}
    \int^{0}_{T_{i_{\alpha}}}\sup_{0<|h|<h_{0}}\left\|\frac{\delta^{2}_{h}\Tilde{u}(\cdot,t)}{|h|^{\gamma}}\right\|_{L^{q_{i_{\alpha}}}(B_{R_{i_{\alpha}+4h_{0}}})}^{q_{i_{\alpha}}}\dt\le\int_{T_{i_{\alpha}}}^{0}\sup_{0<|h|<h_{0}}\left\|\frac{\delta^{2}_{h}\Tilde{u}(\cdot,t)}{|h|^{\frac{1+\theta_{i_{\alpha}}q_{i_{\alpha}}}{q_{i_{\alpha}}}}}\right\|_{L^{q_{i_{\alpha}}}(B_{R_{i_{\alpha}+4h_{0}}})}^{q_{i_{\alpha}}}\dt\leq c,
\end{equation*}
where we have also used $\gamma<1\leq\frac{1+\theta_{i_{\alpha}}q_{i_{\alpha}}}{q_{i_{\alpha}}}$.
Take $\Tilde{\theta}_{i}=\gamma-\frac{1}{q_{i}}$, which implies $\frac{1+\Tilde{\theta}_{i}q_{i}}{q_{i}}=\gamma\in(0,1)$. Then since $s>\frac{1}{2}$, we discover 
\begin{equation*}
    \gamma<1+\frac{q_{i}(\gamma-1)}{q_{i}+1}=\frac{2+\Tilde{\theta}_{i}q_{i}}{q_{i}+1}<\frac{1+2s+\Tilde{\theta}_{i}q_{i}}{q_{i}+1}.
\end{equation*}
Therefore by applying $\Tilde{\theta}_{i}$ instead of $\theta_{i}$ to Lemma \ref{base for iteration}, we have
\begin{align*}
    \int^{T}_{T_{i+1}}\sup_{0<|h|<h_{0}}\left\|\frac{\delta^{2}_{h}\Tilde{u}(\cdot,t)}{|h|^{\gamma}}\right\|_{L^{q_{i+1}}(B_{R_{i}-4h_{0}})}^{q_{i+1}}\dt
    &\leq \int^{T}_{T_{i+1}}\sup_{0<|h|<h_{0}}\left\|\frac{\delta^{2}_{h}\Tilde{u}(\cdot,t)}{|h|^{\frac{1+2s+\Tilde{\theta}_{i}q_{i}}{q_{i}+1}}}\right\|_{L^{q_{i+1}}(B_{R_{i}-4h_{0}})}^{q_{i+1}}\dt\\
    &\leq c\left(\int_{T_{i}}^{T}\sup_{0<|h|<h_{0}}\left\|\frac{\delta^{2}_{h}\Tilde{u}(\cdot,t)}{|h|^{\frac{1+\Tilde{\theta}_{i} q_{i}}{q_{i}}}}\right\|_{L^{q_{i}}(B_{R_{i}+4h_{0}})}^{q_{i}}\dt+1\right),
\end{align*}
for each $i=i_{\alpha},\ldots,i_{\alpha}+j_{\alpha}-1$ and for any $T\in[-\left(\frac{3}{4}\right)^{2s},0]$. By Lemma \ref{Holder for Besov} and \eqref{alpha2}, we conclude 
\begin{equation*}
    \esup_{t\in[-(1/2)^{2s},0]}[\Tilde{u}(\cdot,t)]_{C^{\alpha}(B_{1/2})}\leq c.
\end{equation*}
By scaling back, we complete the proof.
\end{proof}
Next, we will prove a higher H\"older regularity for a solution to \eqref{normalized homo} with respect to time variable. To this end, we first introduce two lemmas. One is a generalized Poincar\'e type inequality from \cite[Lemma 6.1]{BLS}.
\begin{lem}[Generalized Poincar\'e type inequality] 
\label{GPI}
Let $0<s<1$ and $1\leq p<\infty$. For any $u\in W^{s,p}(B_{r})$ and $\psi\in C^{\infty}_{c}(B_{r})$ with $\overline{\psi}_{B_{r}}=1$, there holds 
\begin{equation*}
    \int_{B_{r}}|u-\overline{(u\psi)}_{B_{r}}|^{p}\dx\leq cr^{sp}\|\psi\|_{L^{\infty}(B_{r})}^{p}[u]_{W^{s,p}(B_{r})}^{p},
\end{equation*}
where $c \equiv c(n,s,p)$.
\end{lem}
The other is a gluing lemma regarding a H\"older regularity with respect to time variable.
\begin{lem}[Gluing lemma]
\label{GL}
Let $u$ be a local weak solution to \eqref{eq1} in $Q_{1}$ and let $z_{0}\in Q_{1}$ and $Q_{\rho,\theta}(z_{0})\Subset Q_{1}$.
Let $B_{\rho}\equiv B_{\rho}(x_{0})$ and let $\psi(x) \in C_{c}^{\infty}(B_{3\rho/4})$ be a nonnegative function such that 
\begin{equation}
\label{psi condition}
\overline{\psi}_{B_{\rho}}=1,\ \psi \equiv \|\psi\|_{L^{\infty}(B_{\rho})} \text{ in }B_{\frac{\rho}{2}} \text{ and } \|\nabla\psi\|_{L^{\infty}(\mathbb{R}^{n})}\leq\frac{c}{\rho}.
\end{equation}Then we have
\begin{equation}
\label{gluing lemma inequality}  
\begin{aligned}
    \left|\overline{(u\psi)}_{B_{\rho}}(T_{1})-\overline{(u\psi)}_{B_{\rho}}(T_{0})\right|&\leq c\frac{\theta}{\rho}\dashint_{t_{0}-\theta}^{t_{0}}\int_{B_{\rho}}\dashint_{B_{\rho}}\frac{|u(x,t)-u(y,t)|}{|x-y|^{n+2s-1}}\dx\dy\dt\\
    &\quad+c\theta\dashint_{t_{0}-\theta}^{t_{0}}\int_{\mathbb{R}^{n}\setminus B_{\rho}}\dashint_{B_{{3\rho/4}}}\frac{|u(x,t)-u(y,t)|}{|x_{0}-y|^{n+2s}}\dx\dy\dt,
\end{aligned}
\end{equation}
where $c \equiv c(n,s,\lambda)$ and $t_{0}-\theta < T_{0} < T_{1} < t_{0}$.
\end{lem} 
\begin{proof}
For a sufficiently small $0<\epsilon<\min\{\frac{T_{1}-T_{0}}{4},\frac{1+T_{0}}{2},-\frac{T_{1}}{2}\}$, we define 
\begin{equation*}
\begin{aligned}
    \eta_{\epsilon}(t)=\begin{cases}
    0 &\text{ in } t< T_{0}\\
    \frac{t-T_{0}}{\epsilon} &\text{ in } T_{0}\le t< T_{0}+\epsilon\\
    1 &\text{ in }T_{0}+\epsilon\le t < T_{1}-\epsilon\\
    \frac{T_{1}-t}{\epsilon} &\text{ in } T_{1}-\epsilon\le t< T_{1}\\
    0 &\text{ in } t\ge T_{1}.
    \end{cases}
\end{aligned}    
\end{equation*}
By the condition of \eqref{psi condition}, we obtain
\begin{equation*}
    \|\psi\|_{L^{\infty}(\mathbb{R}^{n})}= \overline{\psi}_{B_{\frac{\rho}{2}}}\leq c\overline{\psi}_{B_{\rho}}=c.
\end{equation*}
{Then we can use $\psi(x)\eta_{\epsilon}(t)$ as a test function}. By the definition of local weak solution and the conditions \eqref{Phi condition}, we see that
\begin{equation}
\label{for gluing11}
\begin{aligned}
&\left|\dashint_{T_{1} - \epsilon}^{T_{1}}\int_{B_{\rho}}u\psi\,dxdt - \dashint_{T_{0}}^{T_{0} + \epsilon}\int_{B_{\rho}}u\psi\,dxdt\right|\\
&= \left|\int_{t_{0} - \theta}^{t_{0}}\int_{B_{\rho}}u\psi\eta'_{\epsilon}\,dxdt\right|\\
&\le c\int_{t_{0} - \theta}^{t_{0}}\int_{\mathbb{R}^{n}}\int_{\mathbb{R}^{n}}\frac{|u(x,t) - u(y,t)|}{|x-y|^{n+2s}}|\psi(x) - \psi(y)|\,dxdydt\\
&\le \frac{c}{\rho}\int_{t_{0} - \theta}^{t_{0}}\int_{B_{\rho}}\int_{B_{\rho}}\frac{|u(x,t) - u(y,t)|}{|x-y|^{n+2s-1}}\,dxdydt + c\int_{t_{0} - \theta}^{t_{0}}\int_{\mathbb{R}^{n} \setminus B_{\rho}}\int_{B_{\rho}}\frac{|u(x,t) - u(y,t)|}{|x-y|^{n+2s}}\psi(x)\,dxdydt,\\
\end{aligned}
\end{equation}
where we used the fact that $|\psi(x)-\psi(y)| \le \frac{c}{\rho}|x-y|$.
Furthermore, the second term on the right hand side can be estimated by
\begin{equation}
\label{for gluing12}
\begin{aligned}
&\int_{t_{0} - \theta}^{t_{0}}\int_{\mathbb{R}^{n} \setminus B_{\rho}}\int_{B_{\rho}}\frac{|u(x,t) - u(y,t)|}{|x-y|^{n+2s}}\psi(x)\,dxdydt \le c\int_{t_{0} - \theta}^{t_{0}}\int_{\mathbb{R}^{n} \setminus B_{\rho}}\int_{B_{{3\rho/4}}}\frac{|u(x,t) - u(y,t)|}{|x_{0}-y|^{n+2s}}\,dxdydt,
\end{aligned}
\end{equation}
where we have used
\begin{align*}
\frac{|x_{0}-y|}{|x-y|} \le 1 + \frac{|x-x_{0}|}{|x-y|} \le 2 \quad \mbox{for} \ x \in B_{3\rho/4},\ y \in B_{\rho}.
\end{align*}
Since $u\psi$ is in $C\left([t_{0}-\theta,t_{0}];L^{2}(B_{\rho})\right)$, 
\begin{equation}
\label{for gluing2}
    \left|\overline{(u\psi)}_{B_{\rho}}(T_{1})-\overline{(u\psi)}_{B_{\rho}}(T_{0})\right|=\lim_{\epsilon\to0}\left|\dashint_{T_{1}-\epsilon}^{T_{1}}\dashint_{B_{\rho}}u\psi\dx\dt-\dashint_{T_{0}}^{T_{0}+\epsilon}\dashint_{B_{\rho}}u\psi\dx\dt\right|.
\end{equation}
Therefore, we combine \eqref{for gluing11}, \eqref{for gluing12} and \eqref{for gluing2} to discover \eqref{gluing lemma inequality}.
\end{proof}

With Lemma \ref{GPI} and Lemma \ref{GL}, we are now ready to prove the following time H\"older regularity.
\begin{lem}[H\"older regularity with respect to time]
\label{higher holder of homo 2}
Let $u$ be a local weak solution to \eqref{normalized homo}. Assume 
\begin{equation}
\label{SpHo}
    \|u\|_{L^{\infty}(Q_{1})}\leq 1,\ \ITail(u;0,1)\leq 1,\text{ and }
    \sup_{t\in[-(1/2)^{2s},0]}[u(\cdot,t)]_{C^{\beta}(B_{1/2})}\leq K_{\beta},
\end{equation}
for some $0<\beta<\min\{2s,1\}$ and $K_{\beta}>0$. Then there exists $c \equiv c(n,s,\lambda,\beta,K_{\beta})$ such that
\begin{equation*}
    |u(x,t)-u(x,\tau)|\leq  c|t-\tau|^{\frac{\beta}{2s}} \text{ for any } (x,t), (x,\tau)\in Q_{\frac{1}{4}}.
\end{equation*}
\end{lem}
\begin{proof}
Let ${z_{0}}=(x_{0},t_{0})\in Q_{\frac{1}{4}}$. Then $Q_{\rho}(z_{0})\subset Q_{\frac{1}{2}}$ for $0<\rho<\min\left\{\left(\frac{2^{2s}-1}{4^{2s}}\right)^{\frac{1}{2s}},\frac{1}{8}\right\}$. Notice that
\begin{align*}
\frac{|y|}{|y-x_{0}|} \le 1 + \frac{|x_{0}|}{|y-x_{0}|} \le 2 \quad \mbox{for} \ y \in B_{1},
\end{align*}
to see 
\begin{equation}
\label{TailEstimate}    
\begin{aligned}
&\esup_{t\in[-(1/2)^{2s},0]}\int_{\mathbb{R}^{n}\setminus B_{1/2}(x_{0})}\frac{|u(y,t)|}{|y-x_{0}|^{n+2s}}\dy\\
&\quad \leq \esup_{t\in[-1,0]}\int_{\mathbb{R}^{n}\setminus B_{1}}\frac{|u(y,t)|}{|y-x_{0}|^{n+2s}}\dy
+\esup_{t\in[-1,0]}\int_{B_{1}\setminus B_{1/2}(x_{0})}\frac{|u(y,t)|}{|y-x_{0}|^{n+2s}}\dy\\
&\quad \leq c\ITail(u;0,1)+2^{n+2s}|B_{1}|\leq c.
\end{aligned}
\end{equation}
Take a nonnegative cutoff function $\psi\in C_{c}^{\infty}(B_{\rho}(x_{0}))$ in Lemma \ref{GL}. Then
\begin{align*}
\begin{split}
    \dashiint_{Q_{\rho}(z_{0})}|u-\overline{u}_{Q_{\rho}(z_{0})}|\dx\dt&\leq \dashiint_{Q_{\rho}(z_{0})}|u-\overline{(u\psi)}_{B_{\rho(x_{0})}}(t)|\dx\dt\\
    &\quad+\dashiint_{Q_{\rho}(z_{0})}|\overline{(u\psi)}_{B_{\rho(x_{0})}}(t)-\overline{(u\psi)}_{Q_{\rho}(z_{0})}|\dx\dt\\
    &\quad+\dashiint_{Q_{\rho}(z_{0})}|\overline{(u\psi)}_{Q_{\rho}(z_{0})}-\overline{u}_{Q_{\rho}(z_{0})}|\dx\dt\\
    &\eqqcolon A_{1}+A_{2}+A_{3}.
\end{split}
\end{align*}
We observe that
\begin{align*}
A_{3} &= |\overline{(u\psi)}_{Q_{\rho}(z_{0})}-\overline{u}_{Q_{\rho}(z_{0})}| 
\le \dashiint_{Q_{\rho}(z_{0})}|u-\overline{(u\psi)}_{Q_{\rho}(z_{0})}|\dx\dt\\
&\le \dashiint_{Q_{\rho}(z_{0})}|u-\overline{(u\psi)}_{B_{\rho(x_{0})}}(t)|\dx\dt
+ \dashiint_{Q_{\rho}(z_{0})}|\overline{(u\psi)}_{B_{\rho(x_{0})}}(t)-\overline{(u\psi)}_{Q_{\rho}(z_{0})}|\dx\dt\\
&= A_{1} + A_{2}.
\end{align*}
Therefore it is sufficient to estimate $A_{1}$ and $A_{2}$. 

\medskip

\noindent
\textbf{Estimate $A_{1}$.}
Using H\"older inequality, Lemma \ref{GPI} and \eqref{SpHo}, we have
\begin{align*}
    A_{1}&\leq c\left(\rho^{2s}\dashint_{t_{0}-\rho^{2s}}^{t_{0}}\dashint_{B_{\rho}(x_{0})}\int_{B_{\rho}(x_{0})}\frac{|u(x,t)-u(y,t)|^{2}}{|x-y|^{n+2s}}\dx\dy\dt\right)^{\frac{1}{2}}\\
    &\leq c\left(\rho^{2s}\dashint_{t_{0}-\rho^{2s}}^{t_{0}}\dashint_{B_{\rho}(x_{0})}\int_{B_{\rho}(x_{0})}\frac{K_{\beta}^{2}}{|x-y|^{n+2s-2\beta}}\dx\dy\dt\right)^{\frac{1}{2}} \leq c\rho^{\beta}.
\end{align*}
\textbf{Estimate $A_{2}$.} From Lemma \ref{GL}, we deduce
\begin{align*}
    A_{2}&\leq \dashint_{t_{0}-\rho^{2s}}^{t_{0}}\dashint_{t_{0}-\rho^{2s}}^{t_{0}}|\overline{(u\psi)}_{B_{\rho}(x_{0})}(t)-\overline{(u\psi)}_{B_{\rho}(x_{0})}(\tau)|\,d\tau\dt\\
    &\leq\sup_{T_{0},T_{1}\in(t_{0}-\rho^{2s},t_{0})}|\overline{(u\psi)}_{B_{\rho}(x_{0})}(T_{0})-\overline{(u\psi)}_{B_{\rho}(x_{0})}(T_{1})|\\
    &\leq c{\rho^{2s-1}}\dashint_{t_{0}-\rho^{2s}}^{t_{0}}\int_{B_{\rho}(x_{0})}\dashint_{B_{\rho}(x_{0})}\frac{|u(x,t)-u(y,t)|}{|x-y|^{n+2s-1}}\dx\dy\dt\\
    &\quad+c\rho^{2s}\dashint_{t_{0}-\rho^{2s}}^{t_{0}}\int_{\mathbb{R}^{n}\setminus B_{\rho}(x_{0})}\dashint_{B_{\rho}(x_{0})}\frac{|u(x,t)-u(y,t)|}{|x_{0}-y|^{n+2s}}\dx\dy\dt\\
    &=A_{2,1}+A_{2,2}
\end{align*}
By \eqref{SpHo}, we see the following  
\begin{align*}
    A_{2,1}\le c{\rho^{2s-1}}\dashint_{t_{0}-\rho^{2s}}^{t_{0}}\int_{B_{\rho}(x_{0})}\dashint_{B_{\rho}(x_{0})}\frac{K_{\beta}}{|x-y|^{n+2s-1-\beta}}\dx\dy\dt \leq c(n,s,\beta,\lambda)K_{\beta}\rho^{\beta}.
\end{align*}
We use \eqref{SpHo} and \eqref{TailEstimate} to discover  
\begin{align*}
    A_{2,2}&\leq c \rho^{2s}\dashint_{t_{0}-\rho^{2s}}^{t_{0}}\int_{\mathbb{R}^{n}\setminus B_{1}}\dashint_{B_{\rho}(x_{0})}\frac{|u(x,t)-u(y,t)|}{|x_{0}-y|^{n+2s}}\dx\dy\dt\\
    &\quad+c \rho^{2s}\dashint_{t_{0}-\rho^{2s}}^{t_{0}}\int_{B_{1}\setminus B_{\rho}(x_{0})}\dashint_{B_{\rho}(x_{0})}\frac{|u(x,t)-u(y,t)|}{|x_{0}-y|^{n+2s}}\dx\dy\dt\\
    &\leq c \rho^{2s}\dashint_{t_{0}-\rho^{2s}}^{t_{0}}\int_{\mathbb{R}^{n}\setminus B_{1}}\dashint_{B_{\rho}(x_{0})}\frac{|u(x,t)|}{|x_{0}-y|^{n+2s}}\dx\dy\dt\\
    &\quad+ c \rho^{2s}\dashint_{t_{0}-\rho^{2s}}^{t_{0}}\int_{\mathbb{R}^{n}\setminus B_{1}}\dashint_{B_{\rho}(x_{0})}\frac{|u(y,t)|}{|x_{0}-y|^{n+2s}}\dx\dy\dt+cK_{\beta}\rho^{\beta} \leq c\rho^{\beta}.
\end{align*}
Combining the estimates $A_{1}$ and $A_{2}$ yields
\begin{equation*}
    \dashiint_{Q_{\rho}(z_{0})}|u-\overline{u}_{Q_{\rho}(z_{0})}|\dx\dt\leq c\rho^{\beta}.
\end{equation*}
With the help of Lemma \ref{Campanato embedding for fractional}, we have 
\begin{equation*}
    |u(x,t)-u(x,\tau)|\leq c|t-\tau|^{\frac{\beta}{2s}} \text{ for any } (x,t), (x,\tau)\in Q_{\frac{1}{4}}.
\end{equation*}
\end{proof}


\section{Higher H\"older regularity by approximation}
\label{section 5}
This section is devoted to prove the main theorem \ref{Holder}.
We will use an approximation argument based on the comparison estimate using \eqref{kernel coefficient condition} to obtain the higher H\"older regularity for the inhomogeneous problem \eqref{eq1}. Throughout this section, we assume that $q, r>1$ satisfy
\begin{equation}
\label{qr condition}
\frac{n}{2qs}+\frac{1}{r}<1.
\end{equation}
\begin{lem}
\label{Basestep}
For any $\epsilon>0$, there exists a small $\delta\equiv\delta(n,s,q,r,\lambda,\epsilon)>0$ such that for any local weak solution $u$ to \eqref{eq1} in {$Q_{4}$} with
\begin{equation*}
\sup_{Q_{4}}|u|\le1 \quad \mbox{and} \quad \ITail(u;0,4)\leq 1,
\end{equation*}
if there hold
\begin{equation*}
\|A-\Tilde{A}\|_{L^{\infty}(\mathbb{R}^{n}\times \mathbb{R}^{n}\times [-4^{2s},0])}\leq \delta \quad \mbox{and} \quad \|f\|_{L^{q,r}(Q_{4})}\leq \delta,
\end{equation*}
where $\Tilde{A}\in\mathcal{L}_{0}(\lambda)$, then there is a local weak solution $v$ to 
\begin{align*}
   \left\{
\begin{alignedat}{3}
\partial_{t}v+\mathcal{L}^{\Phi}_{\Tilde{A}}v &=0 &&\qquad \mbox{in  $Q_{2}$} \\
v&=u&&\qquad  \mbox{on  $\partial_{P}Q_{2}\cup\left(\left(\mathbb{R}^{n}\setminus B_{2}\right)\times[-2^{2s},0]\right)$} 
\end{alignedat} \right. 
\end{align*}
such that
\begin{equation*}
    \|u-v\|_{L^{\infty}(Q_{1})}\leq \epsilon.
\end{equation*}
\end{lem}
\begin{proof}
We prove this lemma by contradiction. Suppose there exist some $\epsilon_{0}>0$, $\{A_{k}\}_{k=1}^{\infty}\subset\mathcal{L}_{0}(\lambda)$, $\{\Tilde{A}_{k}\}_{k=1}^{\infty}\subset\mathcal{L}_{0}(\lambda)$, $\{\Phi_{k}\}_{k=1}^{\infty}$ with \eqref{Phi condition}, $\{f_{k}\}_{k=1}^{\infty}$ with \eqref{qr condition} and $\{u_{k}\}_{k=1}^{\infty}$ such that $u_{k}$ is a local weak solution to
\begin{equation*}
    \partial_{t}u_{k}+\mathcal{L}_{A_{k}}^{\Phi_{k}}u_{k}=f_{k} \text{ in }Q_{4}
\end{equation*}
with 
\begin{equation}
\label{bounded assumption}
    \sup_{Q_{4}}|u_{k}|\le1,\quad \ITail(u_{k};0,4)\leq 1,
\end{equation}
\begin{equation}\label{con assumption}
    \|A_{k}-\Tilde{A}_{k}\|_{L^{\infty}(\mathbb{R}^{n}\times \mathbb{R}^{n}\times[-4^{2s},0])}\leq\frac{1}{k}\quad\mbox{and} \quad\|f_{k}\|_{L^{q,r}(Q_{4})}\leq\frac{1}{k},
\end{equation}
but
\begin{equation}
\label{contradiction}
    \|u_{k}-v_{k}\|_{L^{\infty}(Q_{1})}\geq\epsilon_{0}
\end{equation}
for any local weak solution $v_{k}$ to
\begin{align*}
   \left\{
\begin{alignedat}{3}
\partial_{t}v_{k}+\mathcal{L}^{\Phi_{k}}_{\Tilde{A}_{k}}v_{k} &=0 &&\qquad \mbox{in  $Q_{2}$} \\
v_{k}&=u_{k}&&\qquad  \mbox{on  $\partial_{P}Q_{2}\cup\left(\left(\mathbb{R}^{n}\setminus B_{2}\right)\times[-2^{2s},0]\right)$.} 
\end{alignedat} \right. 
\end{align*}
Then $w_{k}\coloneqq u_{k}-v_{k} \in L^{2}\left(\left(-2^{2s},0\right];W_{0}^{s,2}\left(B_{2}\right)\right)\cap C\left(\left[-2^{2s},0\right];L^{2}\left(B_{2}\right)\right)$ solves
\begin{equation}\label{comparison eq1}
    \partial_{t}w_{k}+\mathcal{L}_{A_{k}}^{\Phi_{k}}u_{k}-\mathcal{L}^{\Phi_{k}}_{\Tilde{A}_{k}}v_{k}=f_{k}\text{ in }Q_{2}.
\end{equation}
For convenience, we write $I=\left[-2^{2s},0\right]$. By an approximation argument, we take $w_{k}$ as a test function for \eqref{comparison eq1} to get
\begin{align}\label{comparison est1}
\begin{split}
    &\int_{B_{2}}\frac{w_{k}^{2}(x,t)}{2}\dx\\
    &\quad+\int_{I}\int_{\mathbb{R}^{n}}\int_{\mathbb{R}^{n}}\frac{\Phi_{k}(u_{k}(x,t)-u_{k}(y,t))}{|x-y|^{n+2s}}(w_{k}(x,t)-w_{k}(y,t))A_{k}(x,y,t)\dx\dy\dt\\
    &\quad-\int_{I}\int_{\mathbb{R}^{n}}\int_{\mathbb{R}^{n}}\frac{\Phi_{k}(v_{k}(x,t)-v_{k}(y,t))}{|x-y|^{n+2s}}(w_{k}(x,t)-w_{k}(y,t))\Tilde{A}_{k}(x,y,t)\dx\dy\dt\\
    &= \int_{I}\int_{B_{2}}f_{k}w_{k}\dx\dt,
\end{split}
\end{align}
where we have used the fact that $w_{k}(x,-2^{2s}) = 0$ for $x \in B_{2}$. Now let us handle
  the second and the third terms on the left hand side. Set
\begin{align*}
\rom{1} = \int_{I}\int_{\mathbb{R}^{n}}\int_{\mathbb{R}^{n}}\frac{\Phi_{k}(u_{k}(x,t)-u_{k}(y,t))}{|x-y|^{n+2s}}(w_{k}(x,t)-w_{k}(y,t))A_{k}(x,y,t)\dx\dy\dt,
\end{align*}
and
\begin{align*}
\rom{2} = \int_{I}\int_{\mathbb{R}^{n}}\int_{\mathbb{R}^{n}}\frac{\Phi_{k}(v_{k}(x,t)-v_{k}(y,t))}{|x-y|^{n+2s}}(w_{k}(x,t)-w_{k}(y,t))\Tilde{A}_{k}(x,y,t)\dx\dy\dt.
\end{align*} 
Note that \eqref{Phi condition} implies  
\begin{align*}
[\Phi_{k}(u_{k}(x,t)-u_{k}(y,t))-\Phi_{k}(v_{k}(x,t)-v_{k}(y,t))](w_{k}(x,t)-w_{k}(y,t)) \ge \frac{1}{c}|w_{k}(x,t)-w_{k}(y,t)|^{2}.
\end{align*}
By using the above inequality and \eqref{Phi condition}, we get
\begin{align*}
&\rom{1} - \rom{2}\\
&= \int_{I}\int_{\mathbb{R}^{n}}\int_{\mathbb{R}^{n}}\frac{\Phi_{k}(u_{k}(x,t)-u_{k}(y,t))-\Phi_{k}(v_{k}(x,t)-v_{k}(y,t))}{|x-y|^{n+2s}}(w_{k}(x,t)-w_{k}(y,t))\Tilde{A}_{k}\dx\dy\dt\\
&\quad +\int_{I}\int_{\mathbb{R}^{n}}\int_{\mathbb{R}^{n}}\frac{\Phi_{k}(u_{k}(x,t)-u_{k}(y,t))}{|x-y|^{n+2s}}(w_{k}(x,t)-w_{k}(y,t))({A}_{k}-\Tilde{A}_{k})\dx\dy\dt\\
&\ge \frac{1}{c}\int_{I}\int_{\mathbb{R}^{n}}\int_{\mathbb{R}^{n}}\frac{|w_{k}(x,t)-w_{k}(y,t)|^{2}}{|x-y|^{n+2s}}\dx\dy\dt\\
&\quad -\int_{I}\int_{\mathbb{R}^{n}}\int_{\mathbb{R}^{n}}\frac{|u_{k}(x,t)-u_{k}(y,t)|}{|x-y|^{n+2s}}|w_{k}(x,t)-w_{k}(y,t)||{A}_{k}-\Tilde{A}_{k}|\dx\dy\dt.
\end{align*}
Putting this inequality into \eqref{comparison est1} and taking essential supremum in $I$, we have
\begin{align*}
\begin{split}
    &J \coloneqq \esup_{t\in I}\int_{B_{2}}w_{k}^{2}(x,t)\dx + \int_{I}\int_{\mathbb{R}^{n}}\int_{\mathbb{R}^{n}}\frac{|w_{k}(x,t)-w_{k}(y,t)|^{2}}{|x-y|^{n+2s}}\dx\dy\dt\\
    &\quad \le c\int_{I}\int_{B_{2}}f_{k}w_{k}\dx\dt\\
    &\qquad +c\int_{I}\int_{\mathbb{R}^{n}}\int_{\mathbb{R}^{n}}\frac{|u_{k}(x,t)-u_{k}(y,t)|}{|x-y|^{n+2s}}|w_{k}(x,t)-w_{k}(y,t)||{A}_{k}-\Tilde{A}_{k}|\dx\dy\dt\\
    &\qquad \eqqcolon cJ_{1} + cJ_{2}.
\end{split}
\end{align*}
\noindent
\textbf{Estimate of $J_{1}$.}
Using Lemma \ref{embedSF}, H\"older's inequality, and Cauchy's inequality, we have
\begin{align*}
    J_{1}\leq c\|f_{k}\|^{2}_{L^{q,r}(Q_{2})}+\frac{1}{8}J.
\end{align*}

\noindent
\textbf{Estimate of $J_{2}$.} From \eqref{con assumption},
\begin{align*}
    J_{2}&\le \frac{c}{k}\Bigg[\int_{I}\int_{B_{3}}\int_{B_{3}}\frac{|u_{k}(x,t)-u_{k}(y,t)|}{|x-y|^{n+2s}}|w_{k}(x,t)-w_{k}(y,t)|\dx\dy\dt\\
    &\quad+\int_{I}\int_{\mathbb{R}^{n} \setminus B_{3}}\int_{B_{3}}\frac{|u_{k}(x,t)-u_{k}(y,t)|}{|x-y|^{n+2s}}|w_{k}(x,t)-w_{k}(y,t)|\dx\dy\dt\Bigg]\\
    &\eqqcolon \frac{c}{k} (J_{2,1} + J_{2,2}).
\end{align*}
To get estimate of $J_{2,1}$, we use H\"older inequality and Cauchy inequality so that
\begin{align*}
    J_{2,1}&\leq c\left(\int_{I}[u_{k}(\cdot,t)]^{2}_{W^{s,2}(B_{3})}\dt\right)^{\frac{1}{2}}\left(\int_{I}[w_{k}(\cdot,t)]^{2}_{W^{s,2}(B_{3})}\dt\right)^{\frac{1}{2}}\\
    &\leq c\int_{I}[u_{k}(\cdot,t)]^{2}_{W^{s,2}(B_{3})}\dt+\frac{J}{8}.
\end{align*}
Moreover, Lemma \ref{caccioppoli estimate} and \eqref{bounded assumption} imply $\int_{I}[u_{k}(\cdot,t)]^{2}_{W^{s,2}(B_{3})}\dt \le c$. Thus $J_{2,1} \le c + J/8$.
Next, we are going to estimate $J_{2,2}$. 
In light of the fact that $w=0$ a.e. in $\left(\mathbb{R}^{n}\setminus B_{2}\right)\times I$, we observe
\begin{align*}
    J_{2,2}&\leq\int_{I}\int_{\mathbb{R}^{n}\setminus B_{3}}\int_{B_{2}}\frac{|u_{k}(x,t)|+|u_{k}(y,t)|}{|x-y|^{n+2s}}|w_{k}(x,t)|\dx\dy\dt\\
    &\leq \int_{I}\int_{\mathbb{R}^{n}\setminus B_{3}}\int_{B_{2}}\frac{{|u_{k}(x,t)|}}{|x-y|^{n+2s}}|w_{k}(x,t)|\dx\dy\dt\\
    &\quad+\int_{I}\int_{\mathbb{R}^{n}\setminus B_{3}}\int_{B_{2}}\frac{|u_{k}(y,t)|}{|x-y|^{n+2s}}|w_{k}(x,t)|\dx\dy\dt\\
    &\eqqcolon J_{2,2,1} + J_{2,2,2}.
\end{align*}
To get estimate of $J_{2,2,1}$, we use the assumption $\sup_{Q_{3}}|u_{k}|\le1$ in \eqref{bounded assumption} to see
\begin{align*}
\int_{I}\int_{\mathbb{R}^{n}\setminus B_{3}}\int_{B_{2}}\frac{|u_{k}(x,t)|}{|x-y|^{n+2s}}|w_{k}(x,t)|\dx\dy\dt
&\le \int_{I}\int_{B_{2}}\int_{\mathbb{R}^{n}\setminus B_{3}}\frac{|w_{k}(x,t)|}{|x-y|^{n+2s}}\dy\dx\dt\\
&\le c\int_{I}\int_{B_{2}}\int_{\mathbb{R}^{n}\setminus B_{3}}\frac{|w_{k}(x,t)|}{|y|^{n+2s}}\dy\dx\dt\\
&\le c\int_{I}\int_{B_{2}}|w_{k}(x,t)|\dx\dt,
\end{align*}
where we have also used the fact that $\frac{|y|}{|x-y|}$ is bounded above for all $x \in B_{2}$ and $y \in \mathbb{R}^{n} \setminus B_{3}$.
Furthermore, to estimate $J_{2,2,2}$, we used the assumption $\ITail(u_{k};0,3)\leq 1$ as follows.
\begin{align*}
\int_{I}\int_{\mathbb{R}^{n}\setminus B_{3}}\int_{B_{2}}\frac{|u_{k}(y,t)|}{|x-y|^{n+2s}}|w_{k}(x,t)|\dx\dy\dt
&\le {c}\int_{I}\int_{B_{2}}\int_{\mathbb{R}^{n}\setminus B_{3}}\frac{|u_{k}(y,t)|}{|y|^{n+2s}}|w_{k}(x,t)|\dy\dx\dt\\
&\le c\ITail(u_{k};0,3)\int_{I}\int_{B_{2}}|w_{k}(x,t)|\dx\dt\\
&\le c\int_{I}\int_{B_{2}}|w_{k}(x,t)|\dx\dt.
\end{align*}
Thus
\begin{align*}
J_{2,2} &\le c\int_{I}\int_{B_{2}}|w_{k}(x,t)|\dx\dt \le c\left(\int_{I}\int_{B_{2}}|w_{k}(x,t)|^{2}\dx\dt\right)^{\frac{1}{2}}.
\end{align*}
Combining all the above estimates, we have 
\begin{equation}
\label{wk estimate}
    J=\int_{I}\int_{B_{2}}\int_{B_{2}}\frac{|w_{k}(x,t)-w_{k}(y,t)|^{2}}{|x-y|^{n+2s}}\dx\dy\dt+\esup_{t\in I}\int_{B_{2}}w_{k}^{2}(x,t)\dx\leq \frac{c}{k}.
\end{equation}
Therefore, we have
\begin{equation}
\label{limit}
    \lim_{k\to\infty}\esup_{t\in I}\int_{B_{2}}w_{k}^{2}(x,t)\dx=0.
\end{equation}
By Lemma \ref{Holder for data}, {$u_{k}$ and $v_{k}$} are H\"older continuous functions in $\overline{Q_{1}}$. In particular, there is a $\rho=\rho(s)>0$ such that for any $z_{0}\in \overline{Q_{1}}$, $Q_{2\rho}(z_{0})\subset Q_{2}$. From Lemma \ref{LBlem}, \eqref{bounded assumption}, and \eqref{wk estimate},  we have  
\begin{align*}
    \|v_{k}\|_{L^{\infty}(Q_{\rho}(z_{0}))}&\leq c\left(\left(\dashiint_{Q_{2\rho_{0}}(z_{0})}v^{2}_{k}\dx\dt\right)^{\frac{1}{2}}+\ITail(v_{k};x_{0},\rho/2,t_{0}-\rho^{2s},t_{0})\right)\\
    &\leq c\left(\left(\dashiint_{Q_{2\rho_{0}}(z_{0})}u^{2}_{k}\dx\dt\right)^{\frac{1}{2}}+\left(\dashiint_{Q_{2\rho_{0}}(z_{0})}w^{2}_{k}\dx\dt\right)^{\frac{1}{2}}\right)\\
    &\quad+c\left(\ITail(u_{k};x_{0},\rho/2,t_{0}-\rho^{2s},t_{0})+\ITail(w_{k};x_{0},\rho/2,t_{0}-\rho^{2s},t_{0})\right) \leq c.
\end{align*}
Similarly, using Lemma \ref{HomoHolder} and Lemma \ref{Holder for data}, we see that there are constants $\beta=\beta(n,s,q,r,\lambda)$ and $c \equiv c(n,s,q,r,\lambda)$ which {are} independent of $k$ such that
\begin{align*}
 \sup_{\overline{Q_{1}}}|v_{k}(x,t)|+[v_{k}]_{C^{\beta,\frac{\beta}{2s}}(\overline{Q_{1}})}\leq c,
\end{align*}
and 
\begin{align*}
    \sup_{\overline{Q_{1}}}|u_{k}(x,t)|+[u_{k}]_{C^{\beta,\frac{\beta}{2s}}(\overline{Q_{1}})}\leq c.
\end{align*}
By Arzel\'a-Ascoli theorem, there exist a subsequence of $w_{k_{j}}$ and a continuous function $w$ in $\overline{Q_{1}}$ such that $w_{k_{j}}\to w$ in $\overline{Q_{1}}$. From \eqref{limit} with the uniqueness of the limit, we have
\begin{equation*}
    \lim_{j\to\infty}\|w_{k_{j}}\|_{L^{\infty}(\overline{Q_{1}})}=0.
\end{equation*}
which contradicts \eqref{contradiction}. This completes the proof. 
\end{proof}
Now using Lemma \ref{Basestep}, we obtain the higher H\"older regularity provided that the kernel coefficient $A$ is sufficiently close to the corresponding kernel coefficient $\Tilde{A}$ which is invariant under the translation and the nonhomogeneous term $f$ is sufficiently small in $L^{q,r}$.
\begin{lem}
\label{almost main}
For any $0<\alpha<\min\left\{2s-(\frac{n}{q}+\frac{2s}{r}),1\right\}$, there is a positive $\delta\equiv\delta(n,s,q,r,\lambda,\alpha)<1$ such that for any local weak solution $u$ to \eqref{eq1} in $Q_{4}\subset Q_{5}$
with
\begin{equation}
\label{Base assumption}
    \sup_{Q_{4}}|u|\leq1 \quad \text{and} \quad \ITail(u;0,4)\leq 1,
\end{equation}
and any $\Tilde{A}\in\mathcal{L}_{1}(\lambda;B_{4}\times B_{4}\times[-4^{2s},0])$, if  
\begin{align}\label{delta assumption}
\|A-\Tilde{A}\|_{L^{\infty}(B_{4}\times B_{4}\times [-4^{2s},0])}\leq \delta \ \mbox{and} \ \|f\|_{L^{q,r}(Q_{4})}\leq \delta,
\end{align}
then we have $u\in C^{\alpha,\frac{\alpha}{2s}}(Q_{1})$ with the estimate
\begin{equation*}
    [u]_{C^{\alpha,\frac{\alpha}{2s}}(Q_{1})}\leq c(n,s,q,r,\lambda,\alpha).
\end{equation*}
\end{lem}
\begin{proof}
\noindent
\textbf{Step 1: Regularity at the origin.} We want to show that there is a sufficiently small $\delta>0$ such that under the assumptions as in \eqref{delta assumption}, there exist $a \equiv a(n,s,q,r,\lambda,\alpha)$, $c \equiv c(n,s,q,r,\lambda,\alpha)$ and small $\rho<\frac{1}{12}$ satisfying 
\begin{equation*}
    \|u(x,t)-a\|_{L^{\infty}(Q_{\rho^{k}})}\leq c\rho^{\alpha k}, \quad k \ge 0.
\end{equation*}
To prove the statement, it is sufficient to show the following.

\medskip

\noindent
\textbf{Claim:} There exist a positive $\rho \equiv \rho(n,s,q,r,\lambda,{\alpha})<\frac{1}{12}$, $c_{1} \equiv c_{1}(n,s,q,r,\lambda,\alpha)$ and $\{a_{i}\}_{i=-1}^{\infty}$ with $a_{-1}=0$ such that 
\begin{equation}
\label{induction1}
    \sup_{Q_{4}}|u(\rho^{i}x,\rho^{2si}t)-a_{i}|\leq \rho^{\alpha i}, \text{ } |a_{i}-a_{i-1}|\leq c_{1}\rho^{\alpha i},\quad i\geq0,
\end{equation}
and
\begin{equation}
\label{induction2}
\esup_{t\in[-4^{2s},0]}\int_{\mathbb{R}^{n}\setminus B_{4}}\frac{|u\left(\rho^{i}y,\rho^{2si}t\right)-a_{i}|}{\rho^{\alpha i}|y|^{n+2s}}\dy\leq 1,\quad i\geq0.
\end{equation}

\medskip

\noindent
\textbf{Proof of the claim.} Suppose a constant $c_{1}>0$ is given, which is to be determined later. Let $\Tilde{\alpha}=\frac{\alpha+\min\{2s,1\}}{2}$ and take a small $\rho \equiv \rho(n,s,\alpha,c_{1}) < \frac{1}{12}$ such that 
\begin{equation}
\label{rho condition}
(4\rho)^{2s}+1<2^{2s},\ \rho^{2s-\alpha}\left(1+(2c_{1}+1)\frac{\omega_{n}}{2s}\right)\leq \frac{1}{8} \ \text{ and } \ c_{1}\rho^{\Tilde{\alpha}}\leq (s-\Tilde{\alpha}/2)\frac{\rho^{\alpha}}{8(1+\omega_{n})},
\end{equation} where $w_{n}$ means the surface area of n-dimensional unit sphere. Take $\epsilon=\frac{s}{2(1+\omega_{n})}\rho^{\alpha}$, then we find a suitable $\delta=\delta(n,s,q,r,\lambda,\alpha)$ as in Lemma \ref{Basestep}.
Now we extend $\Tilde{A}$ by $A$ outside $B_{4}\times B_{4}\times [-4^{2s},0]$.
Then we get
$\Tilde{A}\in\mathcal{L}_{1}(B_{4}\times B_{4}\times(-4^{2s},0))\cap\mathcal{L}_{0}(\lambda)$ and  
\begin{equation*}
\|A-\Tilde{A}\|_{L^{\infty}(\mathbb{R}^{n}\times \mathbb{R}^{n}\times [-4^{2s},0])}\leq \delta. 
\end{equation*}

\noindent {Now we construct $a_{i}$, $i\geq0$,  as follows.}
For $i=0$, set $a_{0}=0$. Then \eqref{Base assumption} directly implies \eqref{induction1} and \eqref{induction2} . Now suppose that there is $a_{i}$ satisfying \eqref{induction1} and \eqref{induction2} for $i=0,1,\ldots,k$. Set 
\begin{equation*}
    u_{k}(x,t) \coloneqq \frac{u\left(\rho^{k}x,\rho^{2sk}t\right)-a_{k}}{\rho^{\alpha k}}, \ f_{k}(x,t) \coloneqq \rho^{(2s-\alpha)k}f\left(\rho^{k}x,\rho^{2sk}t\right),{\quad(x,t)\in \mathbb{R}^{n}\times (-4^{2s},0]},
\end{equation*}
\begin{equation*}
    A_{k} \coloneqq A\left(\rho^{k}x,\rho^{k}y,\rho^{2sk}t\right), \ \Tilde{A}_{k} \coloneqq \Tilde{A}\left(\rho^{k}x,\rho^{k}y,\rho^{2sk}t\right),\quad(x,y,t)\in\mathbb{R}^{n}\times\mathbb{R}^{n}\times\mathbb{R},
\end{equation*}
and
\begin{equation*}
    \Phi_{k}(\xi) \coloneqq \frac{1}{\rho^{\alpha k}}\Phi\left(\rho^{\alpha k}\xi\right),\quad \xi\in\mathbb{R}.
\end{equation*}
Then $u_{k}$ is a local weak solution to
\begin{equation*}
    \partial_{t}u_{k}+\mathcal{L}_{A_{k}}^{\Phi_{k}}u_{k}=f_{k} \text{ in } Q_{4}
\end{equation*}
such that
\begin{equation}\label{uk condition}
    \|u_{k}\|_{L^{\infty}(Q_{4})}\leq1\text{ and }\ITail(u_{k};0,4)\leq 1.
\end{equation}
Moreover, $A_{k},\Tilde{A}_{k}\in\mathcal{L}_{0}(\lambda)$ and $f_{k}$ satisfy 
\begin{equation*}
    \|A_{k}-\Tilde{A}_{k}\|_{L^{\infty}(\mathbb{R}^{n}\times \mathbb{R}^{n}\times[-4^{2s},0])}\leq\delta,
\end{equation*}
and
\begin{equation*}
    \|f_{k}\|_{L^{q,r}(Q_{4})}\leq \rho^{(2s-\alpha)k-\left(\frac{n}{q}+\frac{2s}{r}\right)k}\|f\|_{L^{q,r}(Q_{4\rho^{k}})}\leq \delta.
\end{equation*}
Here we used $2s-\left(\frac{n}{q}+\frac{2s}{r}\right)>\alpha$ and $\rho<1$.
By Lemma \ref{Basestep} with $\epsilon=\frac{s}{2(1+\omega_{n})}\rho^{\alpha}$, there exists a weak solution $v$ to
\begin{align*}
   \left\{
\begin{alignedat}{3}
\partial_{t}v+\mathcal{L}^{\Phi_{k}}_{\Tilde{A}_{k}}v &=0 &&\qquad \mbox{in  $Q_{2}$} \\
v&=u_{k}&&\qquad  \mbox{on  $\partial_{P}Q_{2}\cup\left(\left(\mathbb{R}^{n}\setminus B_{2}\right)\times[-2^{2s},0]\right)$} 
\end{alignedat} \right. 
\end{align*}
with 
\begin{equation}
\label{condition alpha}
    \|u_{k}-v\|_{L^{\infty}(Q_{1})}\leq \frac{s}{2(1+\omega_{n})}\rho^{\alpha}.
\end{equation}
A similar calculation as in \eqref{wk estimate} shows
\begin{equation*}
    \dashiint_{Q_{2}}v^{2}\dx\dt\leq {c}\left[\dashiint_{Q_{2}}(u_{k}-v)^{2}\dx\dt+\dashiint_{Q_{2}}u_{k}^{2}\dx\dt \right]\leq c,
\end{equation*}
where $c \equiv c(n,s,q,r,\lambda)$ is independent on $k$.
Moreover, \eqref{uk condition} implies
\begin{align*}
     \ITail(v;0,1,[-2^{2s},0]) \le \ITail(u_{k};0,1,[-2^{2s},0]) + \ITail(u_{k}-v;0,1,[-2^{2s},0]) \le c,
\end{align*}
where $c \equiv c(n,s,q,r,\lambda)$ is also independent on $k$.
Then in light of Theorem \ref{LBlem}, Lemma \ref{higher holder of homo 1} and Lemma \ref{higher holder of homo 2}, we can take $c_{1}=c_{1}(n,s,q,r,\lambda,{\alpha})$, which is independent on $k$, so that
\begin{equation}
\label{Holder for v}
     |v(0,0)|\leq c_{1} \text{ and }[v]_{C^{\frac{\Tilde{\alpha}}{2}}(\overline{Q_{1}})}\leq c_{1}.
\end{equation}
Note that 
\begin{equation*}
\begin{aligned}
\|v(x,t)-v(0,0)\|_{L^{\infty}(Q_{4\rho})}\leq c_{1}\rho^{\Tilde{\alpha}}
\end{aligned}
\end{equation*}
to discover
\begin{equation*}
\begin{aligned}
\|u_{k}(x,t)-v(0,0)\|_{L^{\infty}(Q_{4\rho})}\leq\|u_{k}(x,t)-v(x,t)\|_{L^{\infty}(Q_{4\rho})}+\|v(x,t)-v(0,0)\|_{L^{\infty}(Q_{4\rho})}\leq c\rho^{\alpha},
\end{aligned}
\end{equation*}
where we have used \eqref{rho condition} and \eqref{condition alpha}.
Take $a_{k+1}=a_{k}+v(0,0)\rho^{\alpha k}$. Then \eqref{induction1} also holds for $i=k+1$.
Furthermore, we estimate
\begin{align*}
    &\esup_{t\in[-4^{2s},0]}\int_{\mathbb{R}^{n}\setminus B_{4}}\frac{\left|u\left(\rho^{k+1}y,\rho^{2s(k+1)}t\right)-a_{k+1}\right|}{\rho^{\alpha(k+1)}|y|^{n+2s}}\dy\\
    &=\esup_{t\in[-4^{2s},0]}\int_{\mathbb{R}^{n}\setminus B_{\frac{4}{\rho}}}\frac{\left|u\left(\rho^{k+1}y,\rho^{2s(k+1)}t\right)-a_{k+1}\right|}{\rho^{\alpha(k+1)}|y|^{n+2s}}\dy\\
    &\quad+\esup_{t\in[-4^{2s},0]}\int_{B_{\frac{4}{\rho}}\setminus B_{\frac{1}{\rho}}}\frac{\left|u\left(\rho^{k+1}y,\rho^{2s(k+1)}t\right)-a_{k+1}\right|}{\rho^{\alpha(k+1)}|y|^{n+2s}}\dy\\
    &\quad+\esup_{t\in[-4^{2s},0]}\int_{B_{\frac{1}{\rho}}\setminus B_{4}}\frac{\left|u\left(\rho^{k+1}y,\rho^{2s(k+1)}t\right)-a_{k+1}\right|}{\rho^{\alpha(k+1)}|y|^{n+2s}}\dy\eqqcolon \rom{1} + \rom{2} + \rom{3}.
\end{align*}
Then using \eqref{induction1} for $i=k+1$ and \eqref{induction2} for $i=k$, we have
\begin{align*}
    \rom{1}&=\rho^{2s-\alpha}\esup_{t\in\left[-(4\rho)^{2s},0\right]}\int_{\mathbb{R}^{n}\setminus B_{4}}\frac{\left|u\left(\rho^{k}y,\rho^{2sk}t\right)-a_{k+1}\right|}{\rho^{\alpha k}|y|^{n+2s}}\dy\\
    &\leq\rho^{2s-\alpha}\left(\esup_{t\in\left[-4^{2s},0\right]}\int_{\mathbb{R}^{n}\setminus B_{4}}\frac{\left|u\left(\rho^{k}y,\rho^{2sk}t\right)-a_{k}\right|}{\rho^{\alpha k}|y|^{n+2s}}\dy+\esup_{t\in\left[-4^{2s},0\right]}\int_{\mathbb{R}^{n}\setminus B_{4}}\frac{|a_{k}-a_{k+1}|}{\rho^{\alpha k}|y|^{n+2s}}\dy\right)\\
    &\leq\rho^{2s-\alpha}\left(1+c_{1}\frac{\omega_{n}}{2s}\right),
\end{align*}
and 
\begin{align*}
\rom{2}&\leq\rho^{2s-\alpha}\left(\esup_{t\in\left[-4^{2s},0\right]}\int_{B_{4}\setminus B_{1}}\frac{\left|u\left(\rho^{k}y,\rho^{2sk}t\right)-a_{k}\right|}{\rho^{\alpha k}|y|^{n+2s}}\dy+\esup_{t\in\left[-4^{2s},0\right]}\int_{B_{4}\setminus B_{1}}\frac{|a_{k}-a_{k+1}|}{\rho^{\alpha k}|y|^{n+2s}}\dy\right)\\
&\leq \rho^{2s-\alpha}(c_{1}+1)\frac{\omega_{n}}{2s}.
\end{align*}
In addition, using \eqref{rho condition}, \eqref{condition alpha} and \eqref{Holder for v}, we obtain
\begin{align*}
    \rom{3}&\leq\rho^{2s-\alpha}\left(\esup_{t\in\left[-(4\rho)^{2s},0\right]}\int_{B_{1}\setminus B_{4\rho}}\frac{\left|u\left(\rho^{k}y,\rho^{2sk}t\right)-\left(a_{k}+\rho^{\alpha k}v(y,t)\right)\right|}{\rho^{\alpha k}|y|^{n+2s}}\dy\right)\\
    &\quad+\rho^{2s-\alpha}\left(\esup_{t\in\left[-(4\rho)^{2s},0\right]}\int_{B_{1}\setminus B_{4\rho}}\frac{\left|\rho^{\alpha k}v(y,t)-\rho^{\alpha k}v(0,0)\right|}{\rho^{\alpha k}|y|^{n+2s}}\dy\right)\\
    &\leq \rho^{2s-\alpha}\left(\esup_{t\in\left[-(4\rho)^{2s},0\right]}\int_{B_{1}\setminus B_{4\rho}}\frac{\left|u_{k}(y,t)-v(y,t)\right|}{|y|^{n+2s}}\dy + \esup_{t\in\left[-(4\rho)^{2s},0\right]}\int_{B_{1}\setminus B_{4\rho}}\frac{c_{1}}{|y|^{n+2s-\Tilde{\alpha}}}\dy\right)\\
    &\leq \frac{\omega_{n}}{4^{2s+1}(1+\omega_{n})}+c_{1}\frac{\rho^{\Tilde{\alpha}-\alpha}\omega_{n}}{2s-\Tilde{\alpha}}\leq\frac{3}{8}.
 \end{align*}
We combine the estimates $\rom{1}$, $\rom{2}$, $\rom{3}$ and \eqref{rho condition} to find
\begin{equation*}
    \esup_{t\in[-4^{2s},0]}\int_{\mathbb{R}^{n}\setminus B_{4}}\frac{\left|u\left(\rho^{k+1}y,\rho^{2s(k+1)}t\right)-a_{k+1}\right|}{\rho^{(k+1)\alpha}|y|^{n+2s}}\dy\\
\leq \rho^{2s-\alpha}\left(1+(2c_{1}+1)\frac{\omega_{n}}{2s}\right)+\frac{3}{8}\leq 1.
\end{equation*}
This completes the proof of the claim.

\medskip

By \eqref{induction1}, $a_{i}$ converges to some $a\in\mathbb{R}$ and
\begin{equation*}
\|u(x,t)-a\|_{L^{\infty}(Q_{\rho^{i}})}\leq c\rho^{\alpha i},\quad i \ge 0.
\end{equation*}
This finishes Step 1.

\medskip

\noindent
\textbf{Step 2: Regularity in $Q_{1}$.} Let $z_{0}\in Q_{1}$. Then $Q_{4\rho}(z_{0})\Subset Q_{2}$ from \eqref{rho condition}. Now set 
\begin{equation*}
    \Tilde{u}(x,t) \coloneqq u\left(\rho x+x_{0},\rho^{2s} t+t_{0}\right),\quad f_{1}(x,t) \coloneqq \rho^{2s}f\left(\rho x+x_{0},\rho^{2s} t+t_{0}\right),\quad(x,t)\in Q_{4},
\end{equation*}
and
\begin{align*}
    A_{1}(x,y,t) \coloneqq A\left(\rho x+x_{0},\rho y+y_{0},\rho^{2s} t+t_{0}\right),\quad (x,y,t)\in\mathbb{R}^{2n}\times\mathbb{R}.
\end{align*}
Then $\Tilde{u}$ is a local weak solution to 
\[\partial_{t}\Tilde{u}+\mathcal{L}^{\Phi}_{A_{1}}\Tilde{u}=f_{1}\text{ in } Q_{4}.\]
Moreover, \eqref{rho condition} implies
\[\sup_{Q_{4}}|\Tilde{u}|\leq1 \text{ and }\ITail(\Tilde{u};0,4)\leq 1.\]
From the result of Step 1, there is some constant $a\in\mathbb{R}$ such that 
\begin{align*}
    \|\Tilde{u}(x,t)-a\|_{L^{\infty}(Q_{\rho^{k}})}\leq c\rho^{\alpha k}, \quad k \ge 0,
\end{align*}
which implies 
\begin{align}
\label{holder for other points}
    \|u(x,t)-a\|_{L^{\infty}(Q_{\rho^{k+1}}(z_{0}))}\leq c\rho^{\alpha(k+1)}, \quad k \ge 0.
\end{align}
{Let $(x_{0},t_{0}),(x_{1},t_{1})\in Q_{1}$. We may assume $t_{0}<t_{1}$. Then there is a nonnegative integer $k\geq0$ such that $(x_{0},t_{0})\in Q_{\rho^{k}}(x_{1},t_{1})\setminus Q_{\rho^{k+1}}(x_{1},t_{1})$, which implies that
\begin{equation}
\label{distx1x2}
\rho^{k+1}<|x_{0}-x_{1}|+|t_{0}-t_{1}|^{\frac{1}{2s}}\leq 2\rho^{k}.
\end{equation} From \eqref{holder for other points}, there is a constant $a_{1}\in\mathbb{R}$ such that
\begin{equation*}
    \|u(x,t)-a_{1}\|_{L^{\infty}(Q_{\rho^{k}}(x_{1},t_{1}))}\leq c\rho^{\alpha k}.
\end{equation*}
Therefore, we have 
\begin{equation*}
    |u(x_{0},t_{0})-u(x_{1},t_{1})|\leq |u(x_{0},t_{0})-a_{1}|+|a_{1}-u(x_{1},t_{1})|\leq 2\|u(x,t)-a_{1}\|_{L^{\infty}(Q_{\rho^{k}}(x_{1},t_{1}))}\leq c\rho^{\alpha k}.
\end{equation*}
We combine the above inequality and \eqref{distx1x2} to see that
\begin{equation}
\label{lasthi}
    |u(x_{0},t_{0})-u(x_{1},t_{1})|\leq c\left(|x_{0}-x_{1}|+|t_{0}-t_{1}|^{\frac{1}{2s}}\right)^{\alpha}.
\end{equation}
Since we have proved \eqref{lasthi} for any $(x_{0},t_{0})$ and $(x_{1},t_{1})\in Q_{1}$, we conclude that
$u\in C^{\alpha,\frac{\alpha}{2s}}(Q_{1})$ with the estimate 
\begin{equation*}
    [u]_{C^{\alpha,\frac{\alpha}{2s}}(Q_{1})}\leq c.
\end{equation*}}
\end{proof}
Now we are ready to prove the main result.

\noindent
\textbf{Proof of Theorem \ref{Holder}.} 
Let $\delta>0$ be a small number to be determined later, depending on $\alpha$.
Let $Q_{\rho_{0}}(z_{0})\Subset\Omega_{T}$. Set 
{\begin{equation}
\label{rhocondforhod}
    \rho=\min\left\{\frac{\rho_{0}}{32},\frac{1}{8}\left(\left(\frac{3}{4}\right)^{2s}-\left(\frac{1}{2}\right)^{2s}\right)^{\frac{1}{2s}}\rho_{0}\right\}
\end{equation}}so that $Q_{4\rho}(\Tilde{z})\Subset Q_{\frac{3\rho_{0}}{4}}(z_{0})$ for any $\Tilde{z}\in Q_{\frac{\rho_{0}}{2}}(z_{0})$. Fix $\Tilde{z}\in Q_{\frac{\rho_{0}}{2}}(z_{0})$. According to the assumption \eqref{assumption closedness of A}, there exist $0<\Tilde{\rho}_{\Tilde{z}}\leq\min\{\rho,\frac{\rho_{\Tilde{z}}}{4}\}$ and $\Tilde{A}_{\Tilde{z}}\in \mathcal{L}_{1}\left(\lambda;B_{\rho_{\Tilde{z}}}(\Tilde{x})\times B_{\rho_{\Tilde{z}}}(\Tilde{x})\times [\Tilde{t}-\rho_{\Tilde{z}}^{2s},\Tilde{t}]\right)$ such that
\begin{equation*}
\|\Tilde{A}_{\Tilde{z}}-A\|_{L^{\infty}\left(B_{4\Tilde{\rho}_{\Tilde{z}}}(\Tilde{x})\times B_{4\Tilde{\rho}_{\Tilde{z}}}(\Tilde{x})\times \left[\Tilde{t}-\left(4\Tilde{\rho}_{\Tilde{z}}\right)^{2s},\Tilde{t}\right]\right)}\leq \delta.
\end{equation*}
We write
\[M_{0}=\|u\|_{L^{\infty}(Q_{4\Tilde{\rho}_{\Tilde{z}}}(\Tilde{z}))}+\ITail\left(u;\Tilde{z},4\Tilde{\rho}_{\Tilde{z}}\right)+\frac{\left(4\Tilde{\rho}_{\Tilde{z}}\right)^{2s-\left(\frac{n}{q}+\frac{2s}{r}\right)}}{\delta}\|f\|_{L^{q,r}(Q_{4\Tilde{\rho}_{\Tilde{z}}}(\Tilde{z}))}\]
to define
\begin{equation*}
    \Tilde{u}(x,t)=\frac{u\left(\Tilde{\rho}_{\Tilde{z}}x+\Tilde{x}, \left(\Tilde{\rho}_{\Tilde{z}}\right)^{2s}t+\Tilde{t}\right)}{M_{0}},\quad f_{1}(x,t)=\frac{\left(\Tilde{\rho}_{\Tilde{z}}\right)^{2s}f\left(\Tilde{\rho}_{\Tilde{z}}x+\Tilde{x},\left(\Tilde{\rho}_{\Tilde{z}}\right)^{2s}t+\Tilde{t}\right)}{M_{0}},\quad (x,t)\in Q_{4},
\end{equation*}
\begin{equation*}
    A_{1}(x,y,t)=A\left(\Tilde{\rho}_{\Tilde{z}}x+\Tilde{x},\Tilde{\rho}_{\Tilde{z}}y+\Tilde{x},\left(\Tilde{\rho}_{\Tilde{z}}\right)^{2s}t+\Tilde{t}\right),\quad (x,y,t)\in\mathbb{R}^{2n}\times\mathbb{R},
\end{equation*}
\begin{equation*}
    \Tilde{A}_{\Tilde{z},1}(x,y,t)=\Tilde{A}\left(\Tilde{\rho}_{\Tilde{z}}x+\Tilde{x},\Tilde{\rho}_{\Tilde{z}}y+\Tilde{x},\left(\Tilde{\rho}_{\Tilde{z}}\right)^{2s}t+\Tilde{t}\right),\quad (x,y,t)\in\mathbb{R}^{2n}\times\mathbb{R},
\end{equation*}
and
\begin{equation*}
    \Phi_{1}(\xi)=\frac{\Phi(M_{0}\xi)}{M_{0}},\quad \xi\in\mathbb{R}.
\end{equation*} 
Then $\Tilde{u}$ is a local weak solution to
\begin{equation*}
    \partial_{t}\Tilde{u}+\mathcal{L}_{A_{1}}^{\Phi_{1}}\Tilde{u}=f_{1} \text{ in }{Q_{4}}
\end{equation*}
with 
\begin{equation*}
    \sup_{Q_{4}}|\Tilde{u}|\leq1 \text{ and }\ITail(\Tilde{u};0,4)\leq 1,
\end{equation*}
where $\Phi_{1}$ satisfies \eqref{Phi condition}.
Also we directly check
\begin{align*}
\|\Tilde{A}_{\Tilde{z},1}-A_{1}\|_{L^{\infty}\left(B_{4} \times B_{4} \times \left[-4^{2s},0 \right]\right)}\leq \delta \ \mbox{and} \ \|f\|_{L^{q,r}(Q_{4})} \le \delta.
\end{align*}
We are now under the assumptions and settings in Lemma \ref{almost main}, which implies that there is a positive constant $\delta(n,s,q,r,\lambda,\alpha)$ such that
\begin{equation*}
 [\Tilde{u}]_{C^{\alpha,\frac{\alpha}{2s}}(Q_{1})}\leq c(n,s,q,r,\lambda,\alpha).
\end{equation*}
Therefore, scaling back, we have
\begin{align*}
[u]_{C^{\alpha,\frac{\alpha}{2s}}(Q_{\Tilde{\rho}_{\Tilde{z}}}(\Tilde{z}))}&\leq \frac{c}{\Tilde{\rho}_{\Tilde{z}}^{\alpha}}\Bigg(\|u\|_{L^{\infty}(Q_{4\Tilde{\rho}_{\Tilde{z}}}(\Tilde{z}))}+\ITail(u;\Tilde{z}, 4\Tilde{\rho}_{\Tilde{z}})+(\Tilde{\rho}_{\Tilde{z}})^{2s-\left(\frac{n}{q}+\frac{2s}{r}\right)}\|f\|_{L^{q,r}(Q_{4\Tilde{\rho}_{\Tilde{z}}}(\Tilde{z}))}\Bigg).
\end{align*}
Let us see the second term of the right hand side. For fixed $t \in \left[t_{0}-(\frac{3}{4}\rho_{0})^{2s},t_{0}\right]$,

\begin{align}\label{main tail}
\begin{split}
\int_{\mathbb{R}^{n}\setminus B_{4\Tilde{\rho}_{\Tilde{z}}}(\Tilde{x})}\frac{|u(y,t)|}{|\Tilde{x}-y|^{n+2s}}\dy
&\le \int_{\mathbb{R}^{n}\setminus B_{\frac{3\rho_{0}}{4}}(x_{0})}\frac{|u(y,t)|}{|\Tilde{x}-y|^{n+2s}}\dy
+ \int_{B_{\frac{3\rho_{0}}{4}}(x_{0}) \setminus B_{4\Tilde{\rho}_{\Tilde{z}}}(\Tilde{x})}\frac{|u(y,t)|}{|\Tilde{x}-y|^{n+2s}}\dy\\
&\le \int_{\mathbb{R}^{n}\setminus B_{\frac{3\rho_{0}}{4}}(x_{0})}\frac{|u(y,t)|}{|x_{0}-y|^{n+2s}}\dy
+ c\|u\|_{L^{\infty}(Q_{\frac{3\rho_{0}}{4}}(z_{0}))}\Tilde{\rho}_{\Tilde{z}}^{-2s},
\end{split}
\end{align}
where we have used 
\begin{equation*}
    |\Tilde{x}-y|\geq|x_{0}-y|-|\Tilde{x}-x_{0}|{\geq} c|x_{0}-y|,\quad y\in\mathbb{R}^{n}\setminus Q_{\frac{3\rho_{0}}{4}}(z_{0}).
\end{equation*}
Thus
\begin{align*}
\ITail(u;\Tilde{z}, 4\Tilde{\rho}_{\Tilde{z}}) \le c\|u\|_{L^{\infty}(Q_{\frac{3\rho_{0}}{4}}(z_{0}))}+c\rho_{0}^{2s}\esup_{t\in\left[t_{0}-(\frac{3}{4}\rho_{0})^{2s},t_{0}\right]}\int_{\mathbb{R}^{n}\setminus B_{\frac{3\rho_{0}}{4}}(x_{0})}\frac{|u(y,t)|}{|x_{0}-y|^{n+2s}}\dy.
\end{align*}
In turn, we use \eqref{main tail}$, Q_{4\Tilde{\rho}_{\Tilde{z}}}(\Tilde{z})\Subset Q_{\frac{3\rho_{0}}{4}}(z_{0})$ and Theorem \ref{LBlem} with a simple modification to derive
\begin{align}\label{covering holder}
\begin{split}
    [u]_{C^{\alpha,\frac{\alpha}{2s}}(Q_{\Tilde{\rho}_{\Tilde{z}}}(\Tilde{z}))}&\leq\frac{c}{\Tilde{\rho}_{\Tilde{z}}^{\alpha}}\Bigg(\|u\|_{L^{\infty}(Q_{\frac{3\rho_{0}}{4}}(z_{0}))}+\rho_{0}^{2s}\esup_{t\in\left[t_{0}-(\frac{3}{4}\rho_{0})^{2s},t_{0}\right]}\int_{\mathbb{R}^{n}\setminus B_{\frac{3\rho_{0}}{4}}(x_{0})}\frac{|u(y,t)|}{|x_{0}-y|^{n+2s}}\dy\\
    &\qquad+\rho_{\Tilde{z}}^{2s-\left(\frac{n}{q}+\frac{2s}{r}\right)}\|f\|_{L^{q,r}(Q_{\frac{3}{4}\rho_{0}}(\Tilde{z}))}\Bigg)\\
    &\leq \frac{c}{\Tilde{\rho}_{\Tilde{z}}^{\alpha}}\Bigg(\left(\dashiint_{Q_{\rho_{0}}(z_{0})}u^{2}(x,t)\dx\dt\right)^{\frac{1}{2}}+\rho_{0}^{2s}\esup_{t\in\left[t_{0}-(\rho_{0})^{2s},t_{0}\right]}\int_{\mathbb{R}^{n}\setminus B_{\frac{\rho_{0}}{2}}(x_{0})}\frac{|u(y,t)|}{|x_{0}-y|^{n+2s}}\dy\\
    &\qquad+\rho_{0}^{2s-\left(\frac{n}{q}+\frac{2s}{r}\right)}\|f\|_{L^{q,r}(Q_{\rho_{0}}(z_{0}))}\Bigg).
\end{split}
\end{align}
 Since $Q_{\frac{\rho_{0}}{2}}(z_{0})$ is compact, there is a finite subcover $\left\{Q_{\frac{\Tilde{\rho}_{\Tilde{z}_{i}}}{4}}(\Tilde{z}_{i})\right\}_{i=1}^{N}$ of $Q_{\frac{\rho_{0}}{2}}(z_{0})$ and we choose  
\begin{equation}\label{rhomincond}
    \rho_{\min}\coloneqq\min_{i=1,2,\ldots,N}\Tilde{\rho}_{\Tilde{z}_{i}}>0.
\end{equation}
Fix $z_{1}=(x_{1},t_{1}),z_{2}=(x_{2},t_{2})\in Q_{\frac{\rho_{0}}{2}}(z_{0})$. 
{Let us assume 
\begin{equation}
\label{distx1x2t1t2}
\max\left\{4|x_{1}-x_{2}|,\left(\frac{4^{2s}}{2^{2s}-1}|t_{1}-t_{2}|\right)^{\frac{1}{2s}}\right\}<\rho_{\min}.
\end{equation} We note that there exists $i$ such that $z_{1}\in Q_{\frac{\Tilde{\rho}_{\Tilde{z}_{i}}}{4}}(\Tilde{z}_{i})$. We are going to show $z_{2}\in Q_{\frac{\Tilde{\rho}_{\Tilde{z}_{i}}}{2}}(\Tilde{z}_{i})$. By \eqref{distx1x2t1t2} and \eqref{rhomincond}, we have
\begin{equation*}
    |x_{2}-\tilde{x}_{i}|\leq |x_{2}-x_{1}|+|x_{1}-\tilde{x}_{i}|<\frac{\rho_{\min}}{4}+\frac{\Tilde{\rho}_{\Tilde{z}_{i}}}{4}\leq \frac{\Tilde{\rho}_{\Tilde{z}_{i}}}{2}
\end{equation*}
and
\begin{equation*}
    |t_{2}-\tilde{t}_{i}|\leq |t_{2}-t_{1}|+|t_{1}-\tilde{t}_{i}|<\frac{(2^{2s}-1)\rho_{\min}^{2s}}{4^{2s}}+\left(\frac{\Tilde{\rho}_{\Tilde{z}_{i}}}{4}\right)^{2s}\leq \left(\frac{\Tilde{\rho}_{\Tilde{z}_{i}}}{2}\right)^{2s},
\end{equation*}
which implies $z_{2}\in  Q_{\frac{\Tilde{\rho}_{\Tilde{z}_{i}}}{2}}(\Tilde{z}_{i}) $.}
Thus \eqref{covering holder} yields
\begin{equation}
\label{holder 1}    
\begin{aligned}
    \frac{|u(z_{1})-u(z_{2})|}{|x_{1}-x_{2}|^{\alpha}+|t_{1}-t_{2}|^{\frac{\alpha}{2s}}}\leq c\Bigg(&\left(\dashiint_{Q_{\rho_{0}}(z_{0})}u^{2}(x,t)\dx\dt\right)^{\frac{1}{2}}+\ITail(u;z_{0},\rho_{0}/2,\rho_{0}^{2s})\\
    &+\rho_{0}^{2s-\left(\frac{n}{q}+\frac{2s}{r}\right)}\|f\|_{L^{q,r}(Q_{\rho_{0}}(z_{0}))}\Bigg).
\end{aligned}
\end{equation}
 On the other hand, if $\max\left\{4|x_{1}-x_{2}|,\left(\frac{4^{2s}}{2^{2s}-1}|t_{1}-t_{2}|\right)^{\frac{1}{2s}}\right\}\geq\rho_{\min}$, then we deduce that
\begin{equation}
\label{holder 2}    
\begin{aligned}
    \frac{|u(z_{1})-u(z_{2})|}{|x_{1}-x_{2}|^{\alpha}+|t_{1}-t_{2}|^{\frac{\alpha}{2s}}}&\leq c\|u\|_{L^{\infty}(Q_{\frac{\rho_{0}}{2}}(z_{0}))}\\
    &\le c\Bigg(\left(\dashiint_{Q_{\rho_{0}}(z_{0})}u^{2}(x,t)\dx\dt\right)^{\frac{1}{2}}+\ITail(u;z_{0},\rho_{0}/2,\rho_{0}^{2s})\\
    &\qquad +\rho_{0}^{2s-\left(\frac{n}{q}+\frac{2s}{r}\right)}\|f\|_{L^{q,r}(Q_{\rho_{0}}(z_{0}))}\Bigg).
\end{aligned}
\end{equation}
From \eqref{holder 1} and \eqref{holder 2}, we see that
\begin{align*}
    [u]_{C^{\alpha,\frac{\alpha}{2s}}(Q_{\frac{\rho_{0}}{2}}(z_{0}))}&\leq c\Bigg(\rho_{0}^{-\frac{n+2s}{2}}\|u\|_{L^{2}(Q_{\rho_{0}}(z_{0}))}+\ITail(u;z_{0},\rho_{0}/2,\rho_{0}^{2s})\\
    &\qquad\quad+\rho_{0}^{2s-\left(\frac{n}{q}+\frac{2s}{r}\right)}\|f\|_{L^{q,r}(Q_{\rho_{0}}(z_{0}))}\Bigg).
\end{align*}
Since $Q_{\rho_{0}}(z_{0})$ is chosen arbitrary, we have $u\in C_{\loc}^{\alpha,\frac{\alpha}{2s}}(\Omega_{T})$.
\qed

{In particular, we obtain the following higher H\"older regularity result with a more refined estimate \eqref{holderrefinedestimate} if the kernel coefficient $A$ is H\"older continuous.
\begin{cor}
\label{main theorem}
Let $u$ be a local weak solution to \eqref{eq1} with \eqref{forcing term}.
Suppose that a kernel coefficient $A$ satisfies
\begin{align}
\label{holdercontiassump}
    \frac{|A(x,y,t)-A(x',y',t')|}{\left(|(x,y)-(x',y')|+|t-t'|^{\frac{1}{2s}}\right)^{\beta}}\leq L, \quad x,x',y,y'\in\Omega\text{ and }t,t'\in(0,T),
\end{align}
for some constants $\beta\in(0,1)$ and $L>0$.
Then $u\in C^{\alpha,\frac{\alpha}{2s}}_{\loc}(\Omega_{T})$ for any $\alpha$ satisfying \eqref{alphacondthm}. In particular, there is a sufficiently small $\rho_{\beta}=\rho_{\beta}(n,s,q,r,\lambda,\alpha,\beta,L)$ such that for any $Q_{\rho_{0}}(z_{0})\Subset\Omega_{T}$ with $\rho_{0}\leq\rho_{\beta}$ , we have
\begin{equation}
\label{holderrefinedestimate}
\begin{aligned}
[u]_{C^{\alpha,\frac{\alpha}{2s}}(Q_{\rho_{0}/2}(z_{0}))}&\leq \frac{c}{\rho_{0}^{\alpha}}\Bigg(\rho_{0}^{-\frac{n+2s}{2}}\|u\|_{L^{2}(Q_{\rho_{0}}(z_{0}))}+\ITail(u;z_{0},\rho_{0}/2,\rho_{0}^{2s})\\
&\qquad\quad+\rho_{0}^{2s-\left(\frac{n}{q}+\frac{2s}{r}\right)}\|f\|_{L^{q,r}(Q_{\rho_{0}}(z_{0}))}\Bigg),
\end{aligned}
\end{equation}
where $c \equiv c(n,s,q,r,\lambda,\alpha)$.
\end{cor} 
\begin{proof}
Fix $\alpha\in\left(0,\min\left\{2s-\left(\frac{n}{q}+\frac{2s}{r}\right),1\right\}\right)$. Let $\delta=\delta(n,s,q,r,\lambda,\alpha)>0$ be determined in Lemma \ref{almost main}. By \eqref{holdercontiassump}, there is a sufficiently small $\rho_{\beta}=\rho_{\beta}(n,s,q,r,\lambda,\alpha,\beta,L)>0$ such that if $(x,y,t),(x',y',t')\in \Omega\times\Omega\times(0,T)$ satisfy 
\begin{equation*}
\left(|(x,y)-(x',y')|+|t-t'|^{\frac{1}{2s}}\right)\leq \rho_{\beta},
\end{equation*}
then
\begin{equation*}
    |A(x,y,t)-A(x',y',t')|\leq \delta.
\end{equation*}
We now take $Q_{\rho_{0}}(z_{0})\Subset\Omega_{T}$ with $\rho_{0}\leq\rho_{\beta}$. Then for any $\tilde{z}\in Q_{\rho_{0}}(z_{0})$, we see that
\begin{equation*}
\|\Tilde{A}_{\Tilde{z}}-A\|_{L^{\infty}\left(B_{4\rho}(\Tilde{x})\times B_{4\rho}(\Tilde{x})\times \left[\Tilde{t}-(4\rho)^{2s},\Tilde{t}\right]\right)}\leq \delta,
\end{equation*} 
where the constant $\rho>0$ is given in \eqref{rhocondforhod} and we take
\begin{align*}
    \Tilde{A}_{\Tilde{z}}(x,y,t)=A(\Tilde{x},\Tilde{x},\Tilde{t}),\quad (x,y,t)\in B_{4\rho}(\Tilde{x})\times B_{4\rho}(\Tilde{x})\times\left[\Tilde{t}-(4\rho)^{2s},\Tilde{t}\right]
\end{align*}
as in (1) of Remark \ref{rmkkernel}. Thus, by following the same lines as in the proof of \eqref{covering holder} with $\tilde{\rho}_{\tilde{z}}$ there, replaced by $\rho$, we have 
\begin{equation*}
\begin{aligned}
    [u]_{C^{\alpha,\frac{\alpha}{2s}}(Q_{\rho}(\tilde{z}))}&\leq\frac{c}{{\rho}^{\alpha}}\Bigg(\left(\dashiint_{Q_{\rho_{0}}(z_{0})}u^{2}\dx\dt\right)^{\frac{1}{2}}+\rho_{0}^{2s}\esup_{t\in\left[t_{0}-(\rho_{0})^{2s},t_{0}\right]}\int_{\mathbb{R}^{n}\setminus B_{\frac{\rho_{0}}{2}}(x_{0})}\frac{|u(y,t)|}{|x_{0}-y|^{n+2s}}\dy\\
    &\qquad+\rho_{0}^{2s-\left(\frac{n}{q}+\frac{2s}{r}\right)}\|f\|_{L^{q,r}(Q_{\rho_{0}}(z_{0}))}\Bigg)
\end{aligned}
\end{equation*}
for some constant $c=c(n,s,q,r,\lambda,\alpha)$.
Using the definition of $\rho$ given in \eqref{rhocondforhod}, we get
\begin{equation*}
\begin{aligned}
    [u]_{C^{\alpha,\frac{\alpha}{2s}}(Q_{\rho}(\tilde{z}))}&\leq\frac{c}{{\rho_{0}}^{\alpha}}\Bigg(\left(\dashiint_{Q_{\rho_{0}}(z_{0})}u^{2}\dx\dt\right)^{\frac{1}{2}}+\ITail(u;z_{0},\rho_{0}/2,\rho_{0}^{2s})\\
    &\qquad+\rho_{0}^{2s-\left(\frac{n}{q}+\frac{2s}{r}\right)}\|f\|_{L^{q,r}(Q_{\rho_{0}}(z_{0}))}\Bigg)
\end{aligned}
\end{equation*}
for some constant $c=c(n,s,q,r,\lambda,\alpha)$. We now use the standard covering argument to conclude \eqref{holderrefinedestimate}.
\end{proof}}

\appendix
\section{Existence and uniqueness of an initial and boundary value problem}
\label{Appendix}
In this section, we prove the existence and uniqueness of a weak solution to \eqref{eq1} with boundary conditions. 
When $r \ge 2$, this is proved in \cite[Appendix A]{BLS}. Here we extend the argument to the case $r>1$ by regularizing the nonhomogeneous term.
To this end, let us introduce appropriate function spaces.
Let $\Omega$ and $\Omega'$ be bounded open sets with $\Omega\Subset\Omega'\subset\mathbb{R}^{n}$.
We introduce a space as in \cite{KKP},
\begin{equation*}
    X_{\phi}^{s,2}(\Omega,\Omega')=\left\{v\in W^{s,2}(\Omega')\cap L^{1}_{2s}(\mathbb{R}^{n})\ ; \ v=\phi \text{ on }\mathbb{R}^{n}\setminus\Omega\right\}\quad \text{for }\phi\in L^{1}_{2s}(\mathbb{R}^{n}).
\end{equation*}
Assume that functions $u_{0}$, $f$ and $\zeta$ satisfy the followings,
\begin{equation}
\label{Data condition of IVP}
    \begin{aligned}
    &u_{0}\in L^{2}(\Omega),\\
    &f\in L^{q,r}(\Omega_{T}) \text{ for }\frac{n}{2qs}+\frac{1}{r}\leq 1+\frac{n}{4s}, \\
    &\zeta\in L^{2}(0,T;W^{s,2}(\Omega'))\cap L^{2}(0,T;L^{1}_{2s}(\mathbb{R}^{n}))\text{ and }\partial_{t}\zeta\in \left(L^{2}(0,T;W^{s,2}(\Omega')\right)^{*}.
\end{aligned}
\end{equation}
We say that $u\in L^{2}(0,T;W^{s,2}(\Omega'))\cap L^{2}(0,T;L^{1}_{2s}(\mathbb{R}^{n}))\cap C([0,T];L^{2}(\Omega))$ is a weak solution to 
\begin{equation}
\label{IVP}
\left\{
\begin{alignedat}{3}
\partial_{t}u+\mathcal{L}^{\Phi}_{A}u&= f&&\qquad \mbox{in  $\Omega\times(0,T)$} \\
u&=\zeta&&\qquad  \mbox{on $\mathbb{R}^{n}\setminus\Omega\times[0,T]$} \\
u(\cdot,0)&=u_{0} &&\qquad \mbox{on  $\Omega$},
\end{alignedat} \right.
\end{equation}
if it satisfies the following three conditions:
\begin{enumerate}
\item $u(\cdot,t)\in X^{s,p}_{\zeta(\cdot,t)}(\Omega,\Omega')$ for almost every $t\in I$.
\item $\lim\limits_{t\to0}\|u(\cdot,t)-u_{0}\|_{L^{2}(\Omega)}=0$.
\item For every $\phi\in L^{2}(t_{1},t_{2};X_{0}^{s,2}(\Omega,\Omega'))\cap C^{1}([t_{1},t_{2}];L^{2}(\Omega))$, we have
\begin{equation*}
\begin{aligned}
&-\int_{t_{1}}^{t_{2}}\int_{\Omega}u(x,t)\partial_{t}\phi(x,t)\dx\dt+\int_{t_{1}}^{t_{2}}\int_{\mathbb{R}^{n}}\int_{\mathbb{R}^{n}}\frac{\Phi(u(x,t)-u(y,t))}{|x-y|^{n+2s}}(\phi(x,t)-\phi(y,t))\dx\dy\dt\\
&\quad=\int_{t_{1}}^{t_{2}}\int_{\Omega}f(x,t)\phi(x,t)\dx\dt-\int_{\Omega}u(x,t)\phi(x,t)\dx\Bigg\rvert_{t=t_{1}}^{t=t_{2}}
\end{aligned}
\end{equation*}
whenever $J\coloneqq[t_{1},t_{2}]\Subset I$.
\end{enumerate}
\begin{lem}
Under the assumptions \eqref{Data condition of IVP}, there is a unique weak solution 
\[u\in L^{2}(0,T;W^{s,2}(\Omega'))\cap L^{2}(0,T;L^{1}_{2s}(\mathbb{R}^{n}))\cap C([0,T];L^{2}(\Omega))\]
to \eqref{IVP}. In particular, if $\zeta\in L^{\infty}(0,T;L^{1}_{2s}(\mathbb{R}^{n}))$, then $u\in L^{\infty}(0,T;L^{1}_{2s}(\mathbb{R}^{n}))$.
\end{lem}
\begin{proof}
Take a sequence $\{f_{n}\}_{n\in\mathbb{N}}\subset C([0,T];L^{q}(\Omega))$ such that $\lim_{n\to\infty}\|f_{n}-f\|_{L^{q,r}(\Omega_{T})}=0$.
Since $\{f_{n}\}_{n\in\mathbb{N}}\subset L^{2}(0,T;L^{q}(\Omega))\subset L^{2}\left(0,T;\left(X_{0}^{s,2}(\Omega,\Omega')\right)^{*}\right)$, we can use \cite[Theorem A.3]{BLS} and \cite[Proposition 4.1]{Nexist} with simple modifications to handle $\Phi(\xi)$. Thus we have a unique weak solution 
\[u_{n}\in L^{2}(0,T;W^{s,2}(\Omega'))\cap L^{2}(0,T;L^{1}_{2s}(\mathbb{R}^{n}))\cap C([0,T];L^{2}(\Omega))\] to
\begin{equation*}
\left\{
\begin{alignedat}{3}
\partial_{t}u_{n}+\mathcal{L}^{\Phi}_{A}u_{n}&=f_{n}&&\qquad \mbox{in  $\Omega\times(0,T)$} \\
u_{n} &=\zeta&&\qquad  \mbox{on $\left(\mathbb{R}^{n}\setminus\Omega\right)\times(0,T)$} \\
u_{n}(\cdot,0)&=u_{0} &&\qquad \mbox{on  $\Omega$}.
\end{alignedat} \right.
\end{equation*}
We next show that $\{u_{n}-\zeta\}_{n\in\mathbb{N}}$ is a Cauchy sequence in $L^{2}(0,T;W^{s,2}(\mathbb{R}^{n}))\cap C([0,T];L^{2}(\Omega))$. Then $u_{n}-u_{m}\in L^{2}(0,T;W_{0}^{s,2}(\Omega))\cap C([0,T];L^{2}(\Omega))$ and 
\begin{equation}\label{Appendix eq}
    \partial_{t}(u_{n}-u_{m})+\mathcal{L}_{A}^{\Phi}u_{n}-\mathcal{L}_{A}^{\Phi}u_{m}=f_{n}-f_{m} \quad\text{ in }\Omega_{T}.
\end{equation}
Using $u_{n}-u_{m}$ as a test function to \eqref{Appendix eq}, we deduce that 
\begin{align*}
    &\sup_{t\in[0,T]}\int_{\Omega}(u_{n}-u_{m})^{2}(x,t)\dx+\int_{0}^{T}[(u_{n}-u_{m})(\cdot,t)]^{2}_{W^{s,2}(\mathbb{R}^{n})}\dt\\
&\quad \leq c\int_{0}^{T}\int_{\Omega}(f_{n}-f_{m})(u_{n}-u_{m})\dx\dt
\end{align*}
for some constant $c \equiv c(\lambda)$. Then using Lemma \ref{embedSF} and Young's inequality, we have
\begin{align*}
    \sup_{t\in[0,T]}\int_{\Omega}(u_{n}-u_{m})^{2}(x,t)\dx+\int_{0}^{T}[(u_{n}-u_{m})(\cdot,t)]^{2}_{W^{s,2}(\mathbb{R}^{n})}\dt\leq c\|f_{n}-f_{m}\|^{2}_{L^{q,r}(\Omega_{T})}.
\end{align*}
This implies that $\{u_{n}-\zeta\}_{n\in\mathbb{N}}$ is a Cauchy sequence in $L^{2}(0,T;W^{s,2}(\mathbb{R}^{n}))\cap C([0,T];L^{2}(\Omega))$.  Then we have 
\begin{equation*}
\Tilde{u}=\lim_{n\to\infty}u_{n}-\zeta\in L^{2}(0,T;W^{s,2}(\mathbb{R}^{n}))\cap C([0,T];L^{2}(\Omega)).
\end{equation*}
Since $\lim\limits_{n\to\infty}u_{n}=\Tilde{u}+\zeta\in L^{2}(0,T;W^{s,2}(\Omega'))\cap L^{2}(0,T;L^{1}_{2s}(\mathbb{R}^{n}))$ and $\Tilde{u}+\zeta=\zeta$ on $\left(\mathbb{R}^{n}\setminus\Omega\right)\times(0,T)$, we have
\begin{equation*}
u=\Tilde{u}+\zeta\in L^{2}(0,T;W^{s,2}(\Omega'))\cap L^{2}(0,T;L^{1}_{2s}(\mathbb{R}^{n}))\cap C([0,T];L^{2}(\Omega))
\end{equation*}   
is a weak solution to $\eqref{IVP}$. 
For the uniqueness, let $u$,$v$ be solutions to \eqref{IVP}.
Then $u-v\in L^{2}(0,T;W_{0}^{s,2}(\Omega))\cap C([0,T];L^{2}(\Omega))$ and it satisfies
\begin{equation*}
    \partial_{t}(u-v)+\mathcal{L}_{A}^{\Phi}u-\mathcal{L}_{A}^{\Phi}v=0 \quad\text{ in }\Omega_{T}.
\end{equation*}
Taking $u-v$ as a test function, we get 
\begin{align*}
    \sup_{t\in[0,T]}\int_{\Omega}(u-v)^{2}(x,t)\dx+\int_{0}^{T}[(u-v)(\cdot,t)]^{2}_{W^{s,2}(\mathbb{R}^{n})}\dt\leq 0,
\end{align*}
which implies $u$ and $v$ are same in $L^{2}(0,T;W_{0}^{s,2}(\Omega))\cap C([0,T];L^{2}(\Omega))$. In addition, if $\zeta\in L^{\infty}(0,T;L^{1}_{2s}(\mathbb{R}^{n}))$, we discover that
\begin{align*}
    \esup_{t\in(0,T)}\int_{\mathbb{R}^{n}}\frac{|w(y,t)|}{1+|y|^{n+2s}}\dy&\leq\esup_{t\in(0,T)}\int_{\Omega}\frac{|w(y,t)|}{1+|y|^{n+2s}}\dy+\esup_{t\in(0,T)}\int_{\mathbb{R}^{n}\setminus\Omega}\frac{|\zeta(y,t)|}{1+|y|^{n+2s}}\dy<\infty,
\end{align*}
which yields that
\[u\in L^{\infty}(0,T;L^{1}_{2s}(\mathbb{R}^{n})).\]
\end{proof}

\textbf{Data Availability} All data generated or analysed during this study are included in this published article.

\textbf{Conflict of Interest} There is no conflict of interest.

\textbf{Acknowledgement} We would like to thank a referee for the careful reading of the early version to give valuable comments and constructive suggestions.

\medskip

\providecommand{\bysame}{\leavevmode\hbox to3em{\hrulefill}\thinspace}
\providecommand{\MR}{\relax\ifhmode\unskip\space\fi MR }
\providecommand{\MRhref}[2]{%
  \href{http://www.ams.org/mathscinet-getitem?mr=#1}{#2}
}

\providecommand{\href}[2]{#2}

\end{document}